\DeclareSymbolFont{AMSb}{U}{msb}{m}{n}
\documentclass[reqno,centertags,11pt,a4paper,noamsfonts]{amsart}
\pdfoutput=1
\usepackage[english]{babel}
\usepackage[babel=true,expansion=alltext,protrusion=alltext-nott,final]{microtype}
\usepackage[margin={3cm}, marginparwidth={2.5cm}]{geometry} 
\usepackage[utf8]{inputenc}
 \usepackage{textcomp}
     \usepackage[sb]{libertine}
     \usepackage[varqu,varl]{inconsolata}
     \usepackage[libertine,bigdelims]{newtxmath}
     \useosf
     \usepackage[supstfm=libertinesups,supscaled=1.2,raised=-.13em]{superiors}
\usepackage{amsthm}
\usepackage{graphicx}
\usepackage{enumitem}
\usepackage[hyperref,svgnames]{xcolor}
\usepackage{bm}
\usepackage{tikz}
\usetikzlibrary{arrows}
\usepackage[unicode,pdfpagelabels,pageanchor=false,linktoc=all]{hyperref}
\usepackage{hyperxmp}
\hypersetup{%
  pdftitle={Variational approach to regularity of optimal transport maps: general cost functions},
  pdfauthor={F. Otto, M. Prod'homme, T. Ried},     
  pdfcopyright={Copyright (C) 2020 by the authors. Faithful reproduction of this article, in its entirety, by any means is permitted for non-commercial purposes.},
  pdfsubject={},
  pdfkeywords={optimal transportation, epsilon-regularity, partial regularity, general cost functions},
  pdfcaptionwriter={},
  pdfcontactaddress={},
  pdfcontactcity={},
  pdfcontactpostcode={},
  pdfcontactcountry={},
  pdfcontactphone={},
  pdfcontactemail={},
  pdfcontacturl={},
  pdflang={en},
  pdfmetalang={en},
  breaklinks=true,
  pdftoolbar=true,
  pdfmenubar=true,
  pdffitwindow=false,
  pdfstartview={FitH},
  pdfnewwindow=true,
  colorlinks=true,
  linkcolor=DarkGreen,
  citecolor=DarkOrange,
  urlcolor=DarkViolet
}
\providecommand{\mr}[1]{\href{http://www.ams.org/mathscinet-getitem?mr=#1}{MR~#1}}
\providecommand{\zbl}[1]{\href{https://zbmath.org/?q=an:#1}{Zbl~#1}}

\providecommand{\doi}[2]{\href{https://doi.org/#1}{#2}}

\newcommand{\RR}{\mathbb{R}}

\newcommand{\NN}{\mathbb{N}}

\newcommand{\BB}{\mathbb{B}}
\newcommand{\II}{\mathbb{I}}

\newcommand{\EE}{\mathcal{E}}
\newcommand{\DD}{\mathcal{D}}

\newcommand{\HH}{\mathcal{H}}

\newcommand{\dd}{\mathrm{d}}
\newcommand{\ee}{\mathrm{e}}

\newcommand{\normpi}{[\![\pi]\!]}

\DeclareMathOperator{\supp}{Spt}
\DeclareMathOperator{\id}{Id}

\DeclareMathOperator{\cexp}{c-exp}
\DeclareMathOperator{\csexp}{c^{\ast}-exp}
\DeclareMathOperator{\cbexp}{\overline{c}-exp}

\DeclareMathOperator{\ctexp}{\widetilde{c}-exp}
\DeclareMathOperator*{\esssup}{esssup}

\newcommand{\crosssymbol}{\mbox{\hskip.3em$|\kern-.ex| \kern-2.6ex=\joinrel=$}}
\DeclareMathOperator{\mathcross}{\crosssymbol}
\newcommand{\cross}[1]{\mathcross_{#1}}

\theoremstyle{plain}
\newtheorem{theorem}{Theorem}[section]
\newtheorem{proposition}[theorem]{Proposition}
\newtheorem{corollary}[theorem]{Corollary}
\newtheorem{lemma}[theorem]{Lemma}
\newtheorem*{theorem*}{Theorem}

\theoremstyle{definition}
\newtheorem{definition}[theorem]{Definition}
\newtheorem{remark}[theorem]{Remark}
\numberwithin{equation}{section}
\begin{document}
\title[Variational approach to regularity of optimal transport maps]{Variational approach to regularity of optimal transport maps: general cost functions}
\author[F.~Otto]{Felix Otto}
\address[F.~Otto]{Max-Planck-Institut für Mathematik in den Naturwissenschaften, Inselstraße 22, 04103 Leipzig, Germany}
\email{Felix.Otto@mis.mpg.de}

\author[M.~Prod'homme]{Maxime Prod'homme}
\address[M.~Prod'homme]{Institut de Mathématiques de Toulouse, Université Paul Sabatier, 118 route de Narbonne, F--31062 Toulouse Cedex 9, France}
\email{maxime.prodhomme@math.univ-toulouse.fr}

\author[T.~Ried]{Tobias Ried}
\address[T.~Ried]{Max-Planck-Institut für Mathematik in den Naturwissenschaften, Inselstraße 22, 04103 Leipzig, Germany}
\email{Tobias.Ried@mis.mpg.de}
\subjclass[2020]{49Q22, 35B65; 53C21}
\keywords{Optimal transportation, $\epsilon$-regularity, partial regularity, general cost functions, almost-minimality}
\thanks{\textcopyright {2021} by the authors. Faithful reproduction of this article, in its entirety, by any means is permitted for non-commercial purposes.}
\date{\today}
\begin{abstract}
    We extend the variational approach to regularity for optimal transport maps initiated by Goldman and the first author to the case of general cost functions. Our main result is an $\epsilon$--regularity result for optimal transport maps between Hölder continuous densities slightly more quantitative than the result by De Philippis--Figalli. One of the new contributions is the use of almost-minimality: if the cost is quantitatively close to the Euclidean cost function, a minimizer for the optimal transport problem with general cost is an almost-minimizer for the one with quadratic cost. This further highlights the connection between our variational approach and De Giorgi's strategy for $\epsilon$--regularity of minimal surfaces. 
\end{abstract}
\maketitle
\setcounter{tocdepth}{1}
\tableofcontents
\setcounter{tocdepth}{2}
\section{Introduction}\label{sec:introduction}

In this paper, we give an entirely variational proof of the $\epsilon$-regularity result for optimal transportation with a general cost function $c$ and H\"older continuous densities, as established by {De Philippis} and {Figalli} \cite{DF14}.
This provides a keystone to the line of research started in \cite{GO17} and continued in \cite{GHO19}: In \cite{GO17}, the variational approach was introduced and $\epsilon$-regularity was established in case of a Euclidean cost function $c(x,y) = \frac{1}{2} |x-y|^2$, see \cite[Theorem 1.2]{GO17}.
In \cite{GHO19}, among other things, the argument was extended to rougher densities, which required a substitute for {McCann}'s displacement convexity; this generalization is crucial here.

\medskip
One motivation for considering this more general setting is the study of optimal transportation on Riemannian manifolds with cost function given by $\frac{1}{2}d^2(x,y)$, where $d$ is the Riemannian distance. In this context, an $\epsilon$--regularity result is of particular interest because, compared to the Euclidean setting, even though $c$ is a compact perturbation of the Euclidean case, there are other mechanisms creating singularities like curvature. 
Indeed, under suitable convexity conditions on the support of the target density, the  so-called {Ma--Trudinger--Wang} (MTW) condition on the cost function $c$, a strong structural assumption, is needed to obtain global smoothness of the optimal transport map, see \cite{MTW05} and \cite{Loe09}. 
Since in the most interesting case of cost $\frac{1}{2}d^2(x,y)$, the MTW condition is quite restrictive\footnote{It is violated whenever a Riemannian sectional curvature is negative at any point of the manifold, see \cite{Loe09}.}, and does not have a simple interpretation in terms of geometric properties of the manifold\footnote{See \cite{KM10} for the concept of cross-curvature and its relation to the MTW condition.}, it is highly desirable to have a regularity theory without further conditions on the cost function $c$ (and on the geometry of the support of the densities).

\medskip
The outer loop of our argument is similar to that of \cite{DF14}: 
a Campanato iteration on dyadically shrinking balls that relies on a one-step improvement lemma, which in turn relies on the closeness of the solution to that of a simpler problem with a high-level interior regularity theory. The main differences are: 
\begin{enumerate}[label=$\circ$]
	\item In \cite{DF14}, the simpler problem is the Monge--Ampère equation with constant right-hand-side and Dirichlet boundary data coming from the convex potential; for us, the simpler problem is the Poisson equation with Neumann boundary data coming from the flux in the Eulerian formulation of optimal transportation \cite{BenamouBrenier}.
	\item In \cite{DF14}, the comparison relies on the maximum principle; in our case, it relies on the fact that the density/flux pair in the Eulerian formulation is a minimizer\footnote{In fact, an \emph{almost minimizer}.} given its own boundary conditions. 
	\item In \cite{DF14}, the interior regularity theory appeals to the $\epsilon$-regularity theory for the Monge--Ampère equation \cite{FK10}, which itself relies on {Caffarelli}'s work \cite{Caffarelli92}; in our case, it is just inner regularity of harmonic functions.
\end{enumerate}

\medskip
Loosely speaking, the Campanato iteration in \cite{DF14} relies on freezing the coefficients, whereas here, it relies on linearizing the problem (next to freezing the coefficients).
In the language of nonlinear elasticity, we tackle the geometric nonlinearity (which corresponds to the nonlinearity inherent to optimal transport) alongside the material nonlinearity (which corresponds to the cost function $c$).
As a consequence of this, we achieve $\mathcal{C}^{2,\alpha}$-regularity in a single Campanato iteration, whereas \cite{DF14} proceeds in three rounds of iterations, namely first $\mathcal{C}^{1,1-}$, then $\mathcal{C}^{1,1}$, and finally $\mathcal{C}^{2,\alpha}$.
Another consequence of this approach via linearization is that we instantly arrive at an estimate that has the same homogeneities as for a linear equation (meaning that the H\"older semi-norm of the second derivatives is estimated by the H\"older semi-norm of the densities and not a nonlinear function thereof).
Likewise, we obtain the natural dependence on the H\"older semi-norm of the mixed derivative of the cost function\footnote{The regularity of the cost function enters our result in a more nonlinear way, too: Its qualitative regularity on the $\mathcal{C}^2$-level determines the energy and length scales below which the linearization regime kicks in.}. When it comes to this dependence on the cost function $c$, we observe a similar phenomenon as for boundary regularity: regarding how the regularity of the data and the regularity of the solution are related, optimal transportation seems better behaved than its linearization, as we shall explain now\footnote{We refer to \cite{MiuraOtto} for a discussion of this phenomenon in the study of boundary regularity.}:
Assuming unit densities for the sake of the discussion, the Euler--Lagrange equation can be expressed on the level of the optimal transport map $T$ as the fully nonlinear (and $x$-dependent) elliptic system given by $\det\nabla T=1$ and $\mathrm{curl}_x(\nabla_x c(x,T(x)))=0$. Since the latter can be re-phrased by imposing that the matrix $\nabla_{xy}c(x,T(x))\nabla T(x)$ is symmetric, the H\"older norm of $\nabla T$ lives indeed on the same footing as the H\"older norm of the mixed derivative $\nabla_{xy}c$.
The linearization around $T(x)=x$ on the other hand is given by the elliptic system $\mathrm{div}\,\delta T=0$ and $\mathrm{curl}_x(\nabla_x c(x,x)+\nabla_{xy}c(x,x)\delta T(x))=0$, which has divergence-form character\footnote{In 2-d this can be seen by re-expressing the second equation in terms of the stream function of $\delta T$.}.
Here H\"older control of $\nabla_{xy}c(x,x)$ matches with H\"older control of only $\delta T$, and not its gradient.

\medskip
Our approach is analogous to {De Giorgi}'s strategy for the
$\epsilon$-regularity of minimal surfaces, foremost in the sense that it proceeds via harmonic approximation. 
In fact, our strategy is surprisingly similar to {Schoen} \& {Simon}'s variant \cite{SchoenSimon} to that regularity theory: 

\begin{enumerate}[label=$\circ$]
	\item Both approaches rely on the fact that the configuration is minimal given its own boundary conditions (Dirichlet boundary conditions there \cite[(43)]{SchoenSimon}, flux boundary conditions here \cite[(3.22)]{GO17}), see \cite[p.428]{SchoenSimon} and \cite[Proof of Proposition 3.3, Step 4]{GO17}; the Euler--Lagrange equation does not play a role in either approach.
	\item Both approaches have to cope with a mismatch in description between the given configuration and the harmonic construction (non-graph vs.~graph there, time-dependent flux vs.~time-independent flux here) which leads to an error term that luckily is superlinear in the energy, see \cite[(38)]{SchoenSimon} and \cite[Proof of Proposition 3.3, Step 4]{GO17}. On a very high-level description, this super-linearity may be considered as coming from the next term in a Taylor expansion of the nonlinearity, which in the right topology can be seen via lower-dimensional isoperimetric principles, see \cite[Lemma 3]{SchoenSimon} and \cite[Lemma 2.3]{GO17}. (However, there is no direct analogue of the Lipschitz approximation here.)
	\item Both approaches have to establish an approximate orthogonality that allows to relate the distance between the minimal configuration and the construction in an energy norm by the energy gap, see \cite[p.426ff.]{SchoenSimon} and \cite[Proof of Proposition 3.3, Step 3]{GO17} in the simple setting or rather \cite[Lemma 1.8]{GHO19} in our setting; it thus ultimately relies on some (strict) convexity, see \cite[(4)]{SchoenSimon}.
	\item In order to establish this approximate orthogonality, both approaches have to smooth out the boundary data (there by simple convolution, here in addition by a nonlinear approximation), see \cite[(34),(40),(52)]{SchoenSimon} and \cite[Proposition 3.6,(3.46),(3.47)]{GHO19}.
	\item In view of this, both approaches have to choose a good radius for the cylinder (in the Eulerian space-time here) on which the construction is carried out, see \cite[p.424]{SchoenSimon} and \cite[Section 3.1.3]{GHO19}.
\end{enumerate}

\medskip
The advantages of a variational approach become particularly apparent in this paper, when we pass from a Euclidean cost function to a more general one:
We may appeal to the concept of \emph{almost minimizers}, which is well-established for minimal surfaces.\footnote{See for instance \cite[Definition III.1]{Alm76}, \cite[Definition 1]{Bom82}. Almost minimizers are often also called $\omega$-minimizers, see e.g. \cite[Section 7.7]{Giu03}.} In our case, this simple concept means that, on a given scale, we interpret the minimizer (always with respect to its own boundary conditions) of the problem with $c$ as an approximate minimizer of the Euclidean problem. This allows us to directly appeal to the Euclidean harmonic approximation \cite[Theorem 1.5]{GHO19}. 
Incidentally, while dealing with H\"older continuous densities like in \cite{GO17} and not general measures as in \cite{GHO19}, we could not appeal to the simpler \cite[Proposition 3.3]{GO17}, since this one relies on the Euler--Lagrange equation in form of {McCann}'s displacement convexity.

\medskip
There are essentially two new challenges we face when passing from a Euclidean to a general cost function (next to the geometric ingredients also present in \cite{DF14}):
\begin{enumerate}[label=$\circ$]
	\item Starting point for the variational approach is always an $L^\infty/L^2$-bound on the displacement, which does rely on the Euler--Lagrange equation in the weak form of monotonicity of the support of the optimal coupling, see \cite[Lemma 3.1]{GO17}, \cite[Lemma 2.9]{GHO19}, and \cite[Proposition 2.2]{MiuraOtto}.
	In this paper, we establish the analogue of \cite[Lemma 3.1]{GO17} based on the $c$-monotonicity, see Proposition \ref{prop:Linfty-main}. Loosely speaking, this relies on  a qualitative argument down to some ($c$-dependent) scale $R_0$, followed by an argument that constitutes a perturbation of the one in \cite[Lemma 3.1]{GO17} for the scales $\le R_0$.
	\item We now address a somewhat hidden, but quite important additional difficulty that has to be overcome when passing from a Euclidean to a general cost functional\footnote{In fact, we missed this point in the first posted version.}:
	While the Kantorowich formulation of optimal transportation has the merit of being convex, it is so in a very degenerate way; 
	the Benamou--Brenier formulation on the contrary uncovers an almost strict\footnote{But not uniformly strict.} convexity. The variational approach to regularity capitalizes on this strict convexity. 
	However, this Eulerian reformulation seems naturally available only in the Riemannian case, and its strict convexity seems apparent only in the Euclidean setting. This is one of the reasons to appeal to the concept of almost minimizer, since it allows us to pass from a general cost function to the Euclidean one. 
	However, for configurations that are not exact minimizers of the Euclidean cost functional, the Lagrangian cost ${\int|y-x|^2\,\dd\pi}$ and the cost $\int\frac{1}{\rho}|j|^2$ of their Eulerian description \eqref{eq:density-flux-associated} are in general different:
	While the Eulerian cost is always dominated by the Lagrangian one, this is typically strict\footnote{\label{foot:Lagrangian-dominates-Eulerian}Here is an easy example: Let $\pi$ be generated by the map $T$ that maps $[-1,0]$ onto $[0,1]$ and $[0,1]$ onto $[-1,0]$, both by shift by one, to the right and to the left, respectively. The Lagrangian cost is obviously given by $2$. For the Eulerian cost we note that the Eulerian velocity $v$ vanishes on the diamond $|z|<\min\{t,1-t\}$, where the density assumes the value $2$. Due to this ``destructive interference'', the Eulerian cost is reduced by $1$.}.
	Hence the prior work \cite{GO17,GHO19,MiuraOtto} on the variational approach used the Euler--Lagrange equation in a somewhat hidden way, namely in terms of the incidence of Eulerian and Lagrangian cost. Luckily, the discrepancy of both functionals can be controlled for \emph{almost minimizers}, see Lemma \ref{prop:quasi-minimality-eulerian-main}.
\end{enumerate}

\subsection{Main results}
Let $X,Y \subset \RR^d$ be compact. We assume that the cost function $c: X\times Y \to \RR$ satisfies:
\begin{enumerate}[label=\textbf{(C\arabic*)}]
	\item\label{item:cost-cont} $c \in \mathcal{C}^{2}(X\times Y)$.
	\item\label{item:cost-inj-y} For any $x\in X$, the map $Y\ni y \mapsto -\nabla_x c(x,y) \in \RR^d$ is one-to-one.
	\item\label{item:cost-inj-x} For any $y\in Y$, the map $X\ni x \mapsto -\nabla_y c(x,y) \in \RR^d$ is one-to-one.
    \item\label{item:cost-non-deg} $\det \nabla_{xy}c(x,y)\neq 0$ for all $(x,y)\in X\times Y$.
\end{enumerate}
Let $\rho_0,\rho_1:\RR^d \to \RR$ be two probability densities, with $\supp\rho_0 \subseteq X$ and $\supp\rho_1 \subseteq Y$. It is well-known that under (an even milder regularity assumption than) condition \ref{item:cost-cont}, the optimal transportation problem
\begin{align}\label{eq:OT-c}
	\inf_{\pi\in\Pi(\rho_0,\rho_1)} \int_{\RR^d\times \RR^d} c(x,y)\,\dd\pi,
\end{align}
where the infimum is taken over all couplings $\pi$ between the measures $\rho_0\,\dd x$ and $\rho_1 \,\dd y$, admits a solution $\pi$, which we call a $c$-optimal coupling.

For $R>0$ we define the set 
\begin{align}\label{eq:cross-def}
   \cross{R} := (B_R \times \RR^d) \cup (\RR^d \times B_R),
\end{align}
which is quite natural in the context of optimal transportation, because it allows for a symmetric treatment of the transport problem: it is suitable to describe all the mass that gets transported out of $B_R$, and all the mass that is transported into $B_R$.
For $\alpha\in(0,1)$ we write 
\begin{align}\label{eq:c-hoelder-def}
	\left[ \nabla_{xy}c \right]_{\alpha, R} 
	:= \sup_{(x,y) \neq (x',y') \in \,\cross{R}} \frac{|\nabla_{xy}c(x,y) - \nabla_{xy}c(x',y')|}{|(x,y)-(x',y')|^{\alpha}}
\end{align}
for the $\mathcal{C}^{0,\alpha}$-semi-norm of the mixed derivative $\nabla_{xy}c$ of the cost function in the cross $\cross{R}$, and denote by 
\begin{align*}
	[\rho]_{\alpha,R} := \sup_{x\neq x'\in B_R} \frac{|\rho(x) - \rho(x')|}{|x-x'|^{\alpha}}
\end{align*}
the $\mathcal{C}^{0,\alpha}$-semi-norm of $\rho$ in $B_R$.

Fixing $\rho_0(0) =\rho_1(0) = 1$, we think of the densities as non-dimensional objects. This means that $\pi(\cross{R})=\int_{B_R}(\rho_0 + \rho_1)-\pi(B_R\times B_R)$ has units of $(\text{length})^d$, so that the \emph{Euclidean transport energy} $\int_{\cross{R}} \frac{1}{2}|x-y|^2\,\mathrm{d}\pi$ has dimensionality $(\text{length})^{d+2}$, and explains the normalization  by $R^{-(d+2)}$ in assumption \eqref{eq:smallness-main-theorem}  and in the definition \eqref{eq:def-energy} of $\EE_R$ below, making it a non-dimensional quantity\footnote{Notice that we are using a different convention here than in \cite{GHO19}, since it is more natural to work with the non-dimensional energy in our context.}. Similarly, the normalization $\nabla_{xy}c(0,0) = -\II$ makes the second derivatives of the cost function non-dimensional.

The main result of this paper is the following $\epsilon$-regularity result:
\begin{theorem}\label{thm:main}
Assume that \ref{item:cost-cont}--\ref{item:cost-non-deg} hold and that $\rho_0(0) = \rho_1(0) =1$, as well as $\nabla_{xy}c(0,0)= -\II$. Assume further that $0$ is in the interior of $\, X \times Y$.

Let $\pi$ be a $c$-optimal coupling from $\rho_0$ to $\rho_1$. 
There exists $R_0= R_0(c)>0$ such that for all $R\leq R_0$ with\footnote{An assumption of the form $f\ll1$ means that there exists $\epsilon>0$, typically only depending on the dimension $d$ and Hölder exponents, such that if $f\leq\epsilon$, then the conclusion holds. We write $\ll_{\Lambda}$ to indicate that $\epsilon$ also depends on the parameter $\Lambda$.
The symbols $\sim$, $\gtrsim$ and $\lesssim$ indicate estimates that hold up to a global constant $C$, which typically only depends on the dimension $d$ and Hölder exponents. For instance, $f\lesssim g$ means that there exists such a constant with $f\leq Cg$. $f\sim g$ means that $f\lesssim g$ and $g\lesssim f$.} 
    \begin{align}\label{eq:smallness-main-theorem}
        \frac{1}{R^{d+2}} \int_{B_{4R}\times \RR^d} |x-y|^2\,\dd \pi 
        + R^{2\alpha}\left([\rho_0]_{\alpha, 4R}^2 + [\rho_1]_{\alpha, 4R}^2 + \left[\nabla_{xy}c\right]_{\alpha, 4R}^2\right) \ll_c 1,
    \end{align}
    there exists a function $T\in \mathcal{C}^{1,\alpha}(B_R)$ such that $(B_R \times \RR^d) \cap \supp \pi  \subseteq \mathrm{graph}\,T$, 
    and the estimate 
    \begin{align}\label{eq:holder-gradient-main}
        [\nabla T]_{\alpha, R}^2 \lesssim \frac{1}{R^{d+2+2\alpha}}\int_{B_{4R}\times \RR^d} |x-y|^2\,\dd \pi + [\rho_0]_{\alpha, 4R}^2 + [\rho_1]_{\alpha, 4R}^2 + \left[\nabla_{xy}c\right]_{\alpha, 4R}^2
    \end{align}
    holds.
\end{theorem}

We stress that the implicit constant in \eqref{eq:holder-gradient-main} is independent of the cost $c$.
The scale $R_0$ below which our $\epsilon$-regularity result holds has to be such that $B_{2R_0} \subseteq X \cap Y$ and such that the qualitative $L^{\infty}/L^2$ bound (Lemma \ref{lem:displacement-qualitative}) holds.
We note that the dependence of $R_0$ on $c$ and the implicit dependence on $c$ in the smallness assumption \eqref{eq:smallness-main-theorem} are only through the qualitative information \ref{item:cost-cont}--\ref{item:cost-non-deg}, see Remark \ref{rem:qualitative-quantitative} and Lemma \ref{lem:displacement-qualitative} for details. Note also that, without appealing to the well-known result that the solution of \eqref{eq:OT-c} is a deterministic coupling $\pi=(\id \times T)_{\#} \rho_0$, this structural property of the optimal coupling is an outcome of our iteration.

\begin{remark}\label{rem:T-inverse}
Under the same assumptions as in Theorem \ref{thm:main}, in particular only asking for the one-sided energy $\frac{1}{R^{d+2}} \int_{B_{4R}\times \RR^d} |x-y|^2\,\dd \pi$ to be small in \eqref{eq:smallness-main-theorem}, we can also prove the existence of a function $T^* \in \mathcal{C}^{1,\alpha}(B_R)$ such that $(\RR^d\times B_R) \cap \supp\pi \subseteq \{(T^*(y),y) \, : \, y \in B_R\}$, with the same estimate on the semi-norm of $\nabla T^*$. This follows from the symmetric nature of the assumptions \ref{item:cost-cont}--\ref{item:cost-non-deg}, of the normalization conditions on the densities and the cost, and of the smallness assumption \eqref{eq:smallness-main-theorem}. We refer the reader to \ref{step:one-sided-to-two-sided} of the proof of Theorem \ref{thm:main} to see how \eqref{eq:smallness-main-theorem} entails smallness of a symmetric version of the Euclidean transport energy, as defined in \eqref{eq:def-energy}, at a smaller scale. 
\end{remark}

\medskip
As in \cite{DF14}, Theorem \ref{thm:main} leads to a partial regularity result for a $c$-optimal transport map $T$, that is, a map such that the $c$-optimal coupling between $\rho_0$ and $\rho_1$ is of the form 
\begin{align*}
    \pi_T := (\id \times T)_{\#} \rho_0.
\end{align*} The existence of such a map is a classical result in optimal transportation under assumptions \ref{item:cost-cont}--\ref{item:cost-inj-y} on the cost, as well as its particular structure, namely the fact that it derives from a potential.
More precisely, there exists a $c$-convex function
	$u: X \to \RR$ such that 
\begin{align*}
	T(x) = T_u(x) := \cexp_x(\nabla u(x)),
\end{align*}
where the \emph{$c$-exponential map} is well-defined in view of \ref{item:cost-cont} and \ref{item:cost-inj-y} via
\begin{align}\label{eq:c-exp}
	\cexp_x(p) = y \quad \Leftrightarrow \quad p = -\nabla_x c(x,y) \quad \text{for any } x\in X, y\in Y, p\in\RR^d.
\end{align}
Recall that a function $u: X \to \RR$ is $c$-convex if there exists a function $\lambda: Y\to \RR\cup\{-\infty\}$ such that \[u(x) = \sup_{y\in Y}\left( \lambda(y) - c(x,y)\right).\] 
Note that by assumption \ref{item:cost-cont} and the boundedness of $Y$, the function $u$ is semi-convex, i.e., there exists a constant $C$ such that $u+ C |x|^2$ is convex. Hence, by {Alexandrov}'s Theorem (see, for instance, \cite[Theorem 6.9]{EG15}, or \cite[Theorem 14.25]{Vil09}), $u$ is twice differentiable at a.e. $x\in X$. For more details on $c$-convexity and its connection to optimal transport and Monge-Ampère equations we refer to \cite[Chapter 5]{Vil09} and \cite[Section 5.3]{Fig17}.

Before stating the partial regularity result, let us mention that our $L^2$-based assumption on the smallness of the Euclidean energy of the forward transport is not more restrictive than the $L^{\infty}$-based assumption of closeness of the Kantorovich potential $u$ to $\frac{1}{2}|\cdot|^2$ in \cite[Theorems 4.3 \& 5.3]{DF14}. However, the assumption on $u$ is not invariant under transformations of $c$ and $u$ that preserve optimality, whereas the optimal transport map $T_u$, and hence our assumption on the energy $R^{-d-2} \int_{B_{4R}} |x-T_u(x)|^2\,\rho_0(x)\,\dd x$, are unaffected. For that reason we additionally have to fix $\nabla_{xx}c(0,0)=0$, and $\nabla_x c(0,0) = 0$ in the following corollary\footnote{Assuming that $T_u(0) = 0$, the assumption $\nabla_x c(0,0) = 0$ fixes $\nabla u(0) = -\nabla_{x}c(0, T_u(0)) = 0$.}, and ask for $[\nabla_{xx}c]_{\alpha, 4R}$ to be small. Hence, in this result we think of the cost $c$ as being close to $-x \cdot y$, which is not necessarily the case in Theorem \ref{thm:main}. 
\begin{corollary}\label{cor:DF}
    Assume that \ref{item:cost-cont}--\ref{item:cost-non-deg} hold and that $\rho_0(0) = \rho_1(0) =1$, as well as $\nabla_{xy}c(0,0)= -\II$, $\nabla_{xx}c(0,0)=0$, and $\nabla_x c(0,0) = 0$. Assume further that $0$ is in the interior of $\, X \times Y$.

Let $T_u$ be the $c$-optimal transport map from $\rho_0$ to $\rho_1$. 
There exists $R_0= R_0(c)>0$ such that for all $R\leq R_0$ with 
    \begin{align}\label{eq:smallness-main-corollary}
        \frac{1}{R^{2}} \left\|u-\tfrac{1}{2}|\cdot|^2 \right\|_{\mathcal{C}^0(B_{8R})}
        + R^{\alpha}\left([\rho_0]_{\alpha, 4R} + [\rho_1]_{\alpha, 4R} + \left[\nabla_{xy}c\right]_{\alpha, 4R} + \left[\nabla_{xx}c\right]_{\alpha, 4R}\right) \ll_c 1,
    \end{align}
    $T_u\in \mathcal{C}^{1,\alpha}(B_R)$ with
    \begin{align}\label{eq:holder-gradient-main-corollary}
        [\nabla T_u]_{\alpha, R} \lesssim\frac{1}{R^{2+\alpha}} \left\|u-\tfrac{1}{2}|\cdot|^2 \right\|_{\mathcal{C}^0(B_{8R})} + [\rho_0]_{\alpha, 4R} + [\rho_1]_{\alpha, 4R} + \left[\nabla_{xy}c\right]_{\alpha, 4R} + \left[\nabla_{xy}c\right]_{\alpha, 4R}.
    \end{align}
\end{corollary}

The partial regularity statement is then as follows:

\begin{corollary}\label{cor:partial}
Let $\rho_0,\rho_1:\RR^d \to \RR$ be two probability densities with the properties that $X= \supp\rho_0$ and $Y= \supp\rho_1$ are bounded with\footnote{$|\cdot|$ denotes the Lebesgue measure on $\RR^d$.} $|\partial X| =|\partial Y|= 0$, $\rho_0$, $\rho_1$ are positive on their supports, and $\rho_0\in \mathcal{C}^{0,\alpha}(X)$, $\rho_1\in \mathcal{C}^{0,\alpha}(Y)$.
Assume that $c \in \mathcal{C}^{2, \alpha}(X \times Y)$ and that \ref{item:cost-inj-y}--\ref{item:cost-non-deg} hold.
	Then there exist open sets $X' \subseteq X$ and $Y' \subseteq Y$ with $|X \setminus X'| = |Y \setminus Y'| = 0$ such that the $c$-optimal transport map $T$ between $\rho_0$ and $\rho_1$ is a $\mathcal{C}^{1,\alpha}$- diffeomorphism between $X'$ and $Y'$.
\end{corollary}

As recently pointed out in \cite{Gol20} for the quadratic case, the variational approach is flexible enough to also obtain $\epsilon$--regularity for optimal transport maps between merely continuous densities. The modifications presented in \cite{Gol20} can be combined with our results to prove an $\epsilon$--regularity result for the class of general cost functions considered above. This will be the context of a separate note \cite{PR20}.

\medskip
Finally, Theorem \ref{thm:main} can also be applied to optimal transportation on a Riemannian manifold $\mathcal{M}$ with cost given by the square of the Riemannian distance function $d$: if $\rho_0, \rho_1\in \mathcal{C}^{0,\alpha}(\mathcal{M})$ are two probability densities, locally bounded away from zero and infinity on $\mathcal{M}$, then the optimal transport map $T: \mathcal{M} \to \mathcal{M}$ sending $\rho_0$ to $\rho_1$ for the cost $c = \frac{d^2}{2}$ is a $\mathcal{C}^{1,\alpha}$-diffeomorphism outside two closed sets $\Sigma_X, \Sigma_Y \subset\mathcal{M}$ of measure zero. See \cite[Theorem 1.4]{DF14} for details.

\subsection{Strategy of the proofs}
In this section we sketch the proof of the $\epsilon$-regularity Theorem \ref{thm:main}. As in \cite{GO17,GHO19} one of the key steps is a harmonic approximation result, which can be obtained by an explicit construction and (approximate) orthogonality on an \emph{Eulerian} level. 

\subsubsection{$L^{\infty}$ bound on the displacement}
A crucial ingredient to the variational approach is a local $L^{\infty}/L^2$-estimate
on the level of the displacement. More precisely, given a scale $R$,
it gives a pointwise estimate on the non-dimensionalized displacement $\frac{y-x}{R}$
in terms of the (non-dimensionalized) Euclidean transport energy 
\begin{align}\label{eq:def-energy}
	\EE_R(\pi) := \frac{1}{R^{d+2}} \int_{\cross{R}} \frac{1}{2}|x-y|^2\,\mathrm{d}\pi,
\end{align}
which amounts to a
squared $L^2$-average of the displacement.
While this looks like an inner regularity estimate in the spirit
of the main result, Theorem \ref{thm:main}, it is not.
In fact, it is rather an interpolation estimate with the $c$-monotonicity of $\supp\pi$
providing an invisible second control next to the energy.
This becomes most apparent in the simple context of \cite[Lemma 3.1]{GO17} where
monotonicity morally amounts to a (one-sided) 
$L^{\infty}$-control of the gradient of the displacement.
The interpolation character of the estimate still shines through in the fractional exponent 
$\frac{2}{d+2}\in(0,1)$ on the $L^2$-norm. 

Following \cite{GHO19}, we here allow for general measures $\mu$ and $\nu$; the natural
local control of these data on the energy scale is given by
\begin{align}\label{eq:D-def}
	\DD_R(\mu,\nu) &:= \frac{1}{R^{d+2}} W^2_{B_{R}}(\mu, \kappa_{\mu}) + \frac{(\kappa_{\mu}-1)^2}{\kappa_{\mu}} +  \frac{1}{R^{d+2}} W^2_{B_{R}}(\nu, \kappa_{\nu}) + \frac{(\kappa_{\nu}-1)^2}{\kappa_{\nu}} ,
\end{align}
which measures locally at scale $R>0$ the distance from given measures $\mu$ and $\nu$ to the Lebesgue measure, where 
\begin{align}\label{eq:kappa-def}
    \kappa_{\mu} = \frac{\mu(B_R)}{|B_R|} \quad \text{and} \quad W_{B_R}^2(\mu,\kappa_\mu) = W^2(\mu\lfloor_{B_R},\kappa_{\mu}\,\dd x \lfloor_{B_R})
\end{align}
is the quadratic Wasserstein distance between $\mu\lfloor_{B_R}$ and $\kappa_{\mu}\,\dd x \lfloor_{B_R}$.\footnote{We use the convention that $W^2(\mu, \nu) = \inf_{\pi\in\Pi(\mu,\nu)} \int\frac{1}{2}|x-y|^2\,\dd\pi$ in this paper.} Notice that if $\mu = \rho_0\,\dd x$ and $\nu = \rho_1\,\dd y$ with Hölder continuous probability densities such that $\frac{1}{2} \leq \rho_j \leq 2$ on $B_R$, $j=0, 1$, then\footnote{Whenever there is no room for confusion, we will drop the dependence of $\DD$ on the measures $\mu$ and $\nu$.}
\begin{align}\label{eq:D-bound-rho}
	\DD_R \lesssim R^{2\alpha} \left([\rho_0]_{\alpha,R}^2 + [\rho_1]_{\alpha,R}^2\right),
\end{align}
see Lemma \ref{lem:D-Holder_rho} in the appendix.

The new aspect compared to \cite[Lemma 2.9]{GHO19} is the general cost function $c$.
Not surprisingly, it turns out that the result still holds provided $c$ is close to Euclidean and that the closeness is measured in the non-dimensional $\mathcal{C}^2$-norm.
We stress the fact that this closeness is not required on
the entire ``cross'' $\cross{5R}$, cf.~\eqref{eq:cross-def}, 
but only to the ``finite cross'' 
\begin{align}\label{eq:def-finite-cross}
    \BB_{5R, \Lambda R} := \left(B_{5R} \times B_{\Lambda R}\right) \cup \left( B_{\Lambda R} \times B_{5R}\right).
\end{align}
This is crucial, since only this smallness is guaranteed by the finiteness of the $\mathcal{C}^{2,\alpha}$-norm, cf.~\eqref{eq:c-mismatch} below. This sharpening is a consequence of the qualitative hypotheses \ref{item:cost-cont}--\ref{item:cost-non-deg}.

\begin{proposition}\label{prop:Linfty-main}
	Assume that the cost function $c$ satisfies \ref{item:cost-cont}--\ref{item:cost-non-deg}, and 
	let $\pi\in\Pi(\mu, \nu)$ be a coupling with $c$-monotone support. 
	
	For all $\Lambda < \infty$ and for all $R>0$ such that 
	\begin{align}
	    \cross{5R} \cap \supp \pi &\subseteq \BB_{5R, \Lambda R}\,\label{eq:Linfty-inclusion-assumption-main}
	\end{align} and for which
	\begin{align}
	    \EE_{6R} + \DD_{6R} &\ll 1, \label{eq:Linfty-smallness-assumption-main}\\
	    \text{and} \quad \|\nabla_{xy}c + \II \,\|_{\mathcal{C}^0(\BB_{5R, \Lambda R})} &\ll 1,\label{eq:Linfty-closeness-assumption-main} 
	\end{align}
	we have that 
	\begin{align}\label{eq:Linfty-main}
	    (x,y)\in \cross{4R} \cap \supp\pi \quad \Rightarrow \quad |x-y| \lesssim R \left( \EE_{6R} + \DD_{6R} \right)^{\frac{1}{d+2}}.
	\end{align}
\end{proposition}

\begin{remark}\label{rem:one-sided-energy}
    A close look at the proof of Proposition \ref{prop:Linfty-main} actually tells us that if $\EE_{6R}$ is replaced in \eqref{eq:Linfty-smallness-assumption-main} by the one-sided energy, that is, if we assume
    \begin{align*}
        \frac{1}{R^{d+2}} \int_{B_{6R} \times \RR^d} |y-x|^2 \, \dd \pi \ll 1,
    \end{align*}
    and if \eqref{eq:Linfty-inclusion-assumption-main} is replaced by the one-sided inclusion
    \begin{align*}
        (B_{5R} \times \RR^d) \cap \supp \pi \subseteq \BB_{5R, \Lambda R}, 
    \end{align*} then we still get a one-sided $L^{\infty}$ bound in the form of 
    \begin{align*}
        (x,y) \in (B_{4R} \times \RR^d) \cap \supp \pi \quad \Rightarrow \quad |x-y| \lesssim R\left(\frac{1}{R^{d+2}} \int_{B_{6R} \times \RR^d} |y-x|^2 \, \dd \pi + \DD_{6R} \right)^{\frac{1}{d+2}}.
    \end{align*}
    This observation will be useful in the proof of Theorem \ref{thm:main} to relate the one-sided energy in \eqref{eq:smallness-main-theorem} to the full energy in Proposition \ref{prop:one-step-main}. 
\end{remark}

Note that due to assumption \eqref{eq:Linfty-inclusion-assumption-main} Proposition \ref{prop:Linfty-main} might appear rather useless: indeed, one basically has to assume a (qualitative) $L^{\infty}$ bound in the sense that there is a constant $\Lambda<\infty$ such that if $x\in B_{5R}$, then $y\in B_{\Lambda R}$, in order to obtain the $L^{\infty}$ bound \eqref{eq:Linfty-main}. 
However, as we show in Lemma \ref{lem:displacement-qualitative}, due to the global assumptions \ref{item:cost-cont}--\ref{item:cost-non-deg} alone, there exists a scale $R_0>0$ and a constant $\Lambda_0<\infty$ such that \eqref{eq:Linfty-inclusion-assumption-main} holds. 
Moreover, in the Campanato iteration used to prove Theorem \ref{thm:main}, which is based on suitable affine changes of coordinates, the qualitative $L^{\infty}$ bound \eqref{eq:Linfty-inclusion-assumption-main} is reproduced in each step of the iteration (with a constant $\Lambda$ that after the first step can be fixed throughout the iteration, e.g. $\Lambda = 27$ works).

\begin{remark}\label{rem:qualitative-quantitative}
There is an apparent mismatch with respect to the domains involved in the closeness assumptions on $c$ in Theorem \ref{thm:main} and Proposition \ref{prop:Linfty-main}: we assume\footnote{The assumption of Theorem \ref{thm:main} is scaled to $6R$ here to match the scale on which smallness of $\EE$ and $\DD$ is assumed in both statements.} $R^{2\alpha}[\nabla_{xy}c]_{\alpha, 6R}^2 \ll 1$ in Theorem \ref{thm:main} and $\|\nabla_{xy}c+\II \|_{\mathcal{C}^0(\BB_{5R,\Lambda R})} \ll 1$ in Proposition \ref{prop:Linfty-main}. 
We are able to relate the two assumptions due to the qualitative $L^{\infty}$ bound \eqref{eq:Linfty-inclusion-assumption-main}:
If $\nabla_{xy}c(0,0)=-\II$, we have for all $\Lambda<\infty$, using the inclusion $\BB_{5R, \Lambda R} \subseteq \cross{6R}$,
	\begin{align}\label{eq:c-mismatch}
		\|\nabla_{xy}c + \II \|_{\mathcal{C}^0(\BB_{5R, \Lambda R})}
		&= \sup_{(x,y) \in \BB_{5R, \Lambda R}} |\nabla_{xy} c(x,y) - \nabla_{xy} c(0,0)|
		\lesssim_{\Lambda}   R^{\alpha} [\nabla_{xy}c]_{\alpha,6R}.
	\end{align}
	Thus, if $R^{2\alpha} [\nabla_{xy}c]_{\alpha,6R}^2$ is chosen small enough, then the assumption \eqref{eq:Linfty-closeness-assumption-main} in Proposition \ref{prop:Linfty-main} is fulfilled.
\end{remark}

\subsubsection{Almost-minimality with respect to Euclidean cost}
One of the main new contributions of this work is showing that the concept of almost-minimality, which is well-established in the theory of minimal surfaces, can lead to important insights also in optimal transportation. The key observation is that if $c$ is quantitatively close to Euclidean cost, then a minimizer of \eqref{eq:OT-c} is almost-minimizing for the quadratic cost. 

One difficulty in applying the concept of almost-minimality is that we are dealing with local quantities, for which local minimality (being minimizing with respect to its own boundary condition) would be the right framework to adopt.

\begin{lemma}\label{lem:optimal-restriction}
	Let $\pi \in \Pi(\mu,\nu)$ be a $c$-optimal coupling between the measures $\mu$ and $\nu$. Then for any Borel set $\Omega \subseteq \RR^d \times \RR^d$ the coupling $\pi_{\Omega}:= \pi \lfloor_{\Omega}$ is $c$-optimal given its own marginals, i.e.\ $c$-optimal between the measures $\mu_{\Omega}$ and $\nu_{\Omega}$ defined via
	\begin{align}\label{eq:mu-nu-Omega}
		\mu_{\Omega}(A) = \pi( (A\times \RR^d)\cap\Omega),
		\quad 
		\nu_{\Omega}(A) = \pi( (\RR^d\times A)\cap\Omega),
	\end{align}
	for any Borel measurable $A\subseteq \RR^d$.
\end{lemma}

This lemma allows us to restrict any $c$-optimal coupling $\pi$ to a ``good'' set, where particle trajectories are well-behaved in the sense that they satisfy an $L^{\infty}$ bound. In particular, we have the following corollary:

\begin{corollary}\label{cor:restriction-properties}
	Let $\pi \in \Pi(\mu,\nu)$ be a $c$-optimal coupling between the measures $\mu$ and $\nu$ with the property that there exists $M\leq 1$ such that for all $(x,y) \in \cross{R} \cap \supp\pi$
	\begin{align*}
		|x-y|\leq M R.
	\end{align*}
	Then the coupling $\pi_R := \pi\lfloor_{\cross{R}}$ is $c$-optimal between the measures $\mu_R$ and $\nu_R$ as defined in \eqref{eq:mu-nu-Omega} and we have that $\supp\mu_R, \supp\nu_R \subseteq B_{2R}$ (in particular $\supp\pi_R \subseteq B_{2R} \times B_{2R}$), $\mu_R = \mu$ and $\nu_R = \nu$ on $B_R$, and $\mu_R\leq\mu$, $\nu_R\leq \nu$.
\end{corollary}

One of the main observations now is that $c$-optimal couplings of the type considered in Corollary~\ref{cor:restriction-properties} are almost-minimizers of the Euclidean transport cost. The following assumptions \eqref{eq:supp-mu-nu-main} and \eqref{eq:mu-nu-close-main} should be read as properties satisfied by the marginal measures $\mu$ and $\nu$ of the restriction of a $c$-optimal coupling to a finite cross on which the $L^{\infty}$ bound \eqref{eq:Linfty-main} holds. 
Moreover, one of the marginals should be close to the Lebesgue measure in the sense that $\mu(B_R)\lesssim R^d$.

\begin{proposition}\label{prop:quasi-minimality-lagrangian-main}
Let $\mu$ and $\nu$ be two measures such that 
\begin{align}
	&\supp \mu \subseteq B_R, \, \supp \nu \subseteq B_R, \label{eq:supp-mu-nu-main} \\
	&\mu(B_R) \leq 2 |B_R|, 
	\label{eq:mu-nu-close-main}
\end{align}
for some $R>0$.
Let $\pi \in \Pi(\mu, \nu)$ be a $c$-optimal coupling between the measures $\mu$ and $\nu$. 
Then $\pi$ is almost-minimizing for the Euclidean cost, in the sense that for any $\widetilde{\pi}\in \Pi(\mu, \nu)$ we have that 
\begin{align}\label{eq:quasimin-main}
    \int \frac{1}{2} |x-y|^2\,\dd \pi 
    \leq \int \frac{1}{2} |x-y|^2\,\dd \widetilde{\pi} + R^{d+2} \Delta_{R},
\end{align}
where 
\begin{align} \label{eq:def-Delta-R-main}
	\Delta_R &:=  C \|\nabla_{xy} c + \II\|_{{C}^{0}(B_R\times B_R)} \, \EE_R(\pi)^{\frac{1}{2}}
\end{align}
for some constant $C$ depending only on $d$.
\end{proposition}

The above statement is most naturally formulated in terms of couplings, that is, in the Kantorovich framework. However, the way (almost-)minimality enters in the proof of the harmonic approximation result (see Theorem \ref{thm:harmonic} below), it is needed in the Eulerian picture, where the construction of a competitor is done. 

\subsubsection{The Eulerian side of optimal transportation}
Given a coupling $\pi\in\Pi(\mu, \nu)$ between measures $\mu$ and $\nu$, we can define its Eulerian description, i.e.\ the density-flux pair $(\rho_t,j_t)$ associated to the coupling $\pi$ by 
\begin{align}\label{eq:density-flux-associated}
	\int \zeta \,\mathrm{d}\rho_t := \int \zeta((1-t)x + ty) \,\mathrm{d}\pi, \quad 
	\int \xi \cdot\mathrm{d}j_t := \int \xi((1-t)x + ty)\cdot(y-x) \,\mathrm{d}\pi 
\end{align}
for $t\in [0,1]$ and for all test functions $\zeta \in \mathcal{C}_c^{0}(\RR^d \times [0,1])$ and fields $\xi \in \mathcal{C}_c^{0}(\RR^d \times [0,1])^d$. It is easy to check that $(\rho_t,j_t)$ is a distributional solution of the continuity equation 
\begin{align}\label{eq:CE-intro}
	\partial_t \rho_t + \nabla\cdot j_t = 0, \qquad \rho_0 = \mu,\quad \rho_1 = \nu,
\end{align}
that is, for any $\zeta \in \mathcal{C}^1_c(\RR^d \times [0,1])$ there holds
\begin{align}\label{eq:CE-weak-intro}
	\int_0^1 \left(\int \partial_t \zeta \,\mathrm{d}\rho_t + \nabla\zeta \cdotp \mathrm{d}j_t \right)\,\mathrm{d}t = \int \zeta_1 \,\mathrm{d}\nu - \int \zeta_0\,\mathrm{d}\mu.
\end{align}
For brevity, we will often write $(\rho,j) := (\rho_t \,\mathrm{d}t, j_t \,\mathrm{d}t)$. Being divergence-free in $(t,x)$, the density-flux pair $(\rho, j)$ admits internal (and external) traces on $\partial(B_R \times(0,1))$ for any $R>0$, see \cite{CF03} for details, i.e., there exists a measure $f_R$ on $\partial B_R \times (0,1)$ such that 
\begin{align}\label{eq:CE-R}
	\int_{B_R\times[0,1]} \left( \partial_t \zeta \,\mathrm{d}\rho + \nabla\zeta \cdot \mathrm{d}j \right) = \int_{B_R} \zeta_1\,\mathrm{d}\nu - \int_{B_R} \zeta_0\,\mathrm{d}\mu + \int_{\partial B_R\times[0,1]} \zeta \,\mathrm{d}f_R,
\end{align}
We also introduce the time-averaged measure $\overline{f}_R$ on $\partial B_R$ defined via
\begin{align}\label{eq:fbar}
	\int_{\partial B_R} \zeta \,\dd \overline{f}_R 
	:= \int_{\partial B_R\times[0,1]} \zeta \,\dd f_R.
\end{align}
Similarly, defining the measure $\overline{j} := \int_0^1\,\dd j(\cdot, t)$, it is easy to see that 
\begin{align*}
	\nabla \cdot \overline{j} = \mu - \nu 
\end{align*} 
and that therefore $\overline{j}$ admits internal and external traces, which agree for all $R>0$ with $|\overline{j}|(\partial B_R) = \mu(\partial B_R) = \nu(\partial B_R) = 0$, and the internal trace agrees with $\overline{f}_R$.

Note that we have the duality \cite[Proposition 5.18]{San15} 
\begin{align}\label{eq:BB-duality}
	\frac{1}{2}\int \frac{1}{\rho} |j|^2 = \sup_{\xi\in \mathcal{C}^{0}_{c}(\RR^d \times [0,1])^d} \int \xi\cdot\dd j - \frac{|\xi|^2}{2}\,\dd \rho,
\end{align}
which immediately implies the subadditivity of $(\rho,j) \mapsto \int \frac{1}{\rho} |j|^2$. 
A localized version of \eqref{eq:BB-duality}, in the form of  
\begin{align}\label{eq:BB-duality-local}
	\frac{1}{2}\int_{B\times[0,1]} \frac{1}{\rho} |j|^2 = \sup_{\xi\in \mathcal{C}^{0}_{c}(B \times [0,1])^d} \int \xi\cdot\dd j - \frac{|\xi|^2}{2}\,\dd \rho,
\end{align}
also holds for any open set $B\subseteq \RR^d$.
From the inequality $\xi \cdot (y-x) -\frac{1}{2} |\xi|^2 \leq \frac{1}{2} |x-y|^2$, which is true for any $\xi, x, y \in \RR^d$, the duality formula \eqref{eq:BB-duality} immediately implies that the Eulerian cost of the density-flux pair $(\rho, j)$ corresponding to a coupling $\pi$ via \eqref{eq:density-flux-associated} is always dominated by the Lagrangian cost of $\pi$, i.e.\
\begin{align}\label{eq:Lagrangian-dominates-Eulerian}
	\frac{1}{2}\int \frac{1}{\rho} |j|^2 
	\leq \frac{1}{2} \int |x-y|^2 \,\dd \pi.
\end{align}
We stress that this inequality is in general strict, see the example in Footnote \ref{foot:Lagrangian-dominates-Eulerian}. 

Contrary to the case of quadratic cost $c(x,y) = \frac{1}{2}|x-y|^2$, or, equivalently, $\widetilde{c}(x,y) = -x \cdot y$, given an optimal coupling $\pi$ for the cost $c$, the density-flux pair $(\rho,j)$ associated to $\pi$ in the sense of \eqref{eq:density-flux-associated} is not optimal for the {Benamou--Brenier} formulation \cite{BenamouBrenier} of optimal transportation, i.e., 
\begin{align}\label{eq:BB}
	W^2(\mu, \nu) = \inf\left\{ \frac{1}{2} \int \frac{1}{\rho} |j|^2: \partial_t \rho + \nabla\cdot j = 0, \rho_0 = \mu, \rho_1 = \nu\right\},
\end{align}
where the continuity equation and boundary conditions are understood in the weak sense \eqref{eq:CE-weak-intro}, see \cite[Chapter 8]{Vil03} for details.

\begin{remark}\label{rem:BB}
	For minimizers $\pi$ of the Euclidean transport cost, i.e.\ $\pi\in\mathrm{argmin}\, W^2(\mu,\nu)$, we have \emph{equality} in \eqref{eq:Lagrangian-dominates-Eulerian},
\begin{align}\label{eq:Lagrangian-equal-Eulerian}
	\frac{1}{2}\int \frac{1}{\rho} |j|^2 
	= \frac{1}{2}  \int |x-y|^2 \,\dd \pi.
\end{align}
Indeed, if $(\rho,j)$ is the Eulerian description of $\pi$, then by \eqref{eq:BB}
\begin{align*}
    \frac{1}{2} \int |x-y|^2 \,\dd \pi = W^2(\mu, \nu) \leq \frac{1}{2} \int\frac{1}{\rho}|j|^2,
\end{align*}
which together with \eqref{eq:Lagrangian-dominates-Eulerian} implies \eqref{eq:Lagrangian-equal-Eulerian}.
\end{remark}

As another consequence, while displacement convexity guarantees in the Euclidean case that the Eulerian density $\rho \leq 1$ (up to a small error), c.f. \cite[Lemma 4.2]{GO17}, in our case $\rho$ is in general merely a measure. This complication is already present in \cite{GHO19} and led to important new insights in dealing with marginals that are not absolutely continuous with respect to Lebesgue measure in the Euclidean case, upon which we are also building in this work.

The Eulerian version of the almost-minimality Proposition~\ref{prop:quasi-minimality-lagrangian-main} can then be obtained via the following lemma:

\begin{lemma}\label{prop:quasi-minimality-eulerian-main}
Let $\pi \in \Pi(\mu, \nu)$ be a coupling between the measures $\mu$ and $\nu$ with the property that there exists a constant $\Delta <\infty$ such that 
\begin{align}\label{eq:lagrangian-almost-min-assumption}
	 \int \frac{1}{2} |x-y|^2\,\dd \pi 
    \leq \int \frac{1}{2} |x-y|^2\,\dd \widetilde{\pi} + \Delta
\end{align}
for any $\widetilde{\pi}\in\Pi(\mu, \nu)$,
and let $(\rho,j)$ be its Eulerian description defined in \eqref{eq:density-flux-associated}. 
Then 
\begin{align*}
	 \int \frac{1}{2} |x-y|^2\,\dd \pi \leq \frac{1}{2} \int \frac{1}{\rho} |j|^2 + \Delta, 
\end{align*}
and
\begin{align*}
	\frac{1}{2} \int \frac{1}{\rho} |j|^2 \leq \frac{1}{2} \int \frac{1}{\widetilde{\rho}} |\widetilde{j}|^2 + \Delta
\end{align*}
for any pair of measures $(\widetilde\rho,\widetilde j)$ satisfying
	\begin{align}\label{eq:CE-tilde-main}
	    \int \left( \partial_t \zeta \,\mathrm{d}\widetilde{\rho} + \nabla\zeta \cdot \mathrm{d}\widetilde{j} \, \right) = \int \zeta_1\,\mathrm{d}\nu - \int \zeta_0\,\mathrm{d}\mu
	\end{align}
for all $\zeta\in \mathcal{C}^\infty_c(\mathbb{R}^d\times[0,1])$.
\end{lemma}

\subsubsection{The harmonic approximation result}
The main ingredient in the proof of Theorem \ref{thm:main} is the harmonic approximation result, which states that if a coupling between two measures supported on a ball (say of radius $7R$ for some $R>0$) satisfies the  $L^{\infty}$ bound Proposition \ref{prop:Linfty-main} globally on its support and is almost-minimizing with respect to the Euclidean cost, then the displacement $y-x$ is quantitatively close to a harmonic gradient field $\nabla\Phi$ in $\cross{R}$. This is actually a combination of a harmonic approximation result in the Eulerian picture (Theorem~\ref{thm:harmonic}) and Lemma~\ref{lem:harmonic-lagrange}, which allows us to transfer the Eulerian information back to the Lagrangian framework.

\begin{theorem}[Harmonic approximation]\label{thm:harmonic}
Let $R>0$ and $\mu,\nu$ be two measures with the property that 
\begin{align}\label{eq:harmonic-assumption-mu-nu}
	\supp\mu \subseteq B_{7R}, \quad \supp\nu \subseteq B_{7R}, \quad \text{and}\quad \mu(B_{7R})\leq 2|B_{7R}|.
\end{align}
Let further $\pi\in\Pi(\mu,\nu)$ be a coupling between the measures $\mu$ and $\nu$, such that: 
\begin{enumerate}
	\item $\pi$ satisfies a global $L^{\infty}$-bound, that is, there exists a constant $M\leq 1$ such that
		\begin{align}\label{eq:Linfty-global}
			|x-y| \leq M R
			\quad \text{for any } (x,y) \in \supp\pi.
		\end{align}	
	\item If $(\rho,j)$ is the Eulerian description of $\pi$ as defined in \eqref{eq:density-flux-associated}, then there exists a constant $\Delta_R<\infty$ such that
	\begin{align}\label{eq:almostminimal-harmonic}
		\int \frac{1}{\rho} |j|^2 \leq \int \frac{1}{\widetilde{\rho}} |\widetilde{j}|^2 + R^{d+2} \Delta_R 
	\end{align}
	for any Eulerian competitor, i.e.\ any pair of measures $(\widetilde\rho,\widetilde j)$ satisfying 
	\begin{equation*}
	    \int \partial_t \zeta \,\dd \widetilde{\rho} + \nabla\zeta\cdot \dd \widetilde{j} = \int_{B_R} \zeta_1 \,\dd\nu - \int_{B_R} \zeta_0 \,\dd \mu + \int_{\partial B_R \times [0,1]} \zeta\,\dd f_R.
	\end{equation*}
\end{enumerate}
Then for every $0<\tau \ll 1$, there exist $\epsilon_{\tau}>0$ and $C,C_{\tau}<\infty$ such that, provided\footnote{Note that by assumption \eqref{eq:harmonic-assumption-mu-nu} the coupling $\pi$ is supported on $B_{7R}\times B_{7R}$, so that by means of the estimate $\EE_{6R}(\pi) \lesssim \EE_{7R}(\pi) = \int |x-y|^2\,\dd\pi$ the smallness assumption of the Euclidean energy of $\pi$ on scale $6R$ could be replaced by a smallness assumption on the global Euclidean energy of $\pi$. However, since $\DD$ behaves nicely under restriction only on average, the applicability of the harmonic approximation result derived in \cite{GHO19} becomes more apparent in the form of assumption \eqref{eq:harmonic-smallness}.}
	\begin{align}\label{eq:harmonic-smallness}
		\EE_{6R}(\pi) + \DD_{6R}(\mu, \nu) \leq \epsilon_{\tau},
	\end{align}
the following holds: 
	There exists a radius $R_* \in (3R, 4R)$ such that if $\Phi$ is the solution, unique up to an additive constant, of \footnote{We recall that since the boundary flux $\overline{f}_{R_*}$ (as defined in \eqref{eq:fbar}) is a measure, equation \eqref{eq:harmonic-approx} has to be understood in the distributional sense, and that $\kappa_{\mu} = \frac{\mu(B_{R_*})}{|B_{R_*}|}$ and $\kappa_{\nu} = \frac{\nu(B_{R_*})}{|B_{R_*}|}$.}
	\begin{align}\label{eq:harmonic-approx}
		\Delta \Phi = \kappa_{\mu} - \kappa_{\nu} \quad \text{in } B_{R_*} 
		\quad \text{and} \quad 
		\nu\cdot \nabla\Phi = \nu\cdot \overline{j} = \overline{f}_{R_*} \quad \text{on } \partial B_{R_*},
	\end{align}
	then\footnote{We refer to \eqref{eq:def-j-rho-phi} for how to  understand the left-hand side of \eqref{eq:harmonic-energy}.} 
	\begin{align}\label{eq:harmonic-energy}
		\frac{1}{R^{d+2}} \int_{B_{2R}\times[0,1]} \frac{1}{\rho} |j-\rho\nabla\Phi|^2 
		\leq \left(\tau + \frac{C M}{\tau}\right) \EE_{6R} + C_{\tau}\DD_{6R} + \Delta_R,
	\end{align}
	and 
	\begin{align}\label{eq:harmonic-estimate}
		\frac{1}{R^2} \sup_{B_{2R}} |\nabla\Phi|^2 + 
		 \sup_{B_{2R}} |\nabla^2\Phi|^2 +
		R^2 \sup_{B_{2R}} |\nabla^3\Phi|^2 
		\lesssim \EE_{6R} + \DD_{6R}.
	\end{align}
\end{theorem}

From the Eulerian version of the harmonic approximation Theorem \ref{thm:harmonic} we can also obtain a Lagrangian version via almost-minimality:
\begin{lemma}\label{lem:harmonic-lagrange}
	Let $R>0$ and let $\pi\in\Pi(\mu,\nu)$ be a coupling between the measures $\mu$ and $\nu$, such that 
\begin{enumerate}
	\item $\pi$ satisfies a global $L^{\infty}$-bound, that is, there exists a constant $M\leq 1$ such that
		\begin{align}\label{eq:Linfty-global-harmonic-lagrange}
			|x-y| \leq M R
			\quad \text{for any } (x,y) \in \supp\pi;
		\end{align}	
	\item if $(\rho,j)$ is the Eulerian description of $\pi$ as defined in \eqref{eq:density-flux-associated}, then there exists a constant $\Delta_R<\infty$ such that
	\begin{align}\label{eq:almost-min-harmonic-lagrange}
		\int |x-y|^2\,\dd\pi \leq \int \frac{1}{\rho} |j|^2 + R^{d+2} \Delta_R.
	\end{align}
\end{enumerate}
Then for any smooth function $\Phi$ there holds
	\begin{align}\label{eq:harmonic-lagrange-lemma}
	\frac{1}{R^{d+2}}\int_{\cross{R}} \int_0^1 |x-y + \nabla\Phi(tx + (1-t)y)|^2 \,\dd t\,\dd \pi 
	\leq \frac{1}{R^{d+2}}\int_{B_{2R} \times [0,1]} \frac{1}{\rho} |j-\rho\nabla\Phi|^2 +  \Delta_R.
\end{align}
\end{lemma}

\subsubsection{One-step improvement and Campanato iteration}
With the harmonic approximation result at hand, we can derive a one-step improvement result, 
which roughly says that if the coupling $\pi$ is quantitatively close to $(\id\times\id)_{\#}\rho_0$ on some scale $R$, expressed in terms of the estimate
\begin{align*}
    \mathcal{E}_R(\pi) + R^{2\alpha}\left([\rho_0]_{\alpha, R}^2 + [\rho_1]_{\alpha, R}^2 + \left[\nabla_{xy}c\right]_{\alpha, R}^2\right) \ll 1,
\end{align*}
and the fact that the (qualitative) $L^{\infty}$ bound on the displacement \eqref{eq:Linfty-inclusion-assumption-main} holds, 
then on a smaller scale $\theta R$, after an affine change of coordinates, it is even closer to $(\id\times\id)_{\#}\rho_0$.
This is the basis of a Campanato iteration to obtain the existence of the optimal transport map $T$ and its $\mathcal{C}^{1,\alpha}$ regularity.

We start with the affine change of coordinates and its properties:
\begin{lemma}\label{lem:coordinate-change}
Let $\pi\in\Pi(\mu,\nu)$ be an optimal transport plan with respect to the cost function $c$ between the measures $\mu(\dd x) = \rho_0(x)\,\dd x$ and $\nu(\dd y) = \rho_1(y)\,\dd y$. 

Given a non-singular matrix $B \in\RR^{d\times d}$ and a vector $b\in\RR^d$, we perform the affine change of coordinates\footnote{We use the notation $B^{-*} = (B^*)^{-1}$, where $B^*$ is the transpose of $B$, and $\nabla^2 c(x, y) = \begin{pmatrix}
	 \nabla_{xx}c(x,y) & \nabla_{yx}c(x,y) \\ \nabla_{xy} c(x,y) & \nabla_{yy}c(x,y)
\end{pmatrix}$.}
\begin{align}\label{eq:Q-transformation}
	\begin{pmatrix}
		\widehat{x} \\
		\widehat{y}
	\end{pmatrix}
	= 
	\begin{pmatrix}
		Bx \\ 
		\gamma B^{-*} D^* (y-b)
	\end{pmatrix}
	=: Q(x,y),
\end{align}
where $D = -\nabla_{xy}c(0,b)$ and $\gamma = \left(\frac{\rho_1(b)}{\rho_0(0)} \frac{|\det B|^2}{|\det D|}\right)^{\frac{1}{d}}$.\footnote{Note that $D$ is non-singular by assumption \ref{item:cost-non-deg}.}
If we let
\begin{align*}
	\widehat{\rho}_0(\widehat{x}) &= \frac{\rho_0(x)}{\rho_0(0)}, \quad 
	\widehat{\rho}_1(\widehat{y}) = \frac{\rho_1(y)}{\rho_1(b)}, \quad  \widehat{c}(\widehat{x},\widehat{y}) = \gamma c(x,y),
\end{align*}
so that in particular $\widehat{\rho}_0(0) = \widehat{\rho}_1(0) = 1$ and 
\begin{equation}\label{eq:mixed-derivative}
    \nabla_{\widehat{x}\widehat{y}} \widehat{c}(\widehat{x},\widehat{y}) = BD^{-1}\nabla_{xy}c(x,y)B^{-1},
\end{equation}
from which it follows that $\nabla_{\widehat{x}\widehat{y}}\widehat{c}(0,0) = -\II$, then the coupling
\begin{align}\label{eq:pihat-def}
	\widehat{\pi} := \frac{|\det B|}{\rho_0(0)} Q_{\#}\pi
\end{align}
is an optimal coupling between the measures $\widehat{\mu}(\dd\widehat{x}) = \widehat{\rho}_0(\widehat{x})\,\dd\widehat{x}$ and $\widehat{\nu}(\dd\widehat{y}) = \widehat{\rho}_1(\widehat{y})\,\dd\widehat{y}$ with respect to the cost function $\widehat{c}$. 
\end{lemma}
In the change of variables we perform, the role of $D$ is to ensure that we get a normalized cost, i.e.\ $\nabla_{\widehat{x}\widehat{y}}\widehat{c}(0,0) = -\II$, while $\gamma$ and $\det B$ in \eqref{eq:pihat-def} are needed for $\widehat{\pi}$ to define a transportation plan between the new densities. 
We refer the reader to Appendix \ref{app:proofcoord} for a proof of this lemma.

\begin{proposition}\label{prop:one-step-main}
Assume that $\rho_0(0) = \rho_1(0) = 1$ and $\nabla_{xy}c(0,0) = -\II$, and let $\pi \in \Pi(\rho_0, \rho_1)$ be $c$-optimal. 

Then for all $\beta \in (0, 1)$, there exist $\theta \in (0,1)$ and $C_{\beta} < \infty$ such that for all $\Lambda<\infty$ and $R>0$ for which 
\begin{align}
	&\cross{8R}\cap \supp\pi \subseteq \BB_{8R,\Lambda R}, \label{eq:qualitative-Linfty-onestep} \\
    &\mathcal{E}_{9R}(\pi) + R^{2\alpha}\left([\rho_0]_{\alpha, 9R}^2 + [\rho_1]_{\alpha, 9R}^2 + \left[\nabla_{xy}c\right]_{\alpha, 9R}^2\right) \ll_{\Lambda} 1, \label{eq:smallness-onestep-main}
\end{align}  
there exist a symmetric matrix $B\in\RR^{d\times d}$ and a vector $b \in \RR^d$ with 
\begin{align}\label{eq:B-b-est-main}
    |B-\II|^2 + \frac{1}{R^2}|b|^2 \lesssim \mathcal{E}_{9R}(\pi) + R^{2\alpha}\left([\rho_0]_{\alpha, 9R}^2 + [\rho_1]_{\alpha, 9R}^2\right), 
\end{align}
such that, performing the change of variables in Lemma \ref{lem:coordinate-change}, the coupling $\widehat{\pi}$ is $\widehat{c}$-optimal between the measures with densities $\widehat{\rho_0}$ and $\widehat{\rho_1}$ and there holds 
\begin{align}\label{eq:one-step-conclusion-main}
   \mathcal{E}_{\theta R}(\widehat{\pi}) \leq \theta^{2\beta} \mathcal{E}_{9R}(\pi) + C_{\beta}R^{2\alpha}\left([\rho_0]_{\alpha, 9R}^2 + [\rho_1]_{\alpha, 9R}^2 + \left[\nabla_{xy}c\right]_{\alpha, 9R}^2\right).
\end{align}
Moreover, 
we have the inclusion 
\begin{align}\label{eq:qualitative-theta-inclusion-main}
    \cross{\frac{8}{9}\theta R} \cap \supp \widehat{\pi} \subseteq \BB_{\frac{8}{9}\theta R, 3\theta R}.
\end{align}
\end{proposition}

Let us give a rough sketch of how the one-step improvement result can now be iterated: 
In a first step, the qualitative bound on the displacement is obtained from the global assumptions \ref{item:cost-cont}--\ref{item:cost-non-deg} on the cost function, see Lemma \ref{lem:displacement-qualitative}. This yields an initial scale $R_0>0$ below which the cost function is close enough to the Euclidean cost function for \eqref{eq:Linfty-inclusion-assumption-main} to hold. 
We may therefore apply Proposition \ref{prop:one-step-main}, so that after an affine change of coordinates the the energy inequality \eqref{eq:one-step-conclusion-main} holds, the transformed densities and cost function are again normalized at the origin, optimality is preserved, and the qualitative $L^{\infty}$ bound \eqref{eq:qualitative-theta-inclusion-main} holds for the new coupling. We can therefore apply the one-step improvement Proposition \ref{prop:one-step-main} again, going to smaller and smaller scales. 
Together with Campanato's characterization of Hölder regularity, this yields the claimed existence and $\mathcal{C}^{1,\alpha}$ regularity of $T$.

The details of the above parts of the proof of our main Theorem \ref{thm:main} are explained in the sections below, with a full proof of Theorem \ref{thm:main} in Section \ref{sec:proof-main}. The proof of Corollary \ref{cor:partial} is essentially a combination of the ideas in \cite{GO17} and \cite{DF14}, and is given for the convenience of the reader in Section \ref{sec:partial}.

We conclude the introduction with a comment on the extension of the results presented above to general almost-minimizers with respect to Euclidean cost in the following sense: 

\begin{definition}[Almost-minimality w.r.t.\ Euclidean cost (on all scales)]\label{def:almost-min-general}
A coupling $\pi\in\Pi(\mu,\nu)$ is almost-minimal with respect to Euclidean cost if there exists $R_0>0$ and $\Delta_{\cdot}: (0,R_0] \to [0,\infty)$ non-decreasing such that for all $r\leq R_0$ and $(x_0,y_0)$ in the interior of $X \times Y$ there holds 
\begin{align}\label{eq:almost-min-general}
    \int |y-x|^2\,\mathrm{d}\pi \leq \int |y-x|^2\,\mathrm{d}\widetilde{\pi} + r^{d+2}\Delta_r
\end{align}
for all $\widetilde{\pi} \in \Pi(\mu,\nu)$ such that $\supp(\pi-\widetilde{\pi}) \subseteq (B_r(x_0)\times\RR^d) \cup (\RR^d \times B_r(y_0))$.
\end{definition}

We will restrict our attention to almost-minimizers in the class of deterministic transport plans coming from a Monge map $T$, i.e.\ $\pi=\pi_T = (\id, T)_{\#} \mu$, and call a transport map almost-minimizing with respect to Euclidean cost if for all $r\leq R_0$ and $x_0 \in \mathrm{int}X$ there holds
\begin{align}\label{eq:almost-min-monge}
    \int |T(x) - x|^2 \,\mu(\mathrm{d}x) \leq \int |\widetilde{T}(x) - x|^2\,\mu(\mathrm{d}x) + r^{d+2} \Delta_r
\end{align}
for all $\widetilde{T}$ such that $\widetilde{T}_{\#}\mu=\nu$ and $\mathrm{graph} \widetilde{T} = \mathrm{graph} T$ outside $(B_r(x_0)\times \RR^d) \cup (\RR^d \times B_r(T(x_0)))$.

In this situation, we get the following generalization\footnote{A more quantitative version of this result is work in progress.} of Theorem~\ref{thm:main}, whose proof will be sketched in Section~\ref{sec:almost-min-general}.
\begin{theorem}\label{thm:almost-min}
    Assume that $\rho_0(0) = \rho_1(0) = 1$, and that $0$ is in the interior of $X\times Y$. 
    Let $T:X\to Y$ be an almost-minimizing transport map from $\mu$ to $\nu$ with rate function $\Delta_r = C r^{2\alpha}$ for some $C<\infty$. Assume further that $T$ is invertible. There exists $R_1>0$ such that for any $R\leq R_1$ with
    \begin{align*}
        \frac{1}{R^{d+2}} \left( \int_{B_{4R}} |T(x) - x|^2\,\mu(\mathrm{d}x) + \int_{B_{4R}} |T^{-1}(y) - y|^2\,\nu(\mathrm{d}y) \right) + R^{2\alpha} \left([\rho_0]_{\alpha, 4R}^2 + [\rho_1]_{\alpha,4R}^2 \right) \ll 1,
    \end{align*}
    there holds $T\in \mathcal{C}^{1,\alpha}(B_R)$.
\end{theorem}
\bigskip
\begin{center}
\textsc{Acknowledgements}
\end{center}
\medskip

We thank Jonas Hirsch for pointing out and explaining the reference \cite{SchoenSimon} and other works related to minimal surfaces.
We also thank Georgiana Chatzigeorgiou for carefully reading a previous version of this article and valuable comments, which led to the identification of a more serious error that has been fixed in the current version. We also thank the referee for very helpful comments, in particular for suggesting to include the more general statement on almost-minimizing transport maps.

MP thanks the Max Planck Institute for Mathematics in the Sciences and TR thanks the University of Toulouse for their kind hospitality. 
MP was partially supported by the Projects MESA (ANR-18-CE40-006) and EFI (ANR-17-CE40-0030) of the French National Research Agency (ANR).

\section{An \texorpdfstring{$L^{\infty}$}{L infinity} bound on the displacement}\label{sec:Linfty}
In this section we establish an $L^{\infty}$ bound on the displacement for transference plans $\pi\in\Pi(\mu, \nu)$ with $c$-monotone support, that is,
\begin{align}\label{eq:c-monotonicity}
	c(x, y) + c(x',y') \leq c(x, y') + c(x',y) \quad \text{for all} \quad (x,y), (x',y') \in \supp \pi,
\end{align}
provided that the transport cost is small, the marginals $\mu, \nu$ are close to the Lebesgue measure, and the cost function $c$ is close to the Euclidean cost function.  
In Lemma \ref{lem:displacement-qualitative} we use the $c$-monotonicity \eqref{eq:c-monotonicity}
combined with the qualitative hypotheses \ref{item:cost-cont}--\ref{item:cost-non-deg} in conjunction with compactness to obtain a more qualitative version of the $L^\infty/L^2$-bound, which just expresses finite expansion.
In Proposition \ref{prop:Linfty-main} this qualitative $L^{\infty}/L^2$ bound in form of \eqref{eq:Linfty-inclusion-assumption-main} is upgraded to the desired quantitative version under the scale-invariant smallness assumption \eqref{eq:Linfty-closeness-assumption-main}. 
The latter is a consequence of the quantitative smallness hypothesis $R^{2\alpha} [\nabla_{xy}c]_{\alpha,R}^2 \ll 1$, as we pointed out in Remark \ref{rem:qualitative-quantitative}. 
In both steps, we need to ensure that there are sufficiently many points in $\supp\pi$ close to the diagonal. This is formulated in Lemma \ref{lem:cone}, which does not rely on monotonicity.

\begin{proof}[Proof of Proposition \ref{prop:Linfty-main}]
    Let $\Lambda < \infty$ and  $R>0$ be such that \eqref{eq:Linfty-inclusion-assumption-main}, \eqref{eq:Linfty-smallness-assumption-main} and \eqref{eq:Linfty-closeness-assumption-main} hold. We only prove the bound \eqref{eq:Linfty-main} for a couple $(x,y) \in (B_{4R} \times \RR^d) \cap \supp\pi$, as the other case $(x,y) \in (\RR^d \times B_{4R}) \cap \supp\pi$ follows by symmetry. 
        
	\begin{enumerate}[label=\textsc{Step \arabic*},leftmargin=0pt,labelsep=*,itemindent=*]
		\item (Rescaling). Let $(\widetilde{x}, \widetilde{y}) = (S_R(x), S_R(y)) := (R^{-1}x, R^{-1}y)$ and set 
		\begin{align*}
		    \widetilde{\mu} &:= {S_R}_{\#} \mu, \quad \widetilde{\nu} := {S_R}_{\#} \nu, \quad \widetilde{\pi} := R^{-d} (S_R \times S_R)_{\#} \pi \quad \text{and} \quad \widetilde{c}(\widetilde{x}, \widetilde{y}) := R^{-2}c(R\widetilde{x}, R\widetilde{y}), 
		\end{align*}
		so that $\widetilde{c}$ still satisfies properties \ref{item:cost-cont}--\ref{item:cost-non-deg}, and we have  $\widetilde{\pi} \in \Pi(\widetilde{\mu}, \widetilde{\nu})$, and $\supp\widetilde{\pi}$ is $\widetilde{c}$-monotone. We also have 
		\begin{align*}
		    \EE_6(\widetilde{\pi}) = \EE_{6R}(\pi), \quad \DD_6(\widetilde{\mu}, \widetilde{\nu}) = \DD_{6R}(\mu, \nu) \quad \text{and} \quad \|\nabla_{\widetilde{x} \widetilde{y}}\widetilde{c} + \II\|_{\mathcal{C}^0(\BB_{5, \Lambda})} = \|\nabla_{xy}c+\II\|_{\mathcal{C}^0(\BB_{5R, \Lambda R})}. 
		\end{align*} This allows us to only consider the  case $R=1$ in the following. We will abbreviate 
		\begin{align*}
		    \EE := \EE_6 \quad \text{and} \quad \DD := \DD_6. 
		\end{align*}
		
		\medskip
		\item (Use of $c$-monotonicity of $\supp \pi$).
    		Let $(x,y) \in (B_{4} \times \RR^d) \cap \supp\pi$. We first show that for all $(x',y') \in (B_{5} \times \RR^d) \cap \supp\pi$ we have 
    		\begin{align}\label{eq:Linfty-step1}
				(x-y) \cdot (x-x') 
				&\leq 3 |x-x'|^2 + |x'-y'|^2 + \delta |x-x'| \, |x-y|,
			\end{align}
			where, recalling \eqref{eq:def-finite-cross},  
			\begin{align}\label{eq:delta-def}
				\delta:= \|\nabla_{xy}c + \II \|_{\mathcal{C}^0(\BB_{5, \Lambda})}.  
			\end{align}
			
			Indeed, setting $x_t := tx + (1-t)x'$ and $y_s := s y + (1-s) y'$ for $t,s \in [0,1]$, $c$-monotonicity \eqref{eq:c-monotonicity} of $\supp\pi$
		    implies that 
		    \begin{align}
				0 &\geq (c(x,y)-c(x',y))-(c(x,y')-c(x',y'))
				= \int_0^1 \int_0^1 (x-x') \cdot \nabla_{xy} c(x_t, y_s) \, (y-y') \,\dd s \dd t  \nonumber\\
		         &= - (x-x')\cdot (y-y') +  \int_0^1 \int_0^1 (x-x') \cdot \left(\nabla_{xy} c(x_t, y_s) + \II \right) \, (y-y') \,\dd s \dd t. \label{eq:Linfty-step1-cmon}
		    \end{align}
			By assumption \eqref{eq:Linfty-inclusion-assumption-main}, we have $(x,y)$, $(x',y') \in B_{5} \times B_{\Lambda}$ and thus $\{(x_t, y_t)\}_{t\in[0,1]} \subseteq B_{5} \times B_{\Lambda} \subset \BB_{5, \Lambda}$. Hence, we obtain from \eqref{eq:delta-def}  
			\begin{align*}
				\left| \int_0^1 \int_0^1 (x-x') \cdot \left(\nabla_{xy} c(x_t, y_s) + \II \right) \, (y-y') \,\dd s \dd t \right| \leq \delta |x-x'| \, |y-y'|,
			\end{align*}
			so that \eqref{eq:Linfty-step1-cmon} turns into
			\begin{align}\label{eq:Linfty-step1-cmon-2}
				0 \geq - (x-x') \cdot (y-y') - \delta |x-x'| \, |y-y'|.
			\end{align}
			Upon writing $y-y' = (y-x) + (x-x') + (x'-y')$, it follows that
			\begin{align*}
				0 &\geq - (x-x') \cdot (y-x) - |x-x'|^2 - (x-x') \cdot (x'-y') - \delta |x-x'| \, |y-x| \\
				&\qquad - \delta |x-x'|^2 -\delta |x-x'| \, |x'-y'|,
			\end{align*}
			from which we obtain the estimate  
			\begin{align*}
				(x-y) \cdot (x-x') 
				&\leq \delta |x-x'| \, |x-y| + (1+\delta) |x-x'|^2 + (1+\delta) |x-x'| \, |x'-y'| \\
				&\leq \delta |x-x'| \, |x-y| +  \frac{3}{2}(1+\delta) |x-x'|^2 + \frac{1}{2} (1+\delta) |x'-y'|^2.
			\end{align*}
			Note that $\delta \ll 1$ by assumption \eqref{eq:Linfty-closeness-assumption-main}, hence \eqref{eq:Linfty-step1} follows.
		\bigskip
		\item (Proof of estimate \eqref{eq:Linfty-main}). 
	Let $r \ll 1$ and $e\in S^{d-1}$ be arbitrary, to be fixed later. Let $\eta$ be supported in $B_{\frac{r}{2}}(x-re)$ and satisfy the bounds 
	\begin{align}\label{eq:Linfty-assumptions-eta}
		\sup |\eta| + r \sup |\nabla\eta| + r^2 \sup |\nabla^2 \eta| \lesssim 1.
	\end{align}
	We make the additional assumption that $\eta$ is normalized in such a way that
	\begin{align}\label{eq:normalisation-eta}
		\int (x-x') \eta(x') \,\dd x' = r^{d+1} e.
	\end{align}
	Note that since $x\in B_{4}$ and $r\ll 1$, we have 
	\begin{align*}
	    \supp \eta \subseteq B_{\frac{r}{2}}(x-re) \subset B_{5}. 
	\end{align*}

	Integrating inequality \eqref{eq:Linfty-step1} against the measure $\eta(x') \, \pi(\dd x' \dd y')$, 
	it follows that
	\begin{align} \label{eq:eta-chi}
	\begin{split}
		(x-y) \cdot \int (x-x') \eta(x')\,\pi(\dd x' \dd y') 
		&\leq 3 \int |x-x'|^2 \eta(x')\,\pi(\dd x' \dd y')\\
		&\quad + \int |x'-y'|^2 \eta(x')\,\pi(\dd x' \dd y')\\
		&\quad + \delta |x-y| \int |x-x'| \eta(x')\,\pi(\dd x' \dd y').
	\end{split}
	\end{align}
	Note that by \eqref{eq:normalisation-eta} the integral on the left-hand-side of inequality \eqref{eq:eta-chi} can be expressed as  
	\begin{align*}
		\int (x-x') \eta(x')\, \pi(\dd x' \dd y')
		&= \int (x-x') \eta(x') \mu(\dd x') \\
		&= \kappa_{\mu} r^{d+1} e 
		+ \int (x-x') \eta(x') \left( \mu(\dd x') - \kappa_{\mu}\,\dd x'\right). 
	\end{align*}
	To estimate the latter integral, we recall the following result from \cite[Lemma 2.8]{GHO19}: for any $\zeta \in \mathcal{C}^{\infty}(B_R)$, 
	\begin{align}\label{eq:GHO19-cutoff-bound}
		\left| \int_{B_{R}} \zeta \,(\dd\mu -\kappa_{\mu}\,\dd x) \right| 
		&\leq \left( \kappa_{\mu} \int_{B_{R}} |\nabla\zeta|^2\,\dd x \, W_{B_{R}}^2(\mu,\kappa_{\mu}) \right)^{\frac{1}{2}} + \frac{1}{2} \sup_{B_{R}} |\nabla^2 \zeta| \, W_{B_{R}}^2(\mu,\kappa_{\mu}).
	\end{align}
	By this estimate with $\zeta = (x-\cdot)\eta$ and using that $\kappa_{\mu}\sim 1$ by assumption  \eqref{eq:Linfty-smallness-assumption-main}, we obtain with \eqref{eq:Linfty-assumptions-eta} that
	\begin{align*}
		\left| \int (x-x') \eta(x') \left(\mu(\dd x') - \kappa_{\mu}\,\dd x'\right) \right|
		&\lesssim (r^d \DD)^{\frac{1}{2}} + \frac{1}{r} \DD 
		\lesssim \epsilon r^{d+1} + \frac{1}{\epsilon} \frac{1}{r}  \DD
	\end{align*}
	for some $0<\epsilon\ll 1$ to be fixed later. Hence,
	\begin{align}\label{eq:Linfty-step2-lhs}
	\begin{split}
		(x-y) &\cdot \int (x-x') \eta(x')\,\pi(\dd x' \dd y') \\
		&\geq \kappa_{\mu} r^{d+1} (x-y)\cdot e - C \left(\epsilon r^{d+1} + \frac{1}{\epsilon r} \DD\right)|x-y|.
	\end{split}
	\end{align}
	We now estimate each term on the right-hand-side of inequality \eqref{eq:eta-chi} separately: 
	\begin{enumerate}[label=\arabic*),leftmargin=0pt,labelsep=*,itemindent=*]
		\item For the first term we estimate 
		\begin{align*}
			&\int |x-x'|^2 \eta(x')\,\pi(\dd x' \dd y') 
			= \int |x-x'|^2 \eta(x') \mu(\dd x') \\
			&\quad \leq \kappa_{\mu} \int |x-x'|^2 \eta(x') \,\dd x'
			+ \left| \int |x-x'|^2 \eta(x') \left(\mu(\dd x') - \kappa_{\mu}\,\dd x'\right) \right|.
		\end{align*}
		Using again \eqref{eq:Linfty-assumptions-eta} and $\kappa_{\mu} \sim 1$ for the first term on the right-hand side, estimate \eqref{eq:GHO19-cutoff-bound} with $\zeta=|x-\cdot|^2 \eta$, and Young's inequality for the second term we obtain 
		\begin{align}\label{eq:Linfty-step2-rhs1}
			\int |x-x'|^2 \eta(x')\,\pi(\dd x' \dd y') 
			\lesssim r^{d+2} + \DD.
		\end{align}
			
		\item For the second term on the right-hand-side of \eqref{eq:eta-chi} we use that $\supp\eta\subseteq B_{5}$ and \eqref{eq:Linfty-assumptions-eta}, recalling also the definition \eqref{eq:def-energy} of $\EE$, to estimate 
		\begin{align}\label{eq:Linfty-step2-rhs2}
			\int |x'-y'|^2 \eta(x')\,\pi(\dd x' \dd y')
			\lesssim \int_{B_{5}\times \RR^d} |x'-y'|^2 \,\pi(\dd x' \dd y')
			\lesssim \EE.
		\end{align}	
		
		\item We may bound the integral in the third term on the right-hand-side of \eqref{eq:eta-chi} as for \eqref{eq:Linfty-step2-rhs1} by using \eqref{eq:Linfty-assumptions-eta}, $\kappa_{\mu} \sim 1$ and estimate \eqref{eq:GHO19-cutoff-bound} with\footnote{Notice that by the assumption that $\eta$ is supported on $B_{\frac{r}{2}}(x-re)$, the function $x' \mapsto |x-x'| \eta(x')$ has no singularity at $x'=x$.} $\zeta=|x-\cdot|\eta$ to get 
		\begin{align}\label{eq:Linfty-step2-rhs3}
			\int |x-x'| \eta(x')\,\pi(\dd x' \dd y')
			\lesssim r^{d+1} + \frac{1}{r} \DD.
		\end{align}
		\end{enumerate}

		Inserting the estimates \eqref{eq:Linfty-step2-lhs}, \eqref{eq:Linfty-step2-rhs1}, \eqref{eq:Linfty-step2-rhs2}, and \eqref{eq:Linfty-step2-rhs3} into inequality \eqref{eq:eta-chi} yields 
			\begin{align*}
				&\kappa_{\mu} r^{d+1} (x-y) \cdot e \lesssim \left(\epsilon r^{d+1} + \frac{1}{\epsilon r} \DD\right)|x-y| 
				+ r^{d+2} + \DD 
				+ \EE
				+ \delta \left(r^{d+1} + \frac{1}{r} \DD \right)|x-y|.
			\end{align*}
		Since $e$ is arbitrary and $\kappa_{\mu} \sim 1$, this turns into
			\begin{align}\label{eq:Linfty-step2-penultimate}
				|x-y| 
				&\lesssim (\epsilon + \delta) |x-y| 
				+ \left(\delta+\frac{1}{\epsilon}\right) \left(\frac{1}{r}\right)^{d+2}\DD \,|x-y| + r + \left(\frac{1}{r}\right)^{d+1} (\EE + \DD).
			\end{align}
		We first choose $\epsilon$ and the implicit constant in \eqref{eq:Linfty-closeness-assumption-main},
        which in view of \eqref{eq:delta-def} governs $\delta$, so small that we may absorb the first
        term on the right-hand-side into the left-hand-side. 
        We then choose $r$ to be a large multiple of 
        $(\EE+\DD)^\frac{1}{d+2}$, so that also the second right-hand-side
        term in \eqref{eq:Linfty-step2-penultimate} can be absorbed. This choice of $r$ is admissible in the sense
        of $r\ll 1$ provided the implicit constant in \eqref{eq:Linfty-smallness-assumption-main} is small enough.
        This yields \eqref{eq:Linfty-main}.
		\qedhere
	\end{enumerate}
\end{proof}

The next lemma shows that due to the global qualitative information on the cost function $c$, that is, \ref{item:cost-cont}--\ref{item:cost-non-deg}, there is a scale below which we can derive a qualitative bound on the displacement. It roughly says that there is a small enough scale after which the cost essentially behaves like Euclidean cost, with an error that is uniformly small due to compactness of the set $X\times Y$.

\begin{lemma}\label{lem:displacement-qualitative}
	Assume that the cost function $c$ satisfies  \ref{item:cost-cont}--\ref{item:cost-non-deg} and 
	let $\pi\in\Pi(\mu, \nu)$ be a coupling with $c$-monotone support. 
	
	There exist $\Lambda_0 <\infty$ and $R_0>0$ such that for all $R\leq R_0$ for which 
	\begin{align}\label{eq:displacement-qualitative-assumption}
		\EE_{6R} + \DD_{6R} \ll 1,
	\end{align}
	we have the inclusion 
	\begin{align*}
		\cross{5R} \cap \supp \pi \subseteq  \BB_{5R, \Lambda_0 R}.
	\end{align*}
\end{lemma}

\begin{proof} We only prove the inclusion $(B_{5R}\times \RR^d) \cap \supp \pi \subseteq B_{5R}\times B_{\Lambda_0 R}$, the other inclusion $(\RR^d \times B_{5R}) \cap \supp \pi \subseteq B_{\Lambda_0 R} \times B_{5R}$ follows analogously since the assumptions are symmetric in $x$ and $y$.
\begin{enumerate}[label=\textsc{Step \arabic*},leftmargin=0pt,labelsep=*,itemindent=*]
	\item \label{step:grad-boundedness} (Use of $c$-monotonicity of $\supp\pi$).
	Let $R>0$ be such that \eqref{eq:displacement-qualitative-assumption} holds, in the sense that we may use Lemma \ref{lem:cone}\ref{item:cone-3}, and set 
	\begin{align}\label{eq:ctilde-def}
		\widetilde{c}(x,y) := c(x,y) - c(x,0) - c(0,y) + c(0,0). 
	\end{align}
	We claim that there exists a constant $\lambda<\infty$, depending only on $\|c\|_{\mathcal{C}^2(X\times Y)}$, such that 
	\begin{align*}
		\nabla_x \widetilde{c}(x,y) \in B_{\lambda R} \quad \text{for all} \quad (x,y) \in (B_{5R} \times \RR^d) \cap \supp \pi. 
	\end{align*}
	
	To show this, we use the $c$-monotonicity \eqref{eq:c-monotonicity} of $\supp \pi$.
	Notice that $c$-monotonicity of $\supp\pi$ implies its $\widetilde{c}$-monotonicity:  
	\begin{align}\label{eq:tilde-monotonicity}
		\widetilde{c}(x,y) - \widetilde{c}(x',y) \leq \widetilde{c}(x,y') - \widetilde{c}(x',y') 
		\quad \text{for all} \quad (x,y),(x',y') \in \supp \pi.
	\end{align}
	With $x_t := tx + (1-t)x'$ we can write
	\begin{align*}
		\widetilde{c}(x,y) - \widetilde{c}(x',y) 
		&= \int_0^1 \nabla_x \widetilde{c}(x_t, y)\,\dd t \cdot (x-x') \\
		&= \nabla_x \widetilde{c}(x, y) \cdot (x-x') + \int_0^1 (\nabla_x \widetilde{c}(x_t, y) - \nabla_x \widetilde{c}(x, y)) \,\dd t \cdot (x-x'),
	\end{align*}
	and, using that $\nabla_x\widetilde{c}(0,0) = 0$,
	\begin{align*}
		\widetilde{c}(x,y') - \widetilde{c}(x',y') 
		&= (\nabla_x \widetilde{c}(0, y') - \nabla_x \widetilde{c}(0,0)) \cdot (x-x') \\
		&\quad + \int_0^1 (\nabla_x \widetilde{c}(x_t, y') - \nabla_x \widetilde{c}(0, y')) \,\dd t \cdot (x-x').
	\end{align*} 
	Inserting these two identities into inequality \eqref{eq:tilde-monotonicity} gives
	\begin{align*}
		\nabla_x \widetilde{c}(x, y) \cdot (x-x') 
		&\leq \int_0^1 |\nabla_x \widetilde{c}(x_t, y) - \nabla_x \widetilde{c}(x, y)| \,\dd t \, |x-x'|  + |\nabla_x \widetilde{c}(0, y') - \nabla_x \widetilde{c}(0,0)| \, |x-x'| \\
		&\quad +\int_0^1 |\nabla_x \widetilde{c}(x_t, y') - \nabla_x \widetilde{c}(0, y')| \,\dd t \, |x-x'|.
	\end{align*}
	Using the boundedness of $\|c\|_{\mathcal{C}^2(X\times Y)}$ we  estimate this expression further by 
	\begin{align}
		\nabla_x \widetilde{c}(x, y) \cdot (x-x') 
		&\leq \|\nabla_{xx} \widetilde{c} \|_{\mathcal{C}^0(X\times Y)} \left( \int_0^1 |x_t - x|\,\dd t + \int_0^1 |x_t|\,\dd t \right) |x-x'| \nonumber\\
		&\quad + \|\nabla_{xy} \widetilde{c} \|_{\mathcal{C}^0(X\times Y)} |y'| \, |x-x'| \nonumber\\
		&\lesssim \|c\|_{\mathcal{C}^2(X\times Y)} \left( |x'| + |x| + |y'| \right) |x-x'|, \label{eq:nablax-estimate}
	\end{align}
	where in the last step we estimated the integrals and used that
	\begin{align*}
		\|\nabla_{xx} \widetilde{c} \|_{\mathcal{C}^0(X\times Y)} 
		\leq 2 \| \nabla_{xx} c\|_{\mathcal{C}^0(X\times Y)}, 
		\quad
		\|\nabla_{xy} \widetilde{c} \|_{\mathcal{C}^0(X\times Y)} 
		= \|\nabla_{xy} c\|_{\mathcal{C}^0(X\times Y)} .
	\end{align*}
	Now by Lemma \ref{lem:cone}\ref{item:cone-3}, given $x \in B_{5R}$, we have $(S_{R}(x,e) \times B_{7R}) \cap \supp\pi \neq \emptyset$ for any direction $e\in S^{d-1}$. Hence, letting $e = \frac{\nabla_x \widetilde{c}(x,y)}{|\nabla_x \widetilde{c}(x,y)|}$, we can find a point $(x',y') \in (S_{R}(x,e) \times B_{7R})\cap \supp\pi$. Since the opening angle of $S_R(x,e)$ is $\frac{\pi}{2}$, we have 
	\begin{align*}
		\nabla_x \widetilde{c}(x,y) \cdot (x-x') 
		= |\nabla_{x}\widetilde{c}(x,y)| |x-x'| e \cdot \frac{x-x'}{|x-x'|}
		\gtrsim |\nabla_{x}\widetilde{c}(x,y)| |x-x'|.
	\end{align*}
	It follows with \eqref{eq:nablax-estimate} that there exists $\lambda<\infty$ such that 
	\begin{align*}
		|\nabla_x \widetilde{c}(x,y)| \lesssim \|c\|_{\mathcal{C}^2(X\times Y)} \left(|x'| + |x| + |y'| \right) \leq \lambda R.
	\end{align*}
		
	\medskip
	\item \label{step:implicitfunction} (Use of global information on $c$). 
	We claim that there exist $R_0>0$ and $\Lambda_0<\infty$ such that for all $R\leq R_0$ 
	and $x\in X$, we have that 
	\begin{align*}
		B_{\lambda R} \subseteq -\nabla_x \widetilde{c}(x, B_{\Lambda_0 R}).
	\end{align*}
	
	Indeed, by assumption \ref{item:cost-inj-y}, for any $x\in X$, the map $-\nabla_x \widetilde{c}(x,\cdot) = -\nabla_x c(x,\cdot) + \nabla_x c(x, 0)$ is one-to-one on $Y$.
	Since $\nabla_{xy}\widetilde{c}(x,y) = \nabla_{xy}c(x,y)$, it follows by \ref{item:cost-non-deg} that $\det\nabla_{xy}\widetilde{c}(x,y) \neq 0$ for all $(x,y) \in X\times Y$. Hence, the map 
	\begin{align*}
		F_x : -\nabla_x \widetilde{c}(x,Y) \to Y, \quad v \mapsto \left[ - \nabla_x \widetilde{c}(x,\cdot) \right]^{-1}(v)
	\end{align*}
	is well-defined and a $\mathcal{C}^1$-diffeomorphism, so that in particular
	\begin{align*}
		F_x(v) = F_x(0) + DF_x(0)v + o_x(|v|).
	\end{align*}
	Using that $-\nabla_x \widetilde{c}(x,0) = 0$, which translates into $F_x(0) = 0$, this takes the form
	\begin{align*}
		F_x(v) = DF_x(0)v + o_x(|v|).
	\end{align*}
	By compactness of $X$, there thus exist a radius $R_0>0$ and a constant $\Lambda_0<\infty$ such that
	\begin{align*}
		\lambda |F_x(v)| \leq \Lambda_0 |v| 
		\quad \text{for all} \quad 
		x\in X \quad \text{and} \quad 
		|v| \leq \lambda R_0,
	\end{align*}
	which we may reformulate as
	\begin{align*}
		F_x(B_{\lambda R}) \subseteq B_{\Lambda_0 R}, \quad \text{i.e.} \quad B_{\lambda R} \subseteq F_x^{-1}(B_{\Lambda_0 R}) = -\nabla_x \widetilde{c}(x,B_{\Lambda_0 R})
	\end{align*}
	for all $R\leq R_0$ and $x\in X$.
	
	\medskip
	\item \label{step:qualitative-bound} (Conclusion). 
	If $(x,y) \in (B_{5R} \times \RR^d) \cap \supp\pi$, then we claim that $|y| \leq \Lambda_0 R$ for $R\leq R_0$.
	
	Indeed, by \ref{step:grad-boundedness} we have $\nabla_x\widetilde{c}(x,y) \in B_{\lambda R}$. 
	Since $B_{\lambda R} \subseteq \nabla_x \widetilde{c}(x,B_{\Lambda_0 R})$ by \ref{step:implicitfunction},
	injectivity of $y \mapsto \nabla_x\widetilde{c}(x,y)$ implies that we must have $y \in B_{\Lambda_0 R}$. \qedhere
\end{enumerate}
\end{proof}

\begin{remark}\label{rem:one-sided-qualitative} 
    Note that if \eqref{eq:displacement-qualitative-assumption} is replaced by smallness of the one-sided energy, i.e.,
    \begin{align*}
        \frac{1}{R^{d+2}} \int_{B_{6R} \times \RR^d} |y-x|^2 \, \dd\pi + \DD_{6R} \ll 1, 
    \end{align*}
    then Lemma \ref{lem:cone} still applies and we obtain the one-sided qualitative bound 
    \begin{align*}
        (B_{5R}\times \RR^d) \cap \supp\pi \subseteq \BB_{5R, \Lambda_0 R}. 
    \end{align*}
\end{remark}
\section{Almost-minimality with respect to Euclidean cost}\label{sec:quasiminimality} 

In this section we show that a minimizer of the optimal transport problem with cost function $c$ is an approximate minimizer for the problem with Euclidean cost function. However, in order to make full use of the Euclidean harmonic approximation result from \cite[Proposition 1.6]{GHO19} on the Eulerian side, we have to be careful in relating Lagrangian and Eulerian energies. This is where the concept of almost-minimality shows its strength, since it provides us with the missing bound of Lagrangian energy in terms of its Eulerian counterpart.

\begin{proof}[Proof of Proposition \ref{prop:quasi-minimality-lagrangian-main}] First, let us observe that we may assume in the following that 
\begin{align}\label{eq:quasimin-wlog}
	\frac{1}{2} \int |x-y|^2 \ \dd(\pi-\widetilde{\pi}) \geq 0,
\end{align}
since otherwise there is nothing to show. 
By the support assumption \eqref{eq:supp-mu-nu-main} on $\mu$ and $\nu$, the couplings $\pi$ and $\widetilde{\pi}$ satisfy
\begin{align}\label{eq:supp-pi-pitilde}
	\supp \pi \subseteq B_R \times B_R, \quad \supp\widetilde{\pi} \subseteq B_R \times B_R.
\end{align}
Together with \eqref{eq:quasimin-wlog} this implies that 
\begin{equation}\label{eq:comparison-E-Etilde}
    \EE_R(\widetilde{\pi}) \leq \EE_R(\pi).
\end{equation}

By the admissibility of $\widetilde{\pi}$, i.e.\ that $\pi$ and $\widetilde{\pi}$ have the same marginals, we may write 
    \begin{align}\label{eq:lag-quasimin-opt}
        \frac{1}{2} \int |x-y|^2 \, \dd(\pi-\widetilde{\pi})
        &= - \int ( c(x,y)+x\cdot y ) \, \dd(\pi-\widetilde{\pi}) + \int c(x,y) \, \dd(\pi-\widetilde{\pi}) \nonumber \\
       	&= - \int (\widetilde{c}(x,y)+x\cdot y ) \, \dd(\pi-\widetilde{\pi}) + \int c(x,y) \, \dd(\pi-\widetilde{\pi}),
    \end{align}
where $\widetilde{c}$ is defined as in \eqref{eq:ctilde-def}. Abbreviating  
\begin{equation*}
    \zeta(x,y) := -(\widetilde{c}(x,y) + x\cdot y),
\end{equation*}
optimality of $\pi$ with respect to the cost function $c$ implies that 
\begin{align}\label{eq:lag-quasimin-split} 
    \frac{1}{2} \int |x-y|^2 \, \dd(\pi-\widetilde{\pi}) \leq \int \zeta(x,y) \, \dd(\pi-\widetilde{\pi}). 
\end{align}
Using again the admissibility of $\widetilde{\pi}$, we may write 
\begin{align*}
    \int \zeta(x,y) \, \dd(\pi-\widetilde{\pi}) = \int (\zeta(x,y)-\zeta(x,x)) \, \dd(\pi-\widetilde{\pi}).
\end{align*}
Note that by the definition \eqref{eq:ctilde-def} of $\widetilde{c}$, the function $\zeta$ satisfies
\begin{align}
    \nabla_x\zeta(x,0) &= 0 \quad \text{for all} \, x \quad \textrm{and} \quad \nabla_y\zeta(0,y) = 0 \quad \text{for all} \, y, \label{eq:grad-y-zeta} \\
    \nabla_{xy}\zeta(x,y) &= -(\nabla_{xy}\widetilde{c}(x,y) + \II) = -(\nabla_{xy}c(x,y) + \II). \label{eq:mixeg-grad-zeta}
\end{align}
Now, by \eqref{eq:grad-y-zeta},
\begin{align*}
    \zeta(x,y) - \zeta(x,x) &= \int_0^1  \left( \nabla_y \zeta(x, sy+(1-s)x) - \nabla_y \zeta(0, sy+(1-s)x) \right) \cdot (y-x) \, \dd s \\
    &= \int_0^1 \int_0^1 x \cdot \nabla_{xy} \zeta(tx, sy+(1-s)x) (y-x) \, \dd t \dd s, 
\end{align*}
so that, using \eqref{eq:mixeg-grad-zeta} and \eqref{eq:supp-pi-pitilde}, it follows that 
\begin{align}\label{eq:lag-quasimin-first-term}
    \left| \int \zeta(x,y) \, \dd(\pi-\widetilde{\pi}) \right| 
    &\lesssim \|\nabla_{xy} c + \II\|_{L^{\infty}(B_R\times B_R)} \int_{B_R \times B_R} |x| |y-x| \, \dd(\pi+\widetilde{\pi}) \nonumber \\
    &\lesssim R \|\nabla_{xy} c + \II\|_{L^{\infty}(B_R\times B_R)} \int |y-x| \, \dd(\pi+\widetilde{\pi}). 
\end{align}
By the estimate
\begin{align*}
    (\pi+\widetilde{\pi})(\RR^d \times \RR^d) 
    = 2 \mu(\RR^d) 
    = 2 \mu(B_R) 
    \stackrel{\eqref{eq:mu-nu-close-main}}{\lesssim} R^d,
\end{align*}
and Hölder's inequality, we get 
\begin{align*}
    \int |x-y| \, \dd(\pi+\widetilde{\pi}) 
    \lesssim R^{\frac{d}{2}} \left( \int |x-y|^2 \, \dd(\pi+\widetilde{\pi})\right)^{\frac{1}{2}} 
    \overset{\eqref{eq:comparison-E-Etilde}}{\lesssim} R^{d+1} \EE(\pi)^{\frac{1}{2}}. 
\end{align*} 
\end{proof}
\hyphenation{Lip-schitz}
\section{Eulerian point of view} 
The purpose of this section is to translate almost-minimality from the Lagrangian setting, as encoded by Proposition \ref{prop:quasi-minimality-lagrangian-main}, to the Eulerian setting so that it may be plugged into the proof of the harmonic approximation result \cite[Proposition 1.7]{GHO19}. This is the purpose of Lemma \ref{prop:quasi-minimality-eulerian-main} below, which relates a (Lagrangian) coupling $\pi$, which we think of being an almost-minimizer, to its Eulerian description $(\rho,j)$ (introduced in \eqref{eq:density-flux-associated}). 

\medskip

The proof of Lemma \ref{prop:quasi-minimality-eulerian-main} relies on the fact that the Eulerian cost is always dominated by the Lagrangian one, 
	while the other inequality in general only holds for minimizers of the Euclidean transport cost. However, in the proof of Lemma~\ref{prop:quasi-minimality-eulerian-main} we can use almost-minimality of $\pi$, together with the equality of Eulerian and Lagrangian energy for minimizers of the Euclidean transport cost (Remark~\ref{rem:BB}) to overcome this nuisance. 
\begin{proof}[Proof of Lemma \ref{prop:quasi-minimality-eulerian-main}]
	Let $\pi^* \in \mathrm{argmin} \, W^2(\mu, \nu)$ be a minimizer of the Euclidean transport cost and let $(\rho^*, j^*)$ be its Eulerian description. Then by \eqref{eq:Lagrangian-dominates-Eulerian} and \eqref{eq:lagrangian-almost-min-assumption} it follows that 
	\begin{align*}
		\int \frac{1}{\rho} |j|^2 
		\leq \int |x-y|^2\,\dd \pi 
		\leq \int |x-y|^2\,\dd \pi^* + \Delta.
	\end{align*}
	Since $\pi^* \in \mathrm{argmin} \, W^2(\mu, \nu)$, we have that 
	\begin{align*}
		 \int |x-y|^2\,\dd \pi^* = \int \frac{1}{\rho^*} |j^*|^2,
	\end{align*}
	see Remark \ref{rem:BB}, so that by minimality of $\pi^*$, which implies minimality of $(\rho^*,j^*)$ in the Benamou--Brenier formulation \eqref{eq:BB} of $W^2(\mu, \nu)$, we obtain
	\begin{align*}
		\int \frac{1}{\rho} |j|^2 
		\leq \int |x-y|^2\,\dd \pi
		\leq \int \frac{1}{\rho^*} |j^*|^2 + \Delta
		\leq \int \frac{1}{\widetilde{\rho}} |\widetilde{j}|^2 + \Delta
	\end{align*}
	for any $(\widetilde\rho,\widetilde j)$ satisfying \eqref{eq:CE-tilde-main}, in particular also for $(\widetilde\rho,\widetilde j)=(\rho,j)$.
\end{proof}

The Eulerian version of almost-minimality also implies the following localized version, which will be needed for the harmonic approximation: 
\begin{corollary}\label{cor:eulerian-quasimin-2}
	Let $\pi \in \Pi(\mu, \nu)$ be a coupling between the measures $\mu$ and $\nu$ with the property that there exists a constant $\Delta <\infty$ such that 
\begin{align}
	 \int \frac{1}{2} |x-y|^2\,\dd \pi 
    \leq \int \frac{1}{2} |x-y|^2\,\dd \widetilde{\pi} + \Delta
\end{align}
for any $\widetilde{\pi}\in\Pi(\mu, \nu)$,
and let $(\rho,j)$ be its Eulerian description defined in \eqref{eq:density-flux-associated}. 
 For any $r > 0$ small enough \footnote{$r$ must be small enough so that $B_r \subset \supp \mu \cap \supp \nu$.}, let $f_r$ be the inner trace of $j$ on $\partial B_r \times (0,1)$ in the sense of \eqref{eq:CE-R}, i.e.\
	\begin{align}\label{eq:CE-r}
		\int_{B_r \times [0,1]} \partial_t \zeta \,\dd \rho + \nabla\zeta\cdot \dd j = \int_{B_r} \zeta_1 \,\dd\nu - \int_{B_r} \zeta_0 \,\dd \mu + \int_{\partial B_r \times [0,1]} \zeta\,\dd f_r.
	\end{align}
	Then for any density-flux pair $(\widetilde{\rho}, \widetilde{j})$ satisfying 
	\begin{align}\label{eq:CE-R-competitor}
		\int \partial_t \zeta \,\dd \widetilde{\rho} + \nabla\zeta\cdot \dd \widetilde{j} = \int_{B_r} \zeta_1 \,\dd\nu - \int_{B_r} \zeta_0 \,\dd \mu + \int_{\partial B_r \times [0,1]} \zeta\,\dd f_r
	\end{align}
	there holds
	\begin{align}\label{eq:Eulerian-competitor}
		\frac{1}{2}\int_{B_r\times[0,1]} \frac{1}{\rho} |j|^2 \leq \frac{1}{2}\int \frac{1}{\widetilde{\rho}} |\widetilde{j}|^2 + \Delta.	
	\end{align}
\end{corollary}

\begin{proof}
	Let $(\widetilde{\rho}, \widetilde{j})$ satisfy \eqref{eq:CE-R-competitor}. Then the density-flux pair $(\widehat{\rho}, \widehat{j}) := (\widetilde{\rho}, \widetilde{j}) + (\rho, j)\lfloor_{B_r^c\times[0,1]}$ obeys \eqref{eq:CE-tilde-main}. Indeed, 
	\begin{align*}
		\int \partial_t \zeta \,\dd\widehat{\rho} + \nabla \zeta\cdot\dd\widehat{j} 
		&= \int \partial_t \zeta \,\dd \widetilde{\rho} + \nabla\zeta\cdot \dd \widetilde{j} + \int_{B_r^c \times [0,1]} \partial_t \zeta \,\dd \rho + \nabla\zeta\cdot \dd j \\
		&= \int \partial_t \zeta \,\dd \widetilde{\rho} + \nabla\zeta\cdot \dd \widetilde{j} + \int \partial_t \zeta \,\dd \rho + \nabla\zeta\cdot \dd j - \int_{B_r \times [0,1]} \partial_t \zeta \,\dd \rho + \nabla\zeta\cdot \dd j  \\
		&\stackrel{\eqref{eq:CE-R-competitor}\&\eqref{eq:CE-weak-intro}\&\eqref{eq:CE-r}}{=} \int \zeta_1\,\dd\nu - \int \zeta_0\,\dd \mu.
	\end{align*}
	Hence, by Lemma~\ref{prop:quasi-minimality-eulerian-main} and subadditivity of the Eulerian cost it follows that 
\begin{align*}
	\frac{1}{2}\int \frac{1}{\rho} |j|^2 
	\leq \frac{1}{2}\int \frac{1}{\widehat{\rho}} |\widehat{j}|^2 + \Delta 
	\leq \frac{1}{2}\int \frac{1}{\widetilde{\rho}} |\widetilde{j}|^2 + \frac{1}{2}\int_{B_r^c\times[0,1]} \frac{1}{\rho} |j|^2 + \Delta,
\end{align*}
which implies \eqref{eq:Eulerian-competitor}.
\end{proof}

Lemma~\ref{prop:quasi-minimality-eulerian-main}, in the form of the bound \eqref{eq:almost-min-harmonic-lagrange}, allows us to relate Eulerian and Lagrangian side of the harmonic approximation result, which will be central in the application to the one-step-improvement Proposition~\ref{prop:one-step-main}. The proof of this Lagrangian version is very similar to \cite[Proof of Theorem 1.5]{GHO19}, however, we stress again that since we are not dealing with minimizers of the Euclidean transport cost, one has to be careful when passing from Eulerian to Lagrangian energies. 

\begin{proof}[Proof of Lemma~\ref{lem:harmonic-lagrange}]
	By scaling, it is enough to show that 
	\begin{align}\label{eq:harmonic-lagrangian-1-bound}
		\int_{\cross{1}} \int_0^1 |x-y + \nabla\Phi(tx+(1-t)y)|^2\,\dd t\,\dd \pi \leq \int_{B_2 \times [0,1]} \frac{1}{\rho}|j-\rho\nabla\Phi|^2 + \Delta_1.
	\end{align}
	By the $L^{\infty}$-bound, we know that if $(x,y) \in \cross{1} \cap \supp\pi$, then for any $t\in [0,1]$ there holds $t x + (1-t)y \in B_{2}$. Hence, we may estimate
	\begin{align}
		&\int_{\cross{1}} \int_0^1 |x-y + \nabla\Phi(tx+(1-t)y)|^2\,\dd t\,\dd \pi \nonumber \\
		&\quad\leq \iint_0^1 |x-y + \nabla\Phi(t x + (1-t)y)|^2 \, 1_{\{t x + (1-t)y \in B_{2}\}} \,\dd t \dd \pi \label{eq:harmonic-lagrangian-1}.
	\end{align}
	Multiplying out the square and using the  definition of $(\rho,j)$ from \eqref{eq:density-flux-associated}, we may write\footnote{Note that by an approximation argument, we may use $\zeta := 1_{B_2} |\nabla \Phi|^2 $ and $\xi := 1_{B_2} \nabla \Phi$ as test functions in \eqref{eq:density-flux-associated}.} 
\begin{align*}
	&\iint_0^1 |x-y + \nabla\Phi(t x + (1-t)y)|^2 \, 1_{\{t x + (1-t)y \in B_{2}\}} \,\dd t \dd\pi\\
	&= \iint_0^1 |x-y|^2 \, 1_{\{t x + (1-t)y \in B_{2}\}} \,\dd t \dd\pi
	+ \iint_0^1 |\nabla\Phi(t x + (1-t)y)|^2 \, 1_{\{t x + (1-t)y \in B_{2}\}}\,\dd t \dd\pi\\
	&\quad+ 2 \iint_0^1 (x-y)\cdot \nabla\Phi(t x + (1-t)y) \, 1_{\{t x + (1-t)y \in B_{2}\}}\,\dd t \dd\pi \\
	&=\int |x-y|^2 \,\dd \pi - \iint_0^1 |x-y|^2 \, 1_{\{t x + (1-t)y \in B_{2}^c\}} \,\dd t \dd\pi 
	+ \int_{B_{2}\times[0,1]} |\nabla\Phi|^2\,\dd\rho 
	- 2 \int_{B_{2}\times[0,1]} \nabla\Phi\cdot \dd j.
\end{align*}
Now note that \eqref{eq:BB-duality-local} implies a local counterpart of \eqref{eq:Lagrangian-dominates-Eulerian}: for any open set $B \subseteq \RR^d$, 
\begin{equation}\label{eq:Lagrangian-dominates-Eulerian-local}
    \int_{B \times [0,1]} \frac{1}{\rho} |j|^2 \leq  \iint_0^1 |x-y|^2 1_{\{tx + (1-t)y \in B\}} \,\dd t \dd \pi.
\end{equation} Arguing with an open $\epsilon$-neighborhood of $B$ and continuity of the right-hand side with respect to $B$, one may show that \eqref{eq:Lagrangian-dominates-Eulerian-local} holds for any closed set $B$, so that in particular 
\begin{align*}
	\int_{B_{2}^c \times [0,1]} \frac{1}{\rho}|j|^2 \leq \iint_0^1 |x-y|^2 1_{\{tx + (1-t)y \in B_{2}^c\}} \,\dd t \dd \pi, 
\end{align*} which, together with \eqref{eq:almost-min-harmonic-lagrange}, gives
\begin{align*}
	\int |x-y|^2 \,\dd \pi - \iint_0^1 |x-y|^2 \, 1_{\{t x + (1-t)y \in B_{2}^c\}} \,\dd t \dd\pi 
	&\leq \int \frac{1}{\rho} |j|^2 + \Delta_1
	- \int_{B_{2}^c\times[0,1]} \frac{1}{\rho} |j|^2  \\
	&= \int_{B_{2}\times[0,1]} \frac{1}{\rho} |j|^2 
	 + \Delta_1,
\end{align*}
so that \eqref{eq:harmonic-lagrangian-1-bound} follows from the identity
\begin{align}
	&\int_{B_{2}\times[0,1]} \frac{1}{\rho} |j|^2 
	+ \int_{B_{2}\times[0,1]} |\nabla\Phi|^2\,\dd\rho 
	- 2 \int_{B_{2}\times[0,1]} \nabla\Phi\cdot \dd j\nonumber\\
	&= \int_{B_{2}\times[0,1]} \left(\left| \frac{\dd j}{\dd\rho}\right|^2 + |\nabla\Phi|^2 - 2 \nabla\Phi\cdot \frac{\dd j}{\dd\rho} \right) \,\dd\rho  
	= \int_{B_{2}\times[0,1]} \left| \frac{\dd j}{\dd\rho} - \nabla\Phi\right|^2 \,\dd\rho \nonumber\\
	&= \int_{B_{2}\times[0,1]} \frac{1}{\rho} |j-\rho\nabla\Phi|^2. \label{eq:def-j-rho-phi}
\end{align}
\end{proof}
\section{Harmonic Approximation}\label{sec:harmonic}

In this section we sketch the proof of the (Eulerian) harmonic approximation Theorem~\ref{thm:harmonic}. 
As already noted in the introduction, the proof of Theorem~\ref{thm:harmonic} is done at the Eulerian level (as in \cite{GO17,GHO19}) by constructing a suitable competitor. 

\begin{proof}[Proof of Theorem~\ref{thm:harmonic}]
By scaling we may without loss of generality assume that $R=1$. 	
Let $(\rho,j)$ be the Eulerian description of the coupling $\pi\in \Pi(\mu, \nu)$. 
The proof of the Eulerian version of the harmonic approximation consists of the following four steps, at the heart of which is the construction of a suitable competitor (\ref{step:harmonic-competitor}). Note that since we want to make the dependence on the parameter $M$ in the $L^{\infty}$-bound \eqref{eq:Linfty-global} precise, one actually has to look a bit closer into the proofs of the corresponding statements in \cite{GHO19}, since in their presentation of the results the estimate $M\lesssim (\EE_6 + \DD_6)^{\frac{1}{d+2}}$ is used.\footnote{We do not assume that this bound holds in assumption \eqref{eq:Linfty-global}. However, in the one-step improvement Proposition~\ref{prop:one-step-main} and the consecutive iteration to obtain the $\epsilon$-regularity Theorem~\ref{thm:main}, this is of course an important ingredient, which holds in view of Proposition~\ref{prop:Linfty-main}.} 

\begin{enumerate}[label=\textsc{Step \arabic*},leftmargin=0pt,labelsep=*,itemindent=*]
	\item (\label{step:harmonic-regularised} Passage to a regularized problem). Choose a good radius $R_*\in(3,4)$ for which the flux $\overline{f}$ is well-behaved on $\partial B_{R_*}$. Actually, since we are working with $L^2$-based quantities, to be able to get $L^2$ bounds on $\nabla\Phi$, we would have to be able to estimate $\overline{f}$ in $L^2$ or at least in the Sobolev trace space $L^{\frac{2(d-1)}{d}}$. However, since the boundary fluxes $j$ are just measures and since for the approximate orthogonality (see \ref{step:harmonic-orthogonality}) a regularity theory is required up to the boundary, one first has to go to a solution $\varphi$ (with $\int_{B_{R_*}} \varphi\,\dd x = 0$) of a regularized problem 
	\begin{align}\label{eq:harmonic-appro-regular}
		\Delta \varphi = \mathrm{const} \quad \text{in } B_{R_*} 
		\quad \text{and} \quad 
		\nu_{R_*} \cdot \nabla\varphi = \widehat{g} \quad \text{on } \partial B_{R_*},
	\end{align}
	where $\mathrm{const}$ is the generic constant for which the equation has a solution, and $\widehat{g}$ is a regularization through rearrangement of $\overline{f}$ with good $L^2$ bounds (see \cite[Section 3.1.2]{GHO19} for details). 
	
	Using properties of the regularized flux $\widehat{g}$ and elliptic regularity, the error made by this approximation can be quantified as 
	\begin{align}\label{eq:harmonic-phi-diff}
		\sup_{B_2} |\nabla(\Phi-\varphi)|^2 + \sup_{B_2} |\nabla^2(\Phi-\varphi)|^2 + \sup_{B_2} |\nabla^3(\Phi-\varphi)|^2 \lesssim M^{\alpha} (\EE_6 + \DD_6)^{1+ \frac{\alpha}{2}},
	\end{align}
	for any $\alpha\in(0,1)$, see \cite[Proof of Proposition 1.7]{GHO19}.
 	Note that by the definition of $\rho$ in \eqref{eq:density-flux-associated} and assumption \eqref{eq:harmonic-assumption-mu-nu} we have that
 	\begin{align*}
 		\int_{B_2\times[0,1]} \dd{\rho} 
 		\leq \int \dd\rho 
 		= \int \dd\pi 
 		= \mu(\RR^d) 
 		= \mu(B_7)
 		\lesssim 1,
 	\end{align*}
 	so that together with \eqref{eq:harmonic-phi-diff} we may estimate, using $M\leq 1$, 
 	\begin{align*}
 		\int_{B_{2}\times[0,1]} \frac{1}{\rho} |j-\rho \nabla\Phi|^2 
 		&\lesssim \int_{B_{2}\times[0,1]} \frac{1}{\rho} |j-\rho \nabla\varphi|^2 + (\EE_6 + \DD_6)^{1+\frac{\alpha}{2}}
 	\end{align*}
 	for any $\alpha\in(0,1)$.  
 	Note that the error term is superlinear in $\EE_6 + \DD_6$. 
	
    \item\label{step:harmonic-orthogonality} (Approximate orthogonality \cite[Proof of Lemma 1.8]{GHO19}). For every $0<\tau \ll 1$ there exist $\epsilon_{\tau} > 0$, $C_{\tau} < \infty$ such that if $\EE_6 + \DD_6 \leq \epsilon_{\tau}$, then
	\begin{align}\label{eq:orthogonality}
		\int_{B_{2}\times[0,1]} \frac{1}{\rho} |j-\rho \nabla\varphi|^2 - \left( \int_{B_{R_*}\times[0,1]} \frac{1}{\rho} |j|^2 - \int_{B_{R_*}} |\nabla\varphi|^2 \right) 
		\leq \tau \EE_6 + C_{\tau} \DD_6.
	\end{align}
	The proof of \eqref{eq:orthogonality} essentially relies on the representation formula 
	\begin{align*}
	    \int_{B_{R_*}\times[0,1]}  \frac{1}{\rho} |j-\rho \nabla\varphi|^2 
	    &= \int_{B_{R_*}\times[0,1]} \frac{1}{\rho} |j|^2 
	    - \int_{B_{R_*}} |\nabla\varphi|^2 \\
	    &\quad+ 2 \int_{B_{R_*}} \varphi \,\dd(\mu-\nu) 
	    + 2 \int_{\partial B_{R_*}} \varphi \,\dd(\widehat{g}-\overline{f}) 
	    + \int_{B_{R_*}\times[0,1]} |\nabla\varphi|^2 \,(\dd \rho - \dd x).
	\end{align*}
	The three error terms in the second line of this equality are then bounded as follows. The first term uses that $\mu$ and $\nu$ are close in Wasserstein distance. An estimate on the second term relies on the fact that $\widehat{g}$ and $\overline{f}$ are close\footnote{Closeness here means closeness of $\widehat{g}_{\pm}$ and $\overline{f}_{\pm}$ (the positive and negative parts of the measures) with respect to the geodesic Wasserstein distance on $\partial B_{R_*}$.}. This bound relies on the choice of a good radius in \ref{step:harmonic-regularised} and $L^2$ estimates up to the boundary on $\nabla\varphi$. The bound on the third error term uses elliptic regularity theory and a restriction result for the Wasserstein distance, which implies that $W^2_{B_{R_*}}(\int_0^1\rho,\kappa) \lesssim \EE_6 + \DD_6$.\footnote{We note that for the case of quadratic cost and Hölder continuous densities treated in \cite{GO17} a bound on this term is easy due to McCann's displacement convexity, which implies that $\rho\leq 1$ up to a small error.} This estimate actually requires a further regularization of $\widehat{g}$, and by relying on interior regularity estimates explains why one has to go from $B_{R_*}$ to the slightly smaller ball $B_{2}$ in the estimate \eqref{eq:orthogonality}.
A close inspection of \cite[Proof of Lemma 1.8]{GHO19} shows that the term involving $M$ in these error estimates comes in product with a superlinear power of $\EE_6 + \DD_6$ as in \ref{step:harmonic-regularised}, so that we may bound $M\leq 1$ and still be able to obtain an arbitrarily small prefactor in \eqref{eq:orthogonality} by choosing $\EE_6 + \DD_6$ small enough.
	
	\item\label{step:harmonic-competitor} (Construction of a competitor \cite[Proof of Lemma 1.9]{GHO19}). For every $0<\tau\ll 1$ there exist $\epsilon_{\tau}>0$ and $C,C_{\tau}<\infty$ such that if $\EE_6+\DD_6\leq \epsilon_{\tau}$, then there exists a density-flux pair $(\widetilde{\rho}, \widetilde{j})$ satisfying \eqref{eq:CE-R-competitor} for $r=R_*$, and such that
	\begin{align}\label{eq:competitor}
		\int \frac{1}{\widetilde{\rho}} |\widetilde{j}|^2 - \int_{B_{R_*}} |\nabla\varphi|^2 
		&\leq \left(\tau+ \frac{CM}{\tau}\right) \EE_6 + C_{\tau} \DD_6.
	\end{align}
	
    \item\label{step:harmonic-almostmin} (Almost-minimality on the Eulerian level). Since the density-flux pair $(\widetilde{\rho}, \widetilde{j})$ satisfies \eqref{eq:CE-R-competitor} for $r=R_*$, Corollary~\ref{cor:eulerian-quasimin-2} implies that 
	\begin{align}
		\int_{B_{R_*}\times[0,1]} \frac{1}{\rho} |j|^2 - \int_{B_{R_*}} |\nabla\varphi|^2 
		&\leq \int \frac{1}{\widetilde{\rho}} |\widetilde{j}|^2  - \int_{B_{R_*}} |\nabla\varphi|^2 + \Delta.
	\end{align}

\end{enumerate}
Combining the above steps, we have proved that for any $0<\tau\ll 1$ there exist $\epsilon_{\tau}>0$ and $C_{\tau}<\infty$ such that if $\Phi$ is the solution of \eqref{eq:harmonic-approx}, then 
\begin{align}\label{eq:harmonic-eulerian}
	\int_{B_{2}\times[0,1]} \frac{1}{\rho} |j-\rho\nabla\Phi|^2 \leq \left(\tau+ \frac{CM}{\tau}\right) \EE_6 + C_{\tau} \DD_6 + \Delta.
\end{align}
\end{proof}
\section{One-step improvement}

The following proposition is a one-step improvement result, which will be the basis of a Campanato iteration in Theorem \ref{thm:main}. Note that the iteration is more complicated than in \cite{GO17}, because at each step we have to restrict the $c$-optimal coupling to the smaller cross where the $L^{\infty}$-bound holds to be able to apply the harmonic approximation result. As a consequence, we have to make sure that the qualitative bound \eqref{eq:displacement-qualitative-assumption} on the displacement (which is an important ingredient in obtaining quantitative version of the $L^{\infty}$-bound \eqref{eq:Linfty-main}) is propagated in each step of the iteration.\footnote{Alternatively, one could devise an argument based on the fact that the qualitative $L^{\infty}$ bound only depends on the cost through its global properties \ref{item:cost-cont}--\ref{item:cost-non-deg} and that the set of cost functions considered in the iteration is relatively compact.}

We start with the short proof of Lemma~\ref{lem:optimal-restriction}, which is the starting point of each iteration step.

\begin{proof}[Proof of Lemma~\ref{lem:optimal-restriction}]
	Let $\pi\in\Pi(\mu, \nu)$ be a $c$-optimal coupling, and let $\pi_{\Omega}:=\pi\lfloor_{\Omega}$ be its restriction to a subset $\Omega\subseteq \RR^d \times\RR^d$. Then $\pi_{\Omega} \in \Pi(\mu_{\Omega}, \nu_{\Omega})$ where the marginal measures $\mu_{\Omega}$, $\nu_{\Omega}$ are defined in \eqref{eq:mu-nu-Omega}. 
	
	Given any coupling $\widetilde{\pi} \in \Pi(\mu_{\Omega}, \nu_{\Omega})$, we can define $\widehat{\pi}:= \pi - \pi_{\Omega} + \widetilde{\pi}$. It is easy to see that $\widehat{\pi}$ is an admissible coupling between the measures $\mu$ and $\nu$, hence by $c$-optimality of $\pi$, we obtain by the additivity of the cost functional with respect to the transference plan
	\begin{align*}
		\int c(x,y) \,\dd\pi 
		\leq \int c(x,y) \,\dd\widehat{\pi}
		= \int c(x,y) \,\dd\pi - \int c(x,y) \,\dd\pi_{\Omega} + \int c(x,y) \,\dd\widetilde{\pi},
	\end{align*}
	hence
	\begin{align*}
		\int c(x,y) \,\dd\pi_{\Omega} \leq \int c(x,y) \,\dd\widetilde{\pi} \quad \text{for any} \quad \widetilde{\pi}\in\Pi(\mu_{\Omega}, \nu_{\Omega}),
	\end{align*}
	that is, $\pi_{\Omega}$ is a $c$-optimal coupling between $\mu_{\Omega}$ and $\nu_{\Omega}$.
\end{proof}

As a direct application and as a further preparation for the one-step-improvement we present the short proof of Corollary~\ref{cor:restriction-properties}. 
\begin{proof}[Proof of Corollary~\ref{cor:restriction-properties}]
	Let $\pi\in\Pi(\mu,\nu)$ be $c$-optimal. Then by Lemma~\ref{lem:optimal-restriction} the coupling $\pi_R := \pi\lfloor_{\cross{R}}$ is a $c$-optimal coupling between the measures $\mu_R$ and $\nu_R$ defined via 
	\begin{align*}
		\mu_R(A) = \pi((A\times\RR^d)\cap\cross{R}), \quad \nu_R(A) = \pi((\RR^d\times A)\cap\cross{R}),
	\end{align*}
	for any $A\subseteq\RR^d$ Borel. 
	In particular, we have that 
	\begin{align*}
		\mu_R(A) \leq \pi(A\times \RR^d) = \mu(A).
	\end{align*}
	If $A\subseteq B_{2R}^c$, then 
	\begin{align*}
		\mu_R(A) = \pi(((A\cap B_R)\times \RR^d) \cup (A\times B_R)) = \pi(A\times B_R) = 0,
	\end{align*}
	since for any $(x,y)\in\supp\pi_R$, by assumption we have that $|x-y| \leq MR \leq R$, so $A\times B_R \cap \supp\pi = \emptyset$. Hence, $\supp\mu_R \subseteq B_{2R}$. 
	Similarly, if $A\subseteq B_R$, then
	\begin{align*}
		\mu_R(A) = \pi((A \times \RR^d) \cup (A\times B_R)) = \pi(A\times \RR^d) = \mu(A),
	\end{align*}
	which implies that $\mu_R = \mu$ on $B_R$. 
	
	By symmetry, the same properties hold for $\nu_R$.
\end{proof}

We now give the proof of the one-step improvement Proposition \ref{prop:one-step-main}, which is the working horse of the Campanato iteration.
\begin{proof}[Proof of Proposition \ref{prop:one-step-main}]
\begin{enumerate}[label=\textsc{Step \arabic*},leftmargin=0pt,labelsep=*,itemindent=*]
	\item (Rescaling). Without loss of generality we can assume that $R=1$. Hence, to simplify notation we will also use the shorthand 
	\begin{align*}
	    \EE := \EE_9(\pi), \quad &\DD := \DD_9(\rho_0, \rho_1), \\
	    \text{and} \quad [\rho_0]_{\alpha} := [\rho_0]_{\alpha, 9}, \quad [\rho_1]_{\alpha} &:= [\rho_1]_{\alpha, 9}, \quad \left[ \nabla_{xy}c \right]_{\alpha} := \left[ \nabla_{xy}c \right]_{\alpha, 9}.
	\end{align*}
	Note that by the normalization $\rho_0(0)=\rho_1(0)=1$ we have that $\frac{1}{2}\leq \rho_j \leq 2$ for $j=1,2$, provided that $[\rho_0]_{\alpha} + [\rho_1]_{\alpha}$ is small enough. In particular, Lemma~\ref{lem:D-Holder_rho} then implies that 
	\begin{align*}
		\DD \lesssim [\rho_0]_{\alpha}^2 + [\rho_1]_{\alpha}^2.
	\end{align*}

	Indeed, let 
	$(\widetilde{x},\widetilde{y}) = S_R(x,y) := (R^{-1}x,R^{-1}y)$ and set 
	\begin{align*}
		\widetilde{\rho}_0(\widetilde{x}) = \rho_0(R\widetilde{x}), \quad 
		\widetilde{\rho}_1(\widetilde{y}) = \rho_1(R\widetilde{y}), \quad 
		\widetilde{c}(\widetilde{x},\widetilde{y}) = R^{-2} c(R\widetilde{x},R\widetilde{y}),
	\end{align*}
	and $\widetilde{b} = R^{-1} b$. We still have $\widetilde{\rho}_0(0) = \widetilde{\rho}_1(0) = 1$ and $\nabla_{\widetilde{x}\widetilde{y}}\widetilde{c}(0,0) = -\II$, and $\widetilde{\pi} := R^{-d}(S_R)_{\#}\pi$ is the $\widetilde{c}$-optimal coupling between $\widetilde{\rho}_0$ and $\widetilde{\rho}_1$. 
	
	\medskip
	\item\label{item:onestep-Linfty} ($L^{\infty}$-bound for $\pi$).
	We claim that the coupling $\pi$ satisfies the $L^{\infty}$-bound 
	\begin{align}\label{eq:Linfty-onestep}
		|x-y| \leq M \quad \text{for all } (x,y) \in \supp\pi\lfloor_{\cross{6}}
	\end{align}
	with $M = C_{\infty} (\EE+\DD)^{\frac{1}{d+2}}$ for some constant $C_{\infty}<\infty$ depending only on $d$. 
	
	In view of the harmonic approximation Theorem~\ref{thm:harmonic} it is therefore natural to consider the coupling $\pi\lfloor_{\cross{6}}$, which by Lemma~\ref{lem:optimal-restriction} and Corollary~\ref{cor:restriction-properties} is a $c$-optimal coupling between its own marginals, denoted by $\mu$ and $\nu$. 

	Indeed, since $\pi$ is $c$-optimal, its support is $c$-monotone. By assumptions \eqref{eq:qualitative-Linfty-onestep}, \eqref{eq:smallness-onestep-main}, and Remark~\ref{rem:qualitative-quantitative} we may therefore
	estimate 
	\begin{align}\label{eq:onestep-assumption-Linfty}
			\|\nabla_{xy}c+\II\|_{\mathcal{C}^0(\BB_{8,\Lambda})} \lesssim_{\Lambda} [\nabla_{xy}c]_{\alpha}\ll_{\Lambda} 1.
	\end{align}
	 Hence, appealing to Proposition~\ref{prop:Linfty-main} (suitably rescaled), we obtain the claimed $L^{\infty}$-bound \eqref{eq:Linfty-onestep}.
	Notice that the dependence of the smallness assumption \eqref{eq:smallness-onestep-main} on $\Lambda$ only enters through the $\Lambda$-dependence of \eqref{eq:onestep-assumption-Linfty}.

    \medskip
    \item \label{step:onestep-marginals} (Properties of the marginals $\mu$ and $\nu$ of $\pi\lfloor_{\cross{6}}$). 
    We claim that the marginals $\mu$ and $\nu$ of $\pi\lfloor_{\cross{6}}$ satisfy 
    \begin{align*}
        &\supp\mu, \supp\nu \subseteq B_7, \quad \text{and} \quad 
        \mu(B_7) \leq 2 |B_7|.
    \end{align*}
    
    Indeed, due to the $L^{\infty}$-bound \eqref{eq:Linfty-onestep} the marginal measures are supported on $B_{7}$ if $\EE+\DD\ll 1$ (such that $M\leq 1$). Furthermore, from Corollary \ref{cor:restriction-properties} we have that $\mu = \rho_0 \,\dd x$ and $\nu = \rho_1 \,\dd y$ on $B_6$, as well as $\mu\leq \rho_0\,\dd x$, which implies that 
	\begin{align*}
		\mu(B_7) \leq \int_{B_7} \rho_0 \,\dd x \leq (1+[\rho_0]_{\alpha}) |B_7| \leq 2 |B_7| 
	\end{align*}
	since by assumption \eqref{eq:smallness-onestep-main} we may assume that $[\rho_0]_{\alpha}\leq 1$. 
    
	\medskip
	\item \label{step:onestep-harmonic-applicable} (Almost-minimality of $\pi\lfloor_{\cross{6}}$ and applicability of the harmonic approximation Theorem~\ref{thm:harmonic}).
	We show next that the coupling $\pi\lfloor_{\cross{6}}$ is an almost-minimizer of the Euclidean transport problem in the sense that for any $\widetilde{\pi}\in\Pi(\mu,\nu)$ there holds
	\begin{align}\label{eq:almost-minimality-lagrangian-onestep}
		\frac{1}{2} \int |x-y|^2 \,\dd \pi\lfloor_{\cross{6}} \leq \frac{1}{2} \int |x-y|^2\,\dd \widetilde{\pi} + \Delta ,
	\end{align}
	with $\Delta = C [\nabla_{xy}c]_{\alpha} \EE^{\frac{1}{2}}$, and that $\pi\lfloor_{\cross{6}}$ satisfies the assumptions of Theorem~\ref{thm:harmonic}.
	
	Indeed, by \ref{step:onestep-marginals} and the $c$-optimality of $\pi\lfloor_{\cross{6}}$, we may  apply Proposition~\ref{prop:quasi-minimality-lagrangian-main} to obtain inequality \eqref{eq:almost-minimality-lagrangian-onestep} with 
	\begin{align}\label{eq:Delta-bound}
		\Delta = C \|\nabla_{xy} c + \II\|_{\mathcal{C}^0(B_7\times B_7)} \EE_7^{\frac{1}{2}} \lesssim [\nabla_{xy}c]_{\alpha} \EE^{\frac{1}{2}},
	\end{align}
	where we also used that $\EE_7 \lesssim \EE$. 
	
	Note that in view of Lemma~\ref{prop:quasi-minimality-eulerian-main} the Eulerian description of $\pi\lfloor_{\cross{6}}$ satisfies \eqref{eq:almostminimal-harmonic}. Together with \ref{item:onestep-Linfty} and \ref{step:onestep-marginals} this implies that the assumptions of Theorem~\ref{thm:harmonic} are fulfilled.

	\medskip
    \item (Definition of $b$ and $B$ and proof of estimate \eqref{eq:B-b-est-main}).
    For given $\tau > 0$, to be chosen later, let $\epsilon_{\tau}$ and $C_{\tau}$ be as in the harmonic approximation result (Theorem \ref{thm:harmonic}). 
    We claim that, with the harmonic gradient field $\nabla\Phi$ defined in \eqref{eq:harmonic-approx},
    \begin{align}\label{def:b-B}
        b:=\nabla\Phi(0), \quad A:=\nabla^2\Phi(0), \quad \text{and} \quad B:= \ee^{\frac{A}{2}},
    \end{align}
	satisfy the bounds \eqref{eq:B-b-est-main}.
    
    Since $\rho_0(0) = \rho_1(0) = 1$,  we have $\frac{1}{2} \leq \rho_j \leq 2$, $j=0,1$, provided $[\rho_0]_{\alpha}^2 + [\rho_1]_{\alpha}^2$ is small enough, so that \eqref{eq:D-bound-rho} holds. In particular, since $\mu$ agrees with $\rho_0 \,\dd x$ on $B_6$ and $\nu$ agrees with $\rho_1 \,\dd y$ on $B_6$, 
    we may estimate
    \begin{align*}
    	\DD_6(\mu, \nu) = \DD_6(\rho_0, \rho_1) 
    	\stackrel{\eqref{eq:D-bound-rho}}{\lesssim} [\rho_0]_{\alpha,6}^2 + [\rho_1]_{\alpha,6}^2 
    	\lesssim [\rho_0]_{\alpha}^2 + [\rho_1]_{\alpha}^2.
    \end{align*} 
    Assume now that 
    \begin{align}\label{eq:epsilon-tau-1}
	    \mathcal{E}_6(\pi) + \mathcal{D}_6(\mu,\nu) \lesssim \EE + [\rho_0]_{\alpha}^2 + [\rho_1]_{\alpha}^2 \leq \epsilon_{\tau},
    \end{align} 
    so that Theorem~\ref{thm:harmonic} allows us to define the vector $b$ and the matrices $A,B$ as in \eqref{def:b-B}. Note that since $A$ is symmetric, so is the matrix $B$.
    By \eqref{eq:harmonic-estimate} and \eqref{eq:D-bound-rho}, we then have
    \begin{align*}
        |b|^2 
        \lesssim \mathcal{E}_6(\pi) +\mathcal{D}_6(\mu, \nu) 
        \lesssim \mathcal{E}+[\rho_0]_{\alpha}^2 + [\rho_1]_{\alpha}^2,
    \end{align*}
    and the same estimate holds for $|A|^2$, so that 
    \begin{align*}
        |B-\II|^2 \lesssim \mathcal{E}+[\rho_0]_{\alpha}^2 + [\rho_1]_{\alpha}^2.
    \end{align*} 
    This proves the estimate \eqref{eq:B-b-est-main}. Furthermore, recalling the definition of $\Phi$ in \eqref{eq:harmonic-approx}, we bound
    \begin{align}\label{eq:detB}
    	|1-\det B^2|^2
    	&= \left|1- \ee^{\mathrm{tr}A}\right|^2
    	= \left|1- \ee^{\Delta\Phi(0)}\right|^2
    	= \left|1-\ee^{\kappa_{\mu}-\kappa_{\nu}}\right|^2 \nonumber\\
    	&\lesssim \left| \kappa_{\nu} - \kappa_{\mu} \right|^2
    	\overset{\eqref{eq:D-def}}{\lesssim} \mathcal{D}_{R_*}(\mu, \nu)
    	\overset{\eqref{eq:D-bound-rho}}{\lesssim} R_*^{2\alpha} \left( [\rho_0]_{\alpha, R_*}^2 + [\rho_1]_{\alpha, R_*}^2 \right) 
    	\lesssim [\rho_0]_{\alpha}^2 + [\rho_1]_{\alpha}^2.
    \end{align}
    In view of \eqref{eq:B-b-est-main} we may assume that $b \in B_9$, so that%
    \footnote{The first inequality follows from $\rho_1\geq \frac{1}{2}$ and the Lipschitz continuity of the function $t\mapsto t^{\frac{1}{d}}$ away from zero.}
	\begin{align}\label{eq:delta-lambda-est}
	    |\rho_1(b)^{\frac{1}{d}} - \rho_1(0)^{\frac{1}{d}}|^2 \lesssim |\rho_1(b) - \rho_1(0)|^2 \leq |b|^{2\alpha} [\rho_1]_{\alpha}^2 \lesssim [\rho_1]_{\alpha}^2, 
	\end{align}
	and similarly for $D = -\nabla_{xy}c(0,b)$,
	\begin{align}\label{eq:M-est-2}
	    |D-\II|^2 &\lesssim \left[\nabla_{xy}c\right]_{\alpha}^2,
	\end{align} 
	which implies that
	\begin{align}\label{eq:detM}
		|1-\det D|^2 
		\lesssim |D-\II|^2 
		\lesssim [\nabla_{xy}c]_{\alpha}^2.
	\end{align}
	Now, noticing that $\gamma = \rho_1(b)^{\frac{1}{d}} \left(\frac{|\det B|^2}{|\det D|}\right)^{\frac{1}{d}}$ because $\rho_0(0) = 1$, \eqref{eq:detB}, \eqref{eq:delta-lambda-est} together with $\rho_1(0)=1$, and \eqref{eq:detM} imply
	\begin{align}\label{eq:gamma-est}
		|1-\gamma|^2 
		\lesssim [\rho_0]_{\alpha}^2 + [\rho_1]_{\alpha}^2 + [\nabla_{xy}c]_{\alpha}^2.
	\end{align}

	\medskip
	\item \label{step:mapping-properties} (Mapping properties of $Q$).
	We show that if $\EE + \DD + [\nabla_{xy}c]_{\alpha}^2 \ll \theta^2$,\footnote{Note that we have not yet fixed $\theta$.} then 
	\begin{align}\label{eq:cross-inclusion}
	    Q^{-1}(\cross{\theta}) \subseteq \cross{2\theta}.
	\end{align} 
	
	Indeed, we have
	\begin{align*}
		Q^{-1}(\cross{\theta})
		&= Q^{-1}(B_{\theta} \times \RR^d) \cup Q^{-1}(\RR^d \times B_{\theta}) \\
		&= \left(B^{-1}B_{\theta} \times \RR^d\right) \cup \left(\RR^d \times ( \gamma^{-1} D^{-*}BB_{\theta} + b)\right),
	\end{align*}
	and from \eqref{eq:B-b-est-main}, \eqref{eq:M-est-2}, and \eqref{eq:gamma-est} it follows that $B^{-1}B_{\theta} \subseteq B_{2\theta}$ and $\gamma^{-1} D^{-*}BB_{\theta} + b \subseteq B_{2\theta}$ if $\EE + \DD + [\nabla_{xy}c]_{\alpha}^2 \ll \theta^2$, which we assume to be true from now on. 
	
	\medskip
	\item (Transformation of the displacement)
	We next show that for all $(\widehat{x},\widehat{y}) = Q(x,y) \in\RR^d \times \RR^d$ and $t\in[0,1]$ there holds
	\begin{align}\label{eq:transformation-displacement}
	    B(\widehat{y}-\widehat{x}) = y-x - \nabla\Phi(t x + (1-t)y) + \delta,
	\end{align}
	where the error $\delta$ is controlled by 
	\begin{align}
	    |\delta| 
	    &\lesssim \left( |\gamma-1| + |D-\II| \right) (|y| + |b|)  + |\nabla^2\Phi(0)|^2 |x|  + |\nabla^2\Phi(0)| |x-y| \nonumber\\
	    &\quad + \sup_{s\in[0,1]} |\nabla^3\Phi(s(tx + (1-t)y))| \,  |tx + (1-t)y|^2. \label{eq:onestep-delta}
	\end{align}
	
	Indeed, by \eqref{eq:Q-transformation} we have
	\begin{align*}
	    B(\widehat{y}-\widehat{x}) 
	    &= \gamma D^* (y-b) - B^2 x 
	    = y-b-B^2 x + (\gamma D^* - \II)(y-b),
	\end{align*}
	and the second term, which will turn out to be an error term, can be bounded by 
	\begin{align}\label{eq:onestep-error-1}
	    |(\gamma D^* - \II)(y-b)|
	    \leq \left(|\gamma-1||D^*| + |D^*-\II|\right) |y-b| 
	    \stackrel{\eqref{eq:M-est-2}}{\lesssim} \left( |\gamma-1| + |D-\II| \right) (|y| + |b|).
	\end{align}
	We show next that $B^2 x + b \approx x + \nabla\Phi(tx + (1-t)y)$, with an error that can be controlled. This relies on the fact that,
	\begin{align}\label{eq:B2-expansion}
	    B^2 = \ee^{A} = \II + A + (\ee^{A} - \II-A),
	\end{align}
	where
	\begin{align}\label{eq:onestep-error-2}
	    |\ee^{A} - \II-A| \lesssim |A|^2,
	\end{align}
	and Taylor approximation. Indeed,
	\begin{align}
	    &|\nabla\Phi(tx + (1-t)y) - (\nabla\Phi(0) + \nabla^2 \Phi(0) (tx + (1-t)y))| \nonumber\\
	    &\lesssim \sup_{s\in[0,1]} |\nabla^3\Phi(s(tx + (1-t)y))| \,  |tx + (1-t)y|^2 \label{eq:onestep-taylor}
	\end{align}
	so that with \eqref{eq:B2-expansion}, \eqref{eq:onestep-error-2}, and \eqref{def:b-B},
	\begin{align}
	    &|B^2 x + b - (x + \nabla\Phi(tx + (1-t)y)) | 
        \lesssim |\nabla^2\Phi(0) x + \nabla\Phi(0) - \nabla\Phi(tx + (1-t)y))| + |\nabla^2\Phi(0)|^2 |x| \nonumber\\
        &\leq |\nabla^2\Phi(0) (tx + (1-t)y) + \nabla\Phi(0) - \nabla\Phi(tx + (1-t)y))| + |1-t| |\nabla^2\Phi(0)| |x-y| + |\nabla^2\Phi(0)|^2 |x| \nonumber\\
        &\stackrel{\eqref{eq:onestep-taylor}}{\lesssim} 
        \sup_{s\in[0,1]} |\nabla^3\Phi(s(tx + (1-t)y))| \,  |tx + (1-t)y|^2 + |\nabla^2\Phi(0)| |x-y| + |\nabla^2\Phi(0)|^2 |x|. \label{eq:onestep-error-3}
	\end{align}
Collecting all the error terms from \eqref{eq:onestep-error-1} and \eqref{eq:onestep-error-3} then yields \eqref{eq:transformation-displacement}.

	\medskip
	\item (Proof of the energy estimate \eqref{eq:one-step-conclusion-main}).
	In this final step we prove the energy estimate \eqref{eq:one-step-conclusion-main}. 
	
	To this end, let us compute
	\begin{align}
		\theta^{d+2} \mathcal{E}_{\theta}(\widehat{\pi})
		&= \int_{\cross{\theta}} |\widehat{x}-\widehat{y}|^2 \, \dd\widehat{\pi} 
		\leq |B^{-1}|^2 \int_{\cross{\theta}} |B(\widehat{x}-\widehat{y})|^2 \, \dd\widehat{\pi} \nonumber 
		= |B^{-1}|^2 \int_0^1\int_{\cross{\theta}} |B(\widehat{x}-\widehat{y})|^2 \, \dd\widehat{\pi} \,\dd t \\
		&\stackrel{\eqref{eq:pihat-def}\&\eqref{eq:transformation-displacement}}{\lesssim} |\det B| |B^{-1}|^2  \int_{Q^{-1}(\cross{\theta})}\int_0^1 \left(|y-x - \nabla\Phi(t x + (1-t)y)|^2 + |\delta|^2 \right)\,\dd t \, \dd \pi \nonumber\\
		&\overset{\eqref{eq:cross-inclusion}\&\eqref{eq:B-b-est-main}}{\lesssim} \int_{\cross{2\theta}}\int_0^1 |y-x - \nabla\Phi(t x + (1-t)y)|^2 \,\dd t\,\dd\pi  + \int_{\cross{2\theta}}\int_0^1 |\delta|^2\,\dd t \,  \dd \pi. \label{eq:one-step-1}
	\end{align}
	For the first term on the right-hand side of \eqref{eq:one-step-1}, we use the harmonic approximation theorem in its Lagrangian version (Lemma~\ref{lem:harmonic-lagrange}), which is applicable by \ref{step:onestep-harmonic-applicable}. 
    In particular, we may estimate, assuming that $\theta \leq \frac{1}{2}$, 
	\begin{align*}
        \int_{\cross{2\theta}}\int_0^1 |y-x - \nabla\Phi(t x + (1-t)y)|^2 \,\dd t\,\dd\pi
	    &\leq
	    \int_{\cross{1}}\int_0^1 |y-x - \nabla\Phi(t x + (1-t)y)|^2 \,\dd t\,\dd\pi \nonumber\\
	    &\leq \int_{B_2 \times[0,1]} \frac{1}{\rho}|j-\rho\nabla\Phi|^2 + \Delta \nonumber\\
	    &\leq \left(\tau + \frac{C M}{\tau}\right) \EE + C_{\tau} \DD + \Delta,
	\end{align*}
	so that by \eqref{eq:Delta-bound}, \eqref{eq:D-bound-rho}, and Young's inequality,
	\begin{align} \label{eq:harmonic-final-1}
        &\int_{\cross{2\theta}}\int_0^1 |y-x - \nabla\Phi(t x + (1-t)y)|^2 \,\dd t\,\dd\pi\nonumber\\
        &\quad\lesssim \left(\tau + \frac{M}{\tau}\right) \EE + C_{\tau} ([\rho_0]_{\alpha}^2 + [\rho_1]_{\alpha}^2 + [\nabla_{xy}c]_{\alpha}^2).
    \end{align}
	
	To estimate the error term $\int_{\cross{2\theta}}\int_0^1 |\delta|^2\,\dd t\,\dd\pi$ we use the following consequence of the $L^{\infty}$-bound \eqref{eq:Linfty-onestep}: if $\EE + [\rho_0]_{\alpha}^2 + [\rho_1]_{\alpha}^2 \ll \theta^{d+2}$, which holds in view of \eqref{eq:smallness-onestep-main}, then
	\begin{align}\label{eq:cross-ball-inclusion}
		\cross{2\theta} \cap \supp\pi \subseteq	B_{3\theta} \times B_{3\theta}.
	\end{align}
	With this in mind, we will bound each term in \eqref{eq:onestep-delta} separately. 
    Note that 
    \begin{align*}
        \int_{\cross{2\theta}} \left(|x|^2 + |y|^2\right) \,\dd \pi 
        &\stackrel{\eqref{eq:cross-ball-inclusion}}{\leq} \int_{ B_{3\theta}\times \RR^d} |x|^2\,\dd\pi  + \int_{\RR^d \times B_{3\theta}} |y|^2\,\dd\pi \\
        &= \int_{B_{3\theta}} |x|^2 \, \rho_0(x)\,\dd x  + \int_{B_{3\theta}} |y|^2 \, \rho_1(y)\,\dd y 
        \lesssim \theta^{d+2},
    \end{align*}
    and similarly $\int_{\cross{2\theta}} \,\dd\pi \lesssim \theta^d$.
    Moreover, by \eqref{eq:cross-ball-inclusion}, we may estimate for any $(x,y) \in \cross{2\theta} \cap \supp \pi$
    \begin{align*}
        \sup_{s\in[0,1]} |\nabla^3\Phi(s(tx + (1-t)y))| \,  |tx + (1-t)y|^2 \lesssim \theta^2 \sup_{B_{3\theta}} |\nabla^3\Phi| \, 
        \lesssim \theta^2 \sup_{B_2} |\nabla^3\Phi|  \, .
    \end{align*}
    Together with the bound \eqref{eq:B-b-est-main} and since we assumed that $\EE + [\rho_0]_{\alpha}^2 + [\rho_1]_{\alpha}^2 \ll \theta^2$, which implies that $|b|^2\lesssim \theta^2$, it follows that 
    \begin{align*}
        \int_{\cross{2\theta}}\int_0^1 |\delta|^2\,\dd t \,  \dd \pi
        &\lesssim \theta^{d+2} \left(|1-\gamma|^2 + |D-\II|^2 \right) \,  
        + \theta^{d+2} |\nabla^2\Phi(0)|^4  \\
        &\quad+ |\nabla^2\Phi(0)|^2 \int_{\cross{2\theta}} |x-y|^2\,\dd\pi 
        + \theta^{d+4} \sup_{B_2}|\nabla^3\Phi|^2 .
    \end{align*}
Hence with the bounds \eqref{eq:gamma-est}, \eqref{eq:M-est-2}, and  \eqref{eq:harmonic-estimate}, together with $\EE + [\rho_0]_{\alpha}^2 + [\rho_1]_{\alpha}^2 \ll \theta^2$ and $\theta \leq 1$, we obtain 
\begin{align}
     \int_{\cross{2\theta}}\int_0^1 |\delta|^2\,\dd t \,  \dd \pi 
      &\lesssim \theta^{d+2}  \left( [\rho_0]_{\alpha}^2 + [\rho_1]_{\alpha}^2 + [\nabla_{xy}c]_{\alpha}^2 \right) 
      + \theta^{d+2} (\EE +  [\rho_0]_{\alpha}^2 + [\rho_1]_{\alpha}^2)^2  \nonumber\\
      &\quad +   (\EE +  [\rho_0]_{\alpha}^2 + [\rho_1]_{\alpha}^2) \theta^{d+2} \EE\, 
      + \theta^{d+4} (\EE +  [\rho_0]_{\alpha}^2 + [\rho_1]_{\alpha}^2)  \nonumber\\
      &\lesssim \theta^{d+2} \left( [\rho_0]_{\alpha}^2 + [\rho_1]_{\alpha}^2 + [\nabla_{xy}c]_{\alpha}^2 \right)  
      + \theta^{d+4} \EE  . \label{eq:harmonic-final-2}
\end{align}

    Inserting the estimates \eqref{eq:harmonic-final-1} and \eqref{eq:harmonic-final-2} into \eqref{eq:one-step-1} yields the existence of a constant $C=C(d,\alpha)$ such that 
	\begin{align*}
		\mathcal{E}_{\theta}(\widehat{\pi})
		&\leq C\left( \tau \theta^{-(d+2)} + \frac{M}{\tau \theta^{d+2}} + \theta^2 \right) \EE + C_{\tau,\theta} \left( [\rho_0]_{\alpha}^2 + [\rho_1]_{\alpha}^2 + [\nabla_{xy}c]_{\alpha}^2 \right).
	\end{align*}    
    We now fix $\theta=\theta(d,\alpha,\beta)$ such that $C \theta^2 \leq \frac{1}{3} \theta^{2\beta}$, which is possible since $\beta<1$, and then $\tau$ small enough so that $C \tau \theta^{-(d+2)} \leq \frac{1}{3} \theta^{2\beta}$. This fixes $\epsilon_{\tau}$ given by Theorem \ref{thm:harmonic}. Finally, we may choose $\EE + [\rho_0]_{\alpha}^2 + [\rho_1]_{\alpha}^2$ small enough so that 
    \begin{align*}
    	C\frac{M}{\tau \theta^{d+2}} 
    	\stackrel{\eqref{eq:Linfty-onestep}}{=} C \frac{C_{\infty}(\EE + \DD)^{\frac{1}{d+2}}}{\tau \theta^{d+2}}
    	\leq \frac{1}{3}\theta^{2\beta},
    \end{align*}
   	yielding \eqref{eq:one-step-conclusion-main}.
	
	\medskip
	\item (Proof of \eqref{eq:qualitative-theta-inclusion-main}).
	Let $(\widehat{x}, \widehat{y}) \in \cross{\frac{8}{9}\theta} \cap \supp \widehat{\pi}$ and set $(x,y) := Q^{-1}(\widehat{x}, \widehat{y})$, see Lemma \ref{lem:coordinate-change} for the definition of $Q$. Then by  \ref{lem:equivalence-support}, there holds  $(x,y) \in \supp\pi$.
	Similarly to \ref{step:mapping-properties}, we have that $Q^{-1}(\cross{\frac{8}{9}\theta}) \subseteq \cross{\theta}$, provided that $\EE + \DD + [\nabla_{xy}c]_{\alpha}^2 \ll \theta^2$, so that
    \begin{align*}
        (x,y) \in \cross{\theta} \cap \supp\pi.
    \end{align*} Using the $L^{\infty}$ bound \eqref{eq:Linfty-main} (notice that $\theta \leq \frac{2}{3}$) and the fact that we assumed $\EE + \DD \ll \theta^{d+2}$, we obtain 
    \begin{align*}
        |x-y| \leq M = C_{\infty} (\EE + \DD)^{\frac{1}{d+2}} \leq \theta, 
    \end{align*} so that $(x,y) \in \BB_{\theta, 2\theta}$. It thus follows that $(\widehat{x}, \widehat{y}) \in Q(\BB_{\theta, 2\theta})$. Following the proof of \ref{step:mapping-properties}, we have 
    \begin{align*}
        Q(\BB_{\theta, 2\theta}) \subseteq \BB_{2\theta, 3\theta},
    \end{align*} and conclude that
        $(\widehat{x}, \widehat{y}) \in \BB_{2\theta, 3\theta} \cap \cross{\frac{8}{9}\theta} = \BB_{\frac{8}{9}\theta, 3\theta}$. \qedhere
\end{enumerate}
\end{proof}
\section{Proof of \texorpdfstring{$\epsilon$}{epsilon}--regularity}\label{sec:proof-main}

To lighten the notation in this subsection, let us set 
\begin{align*}
    \HH_R(\rho_0, \rho_1, c) &:= R^{2\alpha} \left( [\rho_0]_{\alpha, R}^2 + [\rho_1]_{\alpha, R}^2 + \left[\nabla_{xy}c\right]_{\alpha, R}^2\right).
\end{align*}
We are now in the position to give the proof of our main $\epsilon$-regularity theorem, which we restate for the reader's convenience:
\newtheorem*{epsilon-regularity}{Theorem \ref{thm:main}}
\begin{epsilon-regularity}
Assume that \ref{item:cost-cont}--\ref{item:cost-non-deg} hold and that $\rho_0(0) = \rho_1(0) =1$, as well as $\nabla_{xy}c(0,0)= -\II$. Assume further that $0$ is in the interior of $\, X \times Y$.

Let $\pi$ be a $c$-optimal coupling from $\rho_0$ to $\rho_1$. 
There exists $R_0= R_0(c)>0$ such that for all $R\leq R_0$ with 
    \begin{align}\label{eq:one-sided-assumption-main}
        \frac{1}{R^{d+2}} \int_{B_{4R} \times \RR^d} |y-x|^2 \, \dd\pi
        + \HH_{4R}(\rho_0, \rho_1, c) \ll_c 1,
    \end{align}
    there exists a function $T\in \mathcal{C}^{1,\alpha}(B_R)$ such that $\supp \pi \cap (B_R \times \RR^d) \subseteq \mathrm{graph}\,T$, 
    and the estimate 
    \begin{align}\label{eq:holder-gradient}
        [\nabla T]_{\alpha, R}^2 \lesssim \frac{1}{R^{2\alpha}} \left( \frac{1}{R^{d+2}} \int_{B_{4R} \times \RR^d} |y-x|^2 \, \dd\pi + \HH_{4R}(\rho_0, \rho_1, c)\right)
    \end{align}
    holds.
\end{epsilon-regularity}

\begin{proof}[Proof of Theorem~\ref{thm:main}] To simplify notation, we write $\EE_R$ for $\EE_R(\pi)$ and $\HH_R$ for $\HH_{R}(\rho_0, \rho_1, c)$. Note that since $\rho_0(0)=\rho_1(0)=1$ and $R^{\alpha}\left([\rho_0]_{\alpha, 4R}+[\rho_1]_{\alpha, 4R}\right)$ is small by \eqref{eq:one-sided-assumption-main}, we may assume throughout the proof that $\frac{1}{2}\leq \rho_0,\rho_1 \leq 2$ on $B_{4R}$.
\begin{enumerate}[label=\textsc{Step \arabic*},leftmargin=0pt,labelsep=*,itemindent=*]
\item \label{step:one-sided-to-two-sided} (Control of the full energy at scale $2R$). We show that under assumption \eqref{eq:one-sided-assumption-main} we can bound
\begin{align}\label{eq:two-sided-energy-smallness}
    \EE_{2R} \lesssim \frac{1}{R^{d+2}} \int_{B_{4R} \times \RR^d} |y-x|^2 \, \dd\pi.
\end{align}

Indeed, from Remarks \ref{rem:one-sided-energy} and \ref{rem:one-sided-qualitative}, we know that \eqref{eq:one-sided-assumption-main} implies the $L^{\infty}$ bound 
\begin{align}\label{eq:one-sided-Linfty}
\begin{split}
    (x,y) \in \, &(B_{3R}\times \RR^d) \cap \supp\pi \\
    &\Rightarrow \quad |x-y| \lesssim R \left( \frac{1}{R^{d+2}} \int_{B_{4R} \times \RR^d} |y-x|^2 \, \dd\pi + \DD_{4R}(\rho_0,\rho_1) \right)^{\frac{1}{d+2}}.
\end{split}
\end{align} Let us now prove that
\begin{align}\label{eq:one-sided-to-two-sided}
    (\RR^d \times B_{2R}) \cap \supp\pi \subseteq B_{4R} \times B_{2R}, 
\end{align}
from which we get 
\begin{align*}
    \int_{\RR^d \times B_{2R}} |y-x|^2 \, \dd\pi \leq \int_{B_{4R}\times \RR^d} |y-x|^2 \, \dd\pi,
\end{align*} thus yielding \eqref{eq:two-sided-energy-smallness}. 
To this end, assume there exists $(x,y) \in (B_{4R}^c \times B_{2R}) \cap \supp\pi$. Let also $x' \in [x,y] \cap \partial B_{\frac{5}{2}R}$ and $y' \in \RR^d$ such that $(x',y') \in \supp\pi$, see Figure \ref{fig:Linfty}. Then by \eqref{eq:one-sided-Linfty}, \eqref{eq:one-sided-assumption-main} and \eqref{eq:D-bound-rho}, we have 
\begin{align}\label{eq:main-proof-one-sided-step}
    |x'-y'| \ll R. 
\end{align}
From \eqref{eq:Linfty-step1-cmon-2}, recalling the definition \eqref{eq:delta-def} of $\delta$ and the fact that $\delta \ll 1$, we get, upon writing $y-y'=y-x'+x'-y'$, 
\begin{align*}
    0 &\leq (x-x') \cdot (y-x') + (x-x')\cdot (x'-y') + \delta |x-x'||y-x'| + \delta |x-x'||x'-y'| \\
    &\leq (\delta-1)|x-x'||y-x'| + (1+\delta)|x-x'||x'-y'| \\
    &\leq |x-x'| \left(-\frac{1}{2}\frac{R}{2} + (1+\delta) |x'-y'|\right), 
\end{align*} which together with \eqref{eq:main-proof-one-sided-step} yields a contradiction, proving \eqref{eq:one-sided-to-two-sided}. 
\begin{figure}[ht]
\begin{tikzpicture}[line cap=round,line join=round,>=triangle 45,x=1.0cm,y=1.0cm]
\clip(-1.86634036085822,0.4799433083793329) rectangle (4.637400464773399,4.2);
\draw [line width=1.pt] (0.,0.) circle (2.cm);
\draw [line width=1.pt] (0.,0.) circle (4.cm);
\draw [line width=1.pt] (0.20292602528856288,1.1155646844514124)-- (3.7999795164275345,2.6373950076255994);
\draw [line width=1.pt] (0.,0.) circle (2.5cm);
\begin{scriptsize}
\draw[color=black] (-0.821365761844323,1.5) node {\small $B_{2R}$};
\draw[color=black] (-1.0,3.5) node {\small $B_{4R}$};
\draw [fill=black] (0.20292602528856288,1.1155646844514124) circle (1.5pt);
\draw[color=black] (0.2160638942164051,1.35) node {$y$};
\draw [fill=black] (3.7999795164275345,2.6373950076255994) circle (1.5pt);
\draw[color=black] (3.8074567399248167,2.85) node {$x$};
\draw[color=black] (-0.9420848490950258,2.6) node {\small $B_{\frac{5}{2}R}$};
\draw [fill=black] (1.7608292129297964,1.774677571530399) circle (1.5pt);
\draw[color=black] (1.883496286866739,2.0) node {$x'$};
\draw [fill=black] (0.7447898524793695,2.9717365180199287) circle (1.5pt);
\draw[color=black] (0.8649289881889333,3.25) node {$y'$};
\end{scriptsize}
\end{tikzpicture}
\caption{\label{fig:Linfty}The definition of $x'$ and $y'$ in the proof of \eqref{eq:one-sided-to-two-sided}.}
\end{figure}
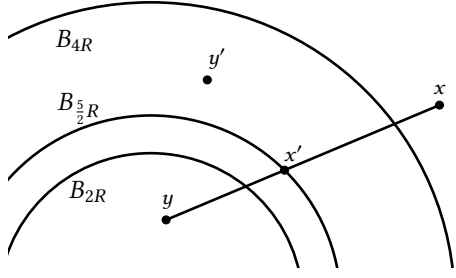

In the following, \ref{step:normalized-setting} -- \ref{step:campanato} are devoted to prove that under the assumption 
\begin{align}\label{eq:smallness-assumption-proof-main}
    \EE_{2R} + \HH_{2R} \ll 1,
\end{align} the following Campanato estimate holds: 
\begin{align}\label{eq:campanato-pi-full}
    \underset{0<r<\frac{R}{2}}{\sup} \ \underset{x_0 \in B_{R}}{\sup} \ \frac{1}{r^{d+2+2\alpha}} \ \underset{A,b}{\inf} \int_{(B_r(x_0)\cap B_R) \times \RR^d} |y-(Ax+b)|^2 \, \dd\pi \lesssim \frac{1}{R^{2\alpha}} \left( \EE_{2R}+\HH_{2R} \right). 
\end{align} 

\medskip
\item \label{step:normalized-setting} (Getting to a normalized setting). We claim that it is enough to prove that if 
\begin{align}\label{eq:local-smallness}
    \EE_R + \HH_R \ll 1
\end{align}
then for all $r \leq \frac{R}{2}$, 
\begin{align}\label{eq:campanato}
    \frac{1}{r^{d+2}} \ \underset{A,b}{\inf} \int_{B_{r}\times \RR^d} |y-(Ax+b)|^2 \, \dd \pi \lesssim \frac{r^{2\alpha}}{R^{2\alpha}} \left( \EE_{R} + \HH_{R} \right).
\end{align}

Let us first notice that for any $x_0 \in B_{R}$, $(B_r(x_0) \cap B_R) \times \RR^d \subseteq \cross{2R}$ so that 
\begin{align*}
    \EE_{x_0, R} := \frac{1}{R^{d+2}} \int_{(B_r(x_0) \cap B_R) \times \RR^d} |y-x|^2 \, \dd \pi  &\lesssim \EE_{2R}, \\ 
    \text{and} \quad \HH_{x_0,R} := R^{2\alpha} \left( [\rho_0]_{\alpha, B_R(x_0)}^2 + [\rho_1]_{\alpha,R}^2 + \left[\nabla_{xy}c\right]_{\alpha, (B_r(x_0) \cap B_R) \times \RR^d}^2 \right) &\lesssim \HH_{2R}.
\end{align*}
Therefore, in view of \eqref{eq:smallness-assumption-proof-main}, it is sufficient to show for all $x_0 \in B_{R}$ that, if 
\begin{align}\label{eq:x0-smallness}
    \EE_{x_0,R} + \HH_{x_0,R} \ll 1,
\end{align}
then for all $r \leq \frac{R}{2}$,
\begin{align}\label{eq:campanato-x0}
    \frac{1}{r^{d+2}} \ \underset{A,b}{\inf} \int_{(B_r(x_0) \cap B_R) \times \RR^d} |y-(Ax+b)|^2 \, \dd \pi \lesssim \frac{r^{2\alpha}}{R^{2\alpha}} \left( \EE_{x_0,R} + \HH_{x_0,R} \right).
\end{align}

Performing a similar change of coordinates as Lemma \ref{lem:coordinate-change}, namely letting $M := -\nabla_{xy}c(x_0,0)$, $\gamma :=  (\rho_0(x_0) \det M)^{\frac{1}{d}}$ and $S(y) := \gamma M^{-1} y$ and defining
\begin{align*}
    \widetilde{\pi} &:= \frac{1}{\rho_0(x_0)} (\id \times \,S^{-1}) _{\#} \pi, \quad \widetilde{\rho}_0 := \frac{\rho_0}{\rho_0(x_0)}, \quad
    \widetilde{\rho}_1 := \rho_1 \circ S, \quad
    \widetilde{c}(x,y) := \, \gamma^{-1} c(x, S(y)),
\end{align*} 
we have that $\widetilde{\rho}_0(x_0) = \widetilde{\rho}_1(0) = 1$ and $\nabla_{xy}\widetilde{c}(x_0,0) = -\II$. Furthermore, $\widetilde{\pi}$ is a  $\widetilde{c}$-optimal coupling between $\widetilde{\rho}_0$ and $\widetilde{\rho}_1$ and $\widetilde{\pi}$, $\widetilde{\rho}_0$,  $\widetilde{\rho}_1$ and $\widetilde{c}$ still satisfy \eqref{eq:x0-smallness}. Without loss of generality, we may thus assume $x_0=0$ and then \eqref{eq:x0-smallness} turns into \eqref{eq:local-smallness} and \eqref{eq:campanato-x0} turns into \eqref{eq:campanato}.

\medskip
\item (First step of the iteration). Recall that \eqref{eq:D-bound-rho} holds, i.e., $\DD_{R} \lesssim R^{2\alpha}\left( [\rho_0]_{\alpha,R}^2 + [\rho_1]_{\alpha,R}^2 \right)$. By (a rescaling of) Lemma \ref{lem:displacement-qualitative}, there exist $\Lambda_0 < \infty$ and $R_0>0$ such that for all $R\leq R_0$ for which \eqref{eq:local-smallness} holds, we have 
\begin{align}\label{eq:iteration-first-inclusion}
    \cross{\frac{8}{9}R} \cap \supp\pi \subseteq \BB_{\frac{8}{9}R, \frac{\Lambda_0}{9} R}.
\end{align}
Let $\beta>0$. In view of \eqref{eq:local-smallness} and \eqref{eq:iteration-first-inclusion}, we may apply Proposition \ref{prop:one-step-main} to get the existence of a symmetric matrix $B_1$ and a vector $b_1$ such that, defining $D_1 := -\nabla_{xy}c_1(0, b_1)$, $\gamma_1 := \left(\rho_1(b_1) \det B_1^2D_1^{-1}\right)^{\frac{1}{d}}$ and $\rho_{0,1}$, $\rho_{1,1}$, $c_1$ and $\pi_1$ as $\widehat{\rho}_0$, $\widehat{\rho}_1$, $\widehat{c}$ and $\widehat{\pi}$ in Lemma \ref{lem:coordinate-change}, we get 
\begin{align}\label{eq:E1-est}
    \mathcal{E}^1 := \EE_{\theta R}(\pi_1) \leq \theta^{2\beta} \mathcal{E}_R + C_{\beta} \HH_R, 
\end{align}
and that $\pi_1$ is a $c_1$-optimal coupling between $\rho_{0,1}$ and $\rho_{1,1}$. From \eqref{eq:qualitative-theta-inclusion-main} we also have the inclusion
\begin{align}\label{eq:iteration-second-inclusion}
    \cross{\frac{8}{9} \theta R} \cap \supp\pi_1 \subseteq \BB_{\frac{8}{9} \theta R,3\theta R},
\end{align}
so that we may fix $\Lambda=27$ from now on (assuming w.l.o.g. that $\Lambda_0\geq 27$).

\medskip
\item (Iterating Proposition \ref{prop:one-step-main}). We now show that we can iterate Proposition \ref{prop:one-step-main} a finite number of times. 

Notice that from the estimates \eqref{eq:B-b-est-main}, \eqref{eq:M-est-2} and \eqref{eq:gamma-est} we have the inclusion 
\begin{align}\label{eq:Q-inclusion}
\gamma_1^{-1} D_1^{-*} B_1 B_{\theta R} + b_1 \subseteq B_R.
\end{align} 
Hence, we can bound
\begin{align*}
    [\rho_{1,1}]_{\alpha, \theta R} &= \rho_1(b_1)^{-1} \underset{y \neq y' \in B_{\theta R}}{\sup} \ \frac{|\rho_1( \gamma_1^{-1} D_1^{-*} B_1 y + b_1) - \rho_1(\gamma_1^{-1} D_1^{-*} B_1 y' + b_1)|}{|y-y'|^{\alpha}} \\
    &\overset{\eqref{eq:delta-lambda-est}}{\lesssim} |\gamma_1^{-1} D_1^{-*} B_1|^{\alpha} \underset{y \neq y' \in B_{\theta R}}{\sup} \ \frac{|\rho_1( \gamma_1^{-1} D_1^{-*} B_1 y + b_1) - \rho_1(\gamma_1^{-1} D_1^{-*} B_1 y' + b_1)|}{|(\gamma_1^{-1} D_1^{-*} B_1 y + b_1) - (\gamma_1^{-1} D_1^{-*} B_1 y' + b_1)|^{\alpha}} \\
    &\stackrel{\eqref{eq:Q-inclusion}}{\leq} \gamma_1^{-\alpha} |D_1^{-1}|^{\alpha} |B_1|^{\alpha} \underset{\widetilde{y} \neq \widetilde{y}' \in B_{R}}{\sup} \ \frac{|\rho_1(\widetilde{y}) - \rho_1(\widetilde{y}')|}{|\widetilde{y} - \widetilde{y}'|^{\alpha}}. 
\end{align*}
Estimates \eqref{eq:B-b-est-main}, \eqref{eq:delta-lambda-est}, \eqref{eq:M-est-2}, and \eqref{eq:gamma-est} thus yield for $i \in \{0,1\}$ ($\rho_{0,1}$ is treated similarly) 
\begin{align*}
    [\rho_{i,1}]_{\alpha, \theta R}^2 \leq \left(1+C(\EE_R^{\frac{1}{2}} + \HH_R^{\frac{1}{2}})\right) [\rho_i]_{\alpha, R}^2, 
\end{align*}
where $C$ is a constant depending only on $d$ and $\alpha$. Using \eqref{eq:mixed-derivative}, the same kind of computation gives 
\begin{align}\label{eq:estimate-seminorm-c}
    \left[\nabla_{xy}c_1\right]_{\alpha, \theta R}^{2} \leq \left(1+C(\EE_R^{\frac{1}{2}} + \HH_R^{\frac{1}{2}})\right) \left[\nabla_{xy}c\right]_{\alpha, R}^{2}, 
\end{align}
so that 
\begin{align}\label{eq:H1-est}
    \HH^1 := \HH_{\theta R}(\rho_{0,1}, \rho_{1,1}, c_1) \leq \theta^{2\alpha} \left(1+C(\EE_R^{\frac{1}{2}} + \HH_R^{\frac{1}{2}})\right) \HH_R.
\end{align}
Therefore, \eqref{eq:local-smallness}, \eqref{eq:E1-est}, \eqref{eq:iteration-second-inclusion}, and \eqref{eq:H1-est} imply that we may iterate Proposition \ref{prop:one-step-main} $K$ times, $K\geq 2$, to find sequences of matrices $B_k$ and $D_k$, a sequence of vectors $b_k$, and a sequence of real numbers $\gamma_k$ such that, setting for $1 \leq k \leq K$  
\begin{align}
    \rho_{0,k}(x) &:= \rho_{0,k-1}(B_k^{-1} x), \quad  \rho_{1,k}(y) := \frac{\rho_{1,k-1}(\gamma_k^{-1} D_k^{-*} B_k y + b_k)}{\rho_{1,k-1}(b_k)}, \nonumber\\ 
    Q_k(x,y) &:= (B_k x, \gamma_k B_k^{-1} D_k^* (y-b_k)), \quad \pi_k := \det B_k \, {Q_k}_{\#} \pi_{k-1},  \label{eq:Tk-def} \\
    \text{and} \quad c_k(x, y) &:= \gamma_k c_{k-1}(B_k^{-1} x,\gamma_k^{-1} D_k^{-*} B_k y + b_k), 
\end{align}
we recover $\rho_{0,k}(0) = \rho_{1,k}(0) = 1$ and $\nabla_{xy}c_k(0,0) = -\II$, $\pi_k$ is a $c_k$-optimal coupling between $\rho_{0,k}$ and $\rho_{1,k}$, and from \eqref{eq:qualitative-theta-inclusion-main} we have 
\begin{align}\label{eq:iteration-k-inclusion}
    \cross{\frac{8}{9}\theta^kR} \cap \supp\pi_k \subseteq \BB_{\frac{8}{9}\theta^kR,3\theta^kR}. 
\end{align} Moreover, defining 
\begin{align}\label{eq:Hk-def}
     \EE^k := \EE_{\theta^k R}(\pi_k) \quad \text{and} \quad \HH^{k} := \HH_{\theta^{k} R}(\rho_{0,k}, \rho_{1,k}, c_{k}), 
\end{align} we have 
\begin{align}
    \mathcal{E}^k &\overset{\eqref{eq:one-step-conclusion-main}}{\leq} \theta^{2\beta} \EE^{k-1} + C_{\beta} \HH^{k-1}, \label{eq:Ek} \\
    |B_k-\II|^2 + \frac{1}{(\theta^{k-1}R)^2} |b_k|^2 &\overset{\eqref{eq:B-b-est-main}}{\lesssim} \mathcal{E}^{k-1} + \theta^{2(k-1)\alpha} R^{2\alpha} \left( [\rho_{0,k-1}]_{\alpha, \theta^{k-1}R}^2 + [\rho_{1,k-1}]_{\alpha, \theta^{k-1}R}^2 \right), \label{eq:Bk-est} \\
    |D_k-\II|^2 &\overset{\eqref{eq:M-est-2}}{\lesssim} \theta^{2(k-1)\alpha} R^{2\alpha} [\nabla_{xy}c_{k-1}]_{\alpha, \theta^{k-1}R}^2,  \label{eq:Mk}\\
    |\gamma_k-1|^2 &\overset{\eqref{eq:gamma-est}}{\lesssim} \HH^{k-1}. \label{eq:gammak}
\end{align} Arguing as for \eqref{eq:H1-est}, there exists a constant $C_1 = C_1(d, \alpha) < \infty$ such that 
\begin{align}\label{eq:rhok}
    \HH^k \leq \theta^{2\alpha} \left( 1+C_1\left((\EE^{k-1})^{\frac{1}{2}} + (\HH^{k-1})^{\frac{1}{2}}\right)\right) \HH^{k-1}. 
\end{align}

\medskip
\item (Smallness at any step of the iteration). 
From now on, we fix $\beta > \alpha$. Let us show by induction that there exists a constant $C_2 = C_2(d, \alpha, \beta)<\infty$ such that for all $K \in \mathbb{N}$  
\begin{align}
    \EE^K &\leq C_2 \theta^{2K\alpha} \left( \EE_R + \HH_R \right), \label{eq:indhypE} \\
    \HH^K &\leq \theta^{2\alpha}(1+\theta^{K\alpha}) \HH^{K-1}. \label{eq:indhypH}
\end{align}
This will show, together with \eqref{eq:iteration-k-inclusion}, that we can keep on iterating Proposition \ref{prop:one-step-main}. As an outcome of this induction, the estimate 
\begin{align}\label{eq:H-estimate}
    \HH^K \lesssim C_3 \theta^{2K\alpha} \HH_R
\end{align} 
holds for all $K\in\NN$ with $C_3 := \prod_{k=1}^{+\infty} (1+\theta^{k\alpha}) < \infty$. 

We set 
\begin{align}\label{eq:C2-def}
    C_2 := \max \left\{\theta^{2(\beta-\alpha)}, \frac{C_3C_{\beta}\theta^{-2\alpha}}{1-\theta^{2(\beta-\alpha)}}\right\}.
\end{align} 
By \eqref{eq:E1-est} and \eqref{eq:H1-est}, the estimates \eqref{eq:indhypE} and \eqref{eq:indhypH} hold for $K=1$ provided $\EE_R +\HH_R$ is small enough, since $C_2 \geq \max \{\theta^{2(\beta-\alpha)}, C_{\beta}\theta^{-2\alpha}\}$. Assume now that \eqref{eq:indhypE} and \eqref{eq:indhypH} hold for all $1 \leq k \leq K-1$. By induction hypothesis, we have 
\begin{align*}
    \EE^{K-1} \leq C_2 \theta^{2(K-1)\alpha} \left( \EE_R + \HH_R \right) \quad \text{and} \quad \HH^{K-1} \leq C_3 \theta^{2(K-1)\alpha} \HH_R.
\end{align*}
Combining these two estimates with \eqref{eq:rhok} for $k=K$ and assuming that 
\begin{align*}
    C_1 \left( C_2^{\frac{1}{2}} (\EE_R^{\frac{1}{2}}+\HH_R^{\frac{1}{2}}) + C_3^{\frac{1}{2}} \HH_R^{\frac{1}{2}} \right) \leq \theta^{\alpha},
\end{align*} which is possible provided $\EE_R+\HH_R$ is small enough, we get \eqref{eq:indhypH}. 
Similarly, by \eqref{eq:Ek} with $k=K$ we obtain, using that $C_2 \geq \frac{C_3C_{\beta}\theta^{-2\alpha}}{1-\theta^{2(\beta-\alpha)}}$ by \eqref{eq:C2-def}, the bound
\begin{align*}
    \theta^{-2K\alpha} \EE^K &\leq \left( \theta^{2(\beta-\alpha)} C_2 + C_{\beta}C_3\theta^{-2\alpha} \right) ( \EE_R + \HH_R ) \leq C_2 ( \EE_R + \HH_R ).
\end{align*}
This concludes the induction. 

\medskip
\item \label{step:campanato} (Campanato estimate). We can now prove the main estimate, that is, assuming \eqref{eq:local-smallness}, we show that \eqref{eq:campanato} holds. 

Let 
\begin{align}
    \overline{\gamma}_k := \gamma_k \hdots \gamma_1, \quad \overline{B}_k &:= B_k \hdots B_1, \quad \overline{D}_k := B_k^{-1} D_k^* \hdots B_1^{-1} D_1^*, \quad
    \overline{b}_k := \sum_{i=1}^k (\gamma_k B_k^{-1} D_k^*) \hdots (\gamma_i B_i^{-1} D_i^*)b_i \nonumber
\end{align}
and  
\begin{align}
    \overline{Q}_k(x,y) &:= (\overline{B}_k x, \overline{\gamma}_k \overline{D}_k y - \overline{b}_k). \label{eq:Qbar-def}
\end{align} We see that, recalling \eqref{eq:Tk-def} and noticing that $\overline{Q}_k = Q_k \circ \cdots \circ Q_1$,
\begin{align}\label{eq:Tk-wrt-Ak}
    \pi_k = \det \overline{B}_k {\overline{Q}_k}_{\#} \pi \quad \text{and} \quad \rho_{0,k}(x) = \rho_0(\overline{B}_k^{\kern.2ex-1}x).
\end{align} 
Moreover, from \eqref{eq:Bk-est}, \eqref{eq:indhypE}, and \eqref{eq:indhypH}, we have the estimate 
\begin{align}\label{eq:Ak-Dk}
    |\overline{B}_k-\II|^2 \lesssim \EE_R + \HH_R \ll 1,
\end{align}
so that 
\begin{align}\label{eq:inclusion}
    B_{\frac{1}{2}\theta^kR} \times \RR^d \subseteq \overline{B}_k^{\kern.2ex-1}(B_{\theta^kR}) \times \RR^d = \overline{Q}_k^{\kern.2ex-1}(B_{\theta^kR} \times \RR^d).
\end{align} 
Similarly, from \eqref{eq:Bk-est}, \eqref{eq:Mk} and \eqref{eq:gammak}, 
\begin{align}\label{eq:Lambdak}
    |\overline{\gamma}_k-1|^2 + |\overline{D}_k-\II|^2 \lesssim \EE_R + \HH_R \ll 1.
\end{align}
Let us now compute 
\begin{align*}
    \frac{1}{\left(\frac{1}{2} \theta^kR\right)^{d+2}} \, &\inf_{A,b} \int_{B_{\frac{1}{2} \theta^kR}\times \RR^d} |y-(Ax+b)|^2 \, \dd \pi \\
    &\overset{\eqref{eq:inclusion}}{\lesssim} \frac{1}{(\theta^kR)^{d+2}} \int_{\overline{Q}_k^{\kern.2ex-1}(B_{\theta^kR}\times \RR^d)} |y-\overline{\gamma}_k^{\kern.2ex-1}\overline{D}_k^{\kern.2ex-1}\overline{B}_k x - \overline{\gamma}_k^{\kern.2ex-1}\overline{D}_k^{\kern.2ex-1}\overline{b}_k|^2 \, \dd \pi\\
    &\overset{\eqref{eq:Tk-wrt-Ak}}{=} \frac{(\det \overline{B}_k)^{-1}}{(\theta^kR)^{d+2}} \int_{B_{\theta^kR}\times \RR^d} |\overline{\gamma}_k^{\kern.2ex-1}\overline{D}_k^{\kern.2ex-1} (y-x)|^2 \, \dd \pi_k \\
    &\overset{\eqref{eq:Ak-Dk} \& \eqref{eq:Lambdak}}{\lesssim} \frac{1}{(\theta^kR)^{d+2}} \int_{B_{\theta^kR}\times \RR^d} |y-x|^2 \, \dd \pi_k \overset{\eqref{eq:Hk-def}}{=} \mathcal{E}^k.
\end{align*}
By \eqref{eq:indhypE}, we obtain
\begin{align*}
    \frac{1}{\left(\frac{1}{2} \theta^kR\right)^{d+2}} \, \underset{A,b}{\inf} \int_{B_{\frac{1}{2} \theta^kR}\times \RR^d} |y-(Ax+b)|^2 \, \dd \pi \lesssim \theta^{2k\alpha} \left( \EE_R + \HH_R \right),
\end{align*}
from which \eqref{eq:campanato} follows, concluding the proof of \eqref{eq:campanato-pi-full}

\medskip
\item\label{item:supp-T} ($\supp\pi$ is contained in the graph of a function $T$ within $B_R \times \RR^d$).
We claim that \eqref{eq:campanato-pi-full} implies the existence of a function $T:B_R \to Y$ such that 
\begin{align}\label{eq:support-pi-T}
     (B_R \times \RR^d) \cap \supp \pi \subseteq \mathrm{graph\, T}.
\end{align}

In the following, we abbreviate 
\begin{align*}
    \normpi_{\alpha}^2 &:= \sup_{0<r<\frac{R}{2}} \sup_{x_0 \in B_{R}} \frac{1}{r^{d+2+2\alpha}} \inf_{A,b} \int_{(B_r(x_0)\cap B_R)\times \RR^d} |y-(Ax+b)|^2 \,\dd\pi.
\end{align*} To prove the claim, fix $x_0 \in B_R$ and notice that \eqref{eq:campanato-pi-full} implies that for any $r>0$ small enough, there holds 
\begin{align}\label{eq:campanato-pi-2}
    \frac{1}{r^{d+2}} \inf_{A,b} \int_{B_r(x_0) \times \RR^d} |y-(Ax+b)|^2 \,\dd\pi \lesssim \frac{r^{2\alpha}}{R^{2\alpha}} \left(  \frac{1}{R^{d+2}} \int_{B_{4R}\times \RR^d} |x-y|^2\,\dd \pi +\HH_{4R} \right).
\end{align}
\begin{enumerate}[label=\textsc{Step 7.\Alph*.},labelsep=*,itemindent=*]
\item It is easy to see that the infimum in \eqref{eq:campanato-pi-2} is attained at some $A_r = A_r(x_0)$ and $b_r = b_r(x_0)$. Analogous to \cite[Lemma 3.IV]{Cam64} one can show that there exist a matrix $A_0 = A_0(x_0)$ and a vector $b_0 = b_0(x_0)$ such that $A_r \to A_0$ and $b_r \to b_0$ as $r\to 0$ (uniformly in $x_0$) with rates
\begin{align}\label{eq:A-b-conv-rates}
    |A_r - A_0| \lesssim \normpi_{\alpha} r^{\alpha} \quad \text{and} \quad 
    |b_r - b_0| \lesssim \normpi_{\alpha} r^{\alpha+1}.
\end{align}
We refer the reader to Appendix \ref{app:campanato} for a proof of the convergences and \eqref{eq:A-b-conv-rates}.

\item We claim that 
\begin{align}\label{eq:T-derivation-limit}
    \frac{1}{r^d} \int_{B_r(x_0)\times \RR^d} |y- (A_0 x_0 + b_0)|^2\,\dd\pi \to 0 \quad \text{as } r \searrow 0.
\end{align}
Indeed, we can split
\begin{align}
    \frac{1}{r^d} \int_{B_r(x_0)\times \RR^d} |y- (A_0 x_0 + b_0)|^2\,\dd\pi 
    &\lesssim \frac{1}{r^d} \int_{B_r(x_0)\times \RR^d} |y- (A_r x + b_r)|^2\,\dd\pi \nonumber \\
    &\quad+ \frac{1}{r^d} \int_{B_r(x_0)\times \RR^d} |(A_r-A_0) x + (b_r- b_0)|^2\,\dd\pi \label{eq:T-derivation-bound} \\
    &\quad+ \frac{1}{r^d} \int_{B_r(x_0)\times \RR^d} |A_0 (x-x_0)|^2\,\dd\pi \nonumber.
\end{align}
By definition of $A_r, b_r$, we have
\begin{align*}
    \frac{1}{r^d} \int_{B_r(x_0)\times \RR^d} |y- (A_r x + b_r)|^2\,\dd\pi 
    = \frac{1}{r^d} \inf_{A,b} \int_{B_r(x_0)\times \RR^d} |y- (A x + b)|^2\,\dd\pi 
    \lesssim \normpi_{\alpha}^2 r^{2\alpha+2}.
\end{align*}
Using \eqref{eq:A-b-conv-rates}, $\rho_0\leq 2$, and $x_0\in B_R$, it follows that 
\begin{align*}
    \frac{1}{r^d} \int_{B_r(x_0)\times \RR^d} |(A_r-A_0) x + (b_r- b_0)|^2\,\dd\pi
    &= \frac{1}{r^d} \int_{B_r(x_0)} |(A_r-A_0) x + (b_r- b_0)|^2\,\rho_0(x)\dd x\\
    &\lesssim |A_r-A_0|^2 R^2 + |b_r - b_0|^2 \\
    &\lesssim \normpi_{\alpha}^2 r^{2\alpha} (R^2 + r^2).
\end{align*}
Finally, the last term in \eqref{eq:T-derivation-bound} is estimated by
\begin{align*}
    \frac{1}{r^d} \int_{B_r(x_0)\times \RR^d} |A_0 (x-x_0)|^2\,\dd\pi 
    &= \frac{1}{r^d} \int_{B_r(x_0)} |A_0 (x-x_0)|^2\,\rho_0(x)\dd x \lesssim r^2.
\end{align*}
Letting $r\to 0$ in the above estimates proves the claim \eqref{eq:T-derivation-limit}.

\item By disintegration, there exists a family of measures $\{\pi_x\}_{x\in X}$ on $Y$ such that 
\begin{align}\label{eq:disintegration}
    \frac{1}{r^d} \int_{B_r(x_0)\times \RR^d} |y- (A_0 x_0 + b_0)|^2\,\dd\pi 
    = \frac{1}{r^d} \int_{B_r(x_0)} \int |y- (A_0 x_0 + b_0)|^2\,\pi_x(\dd y) \, \rho_0(x)\dd x.
\end{align}
Since the left-hand side of \eqref{eq:disintegration} tends to zero as $r\to 0$ by Step 7.B, it follows that if $x_0$ is a Lebesgue point, we must have 
\begin{align*}
    \int |y- (A_0 x_0 + b_0)|^2\,\pi_{x_0}(\dd y) = 0,
\end{align*}
hence $\pi_{x_0} = \delta_{A_0 x_0 + b_0}$. 

\item For any Lebesgue point $x_0 \in B_R$, define $T(x_0):= A_0(x_0) x_0 + b_0(x_0)$. Then the previous Step 7.C shows that
\begin{align*}
    \pi\big\lfloor_{B_R \times \RR^d} = (\mathrm{Id} \times T)_{\#}\rho_0,
\end{align*}
that is \eqref{eq:support-pi-T}.
\end{enumerate}

\medskip
\item \label{step:hoelder-regularity} By boundedness of $\rho_0$, \eqref{eq:campanato-pi-full} implies the bound
\begin{align*}
        \sup_{0<r<\frac{R}{2}} \sup_{x_0 \in B_{R}}  \frac{1}{r^{d+2+2\alpha}} \inf_{A,b} &\int_{B_r(x_0)\cap B_R} |T(x)-(Ax+b)|^2 \,\dd x \\
        &\lesssim \frac{1}{R^{2\alpha}} \left(  \frac{1}{R^{d+2}} \int_{B_{4R}\times \RR^d} |x-y|^2\,\dd \pi +\HH_{4R} \right),
\end{align*}
which by means of Campanato's theory \cite{Cam64} proves that $T\in \mathcal{C}^{1,\alpha}(B_R)$ and that the Hölder seminorm of $\nabla T$ satisfies \eqref{eq:holder-gradient}.
\hfill \qedhere
\end{enumerate}
\end{proof}

\begin{remark}
    The deterministic structure of the $c$-optimal coupling, that is, the existence of $T$ such that $\pi = (\mathrm{Id}\times T)_{\#}\rho_0$, is a classical result in optimal transportation. If we had used this result, the proof would have become shorter, as \ref{item:supp-T} would not have been needed. 
\end{remark}

Before we give the proof of Corollary \ref{cor:DF}, let us remark that one can show the following variant of our qualitative $L^{\infty}$ bound on the displacement (Lemma \ref{lem:displacement-qualitative}): 
\begin{lemma}\label{lem:qualitative-2}
    Assume that the cost function $c$ satisfies \ref{item:cost-cont}--\ref{item:cost-non-deg} and that $\nabla_x c(0,0)=0$. Let $u$ be a $c$-convex function. There exist $\Lambda_0 < \infty$ and $R_0'>0$ such that for all $R\leq R_0'$ for which $\frac{1}{R^2} \left\| u - \tfrac{1}{2}|\cdot|^2\right\|_{\mathcal{C}^0(B_{8R})} \leq 1$ we have 
    \begin{align}\label{eq:qualitative-u}
        \esssup_{x\in B_{4R}} \left|\cexp_x(\nabla u(x)) \right| \leq \Lambda_0 R.
    \end{align}
\end{lemma}

\begin{proof}
    Since $u$ is $c$-convex, it is differentiable a.e. For any $x \in B_{4R}$ such that $\nabla u(x)$ exists, let $y = \cexp_x(\nabla u(x))$, that is,
    \begin{align}\label{eq:y-c-exp-def}
         \nabla u(x) + \nabla_x c(x,y) = 0.
    \end{align}
    Let $\widetilde{c}$ be defined as in \eqref{eq:ctilde-def}. Then, using $\nabla_x c(0,0) = 0$, we have 
    \begin{align*}
        -\nabla_x\widetilde{c}(x,y) 
        &= -\nabla_x c(x,y) + \nabla_x c(x,0) 
        \stackrel{\eqref{eq:y-c-exp-def}}{=} \nabla u(x) +  \nabla_x c(x,0) - \nabla_x c(0,0) \\
        &= \nabla u(x) -x + x +  \int_{0}^1 \nabla_{xx} c(t x,0) \,\dd t x, 
    \end{align*}
    so that 
    \begin{align*}
        |\nabla_x\widetilde{c}(x,y)| \leq |\nabla u(x) -x| + |x| +  \|\nabla_{xx} c\|_{\mathcal{C}^0(X\times Y)} |x|.
    \end{align*}
    Being $c$-convex, the function $u$, and therefore also the function $x\mapsto u(x) - \frac{1}{2}|x|^2$, is semi-convex, which implies that 
    \begin{align}
        \esssup_{x\in B_{4R}} |\nabla u(x) - x| \lesssim \frac{1}{R} \sup_{x\in B_{8R}} |u-\tfrac{1}{2}|x|^2|,
    \end{align}
    see Lemma \ref{lem:semi-convex} in the appendix. 
    By the closeness assumption on $u$ and \ref{item:cost-cont} we may therefore bound 
    \begin{align*}
        |\nabla_x\widetilde{c}(x,y)| \leq \lambda R.
    \end{align*}
    Steps 2 and 3 of the proof of Lemma \ref{lem:displacement-qualitative} then imply that there exist $\Lambda_0<\infty$ and $R_0'>0$ (depending on $c$ only through assumptions \ref{item:cost-cont}--\ref{item:cost-non-deg}) such that for all $R \leq R_0'$ we have
    \begin{align*}
        |\nabla_x\widetilde{c}(x,y)| \leq \lambda R \quad \Rightarrow \quad |y| \leq \Lambda_0 R,
    \end{align*}
    that is, \eqref{eq:qualitative-u} holds.
\end{proof}

\begin{proof}[Proof of Corollary \ref{cor:DF}]
    By Lemma \ref{lem:qualitative-2} there exist $\Lambda_0<\infty$ and $R_0'>0$, depending only on the qualitative assumptions \ref{item:cost-cont}--\ref{item:cost-non-deg} on $c$ such that for all $R\leq R_0'$ for which \eqref{eq:smallness-main-corollary} holds, we have 
    \begin{align}\label{eq:Linfty-Tu}
        \|T_u\|_{L^{\infty}(B_{4R})} \leq \Lambda_0 R. 
    \end{align}
    We claim that 
    \begin{align}\label{eq:cor-df-Linfty-bound}
        \frac{1}{R} \left\| x - T_u \right\|_{L^{\infty}(B_{4R})} 
        \lesssim \frac{1}{R^2} \left\| u - \tfrac{1}{2}|\cdot|^2\right\|_{\mathcal{C}^0(B_{8R})}+ R^{\alpha} \left([\nabla_{xy}c]_{\alpha,4R} + [\nabla_{xx}c]_{\alpha,4R}\right),
    \end{align}
    which immediately implies that 
    \begin{align}\label{eq:cor-df-bound}
    \begin{split}
        \frac{1}{R^{d+2}} \int_{B_{4R}\times \RR^d} |x-y|^2\,\dd \pi 
        &=  \frac{1}{R^{d+2}} \int_{B_{4R}} |x-T_u(x)|^2\,\rho_0(x)\dd x \\
        &\lesssim \frac{1}{R^4}\left\| u - \tfrac{1}{2}|\cdot|^2\right\|_{\mathcal{C}^0(B_{8R})}^2 + R^{2\alpha} \left([\nabla_{xy}c]_{\alpha,4R}^2 + [\nabla_{xx}c]_{\alpha,4R}^2\right) \ll 1.
    \end{split}
    \end{align}
    In particular, it follows by Theorem \ref{thm:main}, that there exists a potentially smaller scale $R_0\leq R_0'$ such that for all $R\leq R_0$ for which \eqref{eq:smallness-main-corollary} holds, we have that $T_u \in \mathcal{C}^{1,\alpha}(B_R)$ and $\nabla T_u$ satisfies the bound \eqref{eq:holder-gradient-main}. Applying \eqref{eq:cor-df-bound} once more, we see that \eqref{eq:holder-gradient-main-corollary} holds. 
    
    To prove the claim \eqref{eq:cor-df-Linfty-bound}, we appeal to semi-convexity of the $c$-convex function $u$ (which implies semi-convexity of the function $x\mapsto u(x) - \frac{1}{2}|x|^2$), in particular Lemma \ref{lem:semi-convex}, to bound 
    \begin{align*}
        \left\| x - T_u \right\|_{L^{\infty}(B_{4R})} 
        &\leq \|x-\nabla u\|_{L^{\infty}(B_{4R})} + \|\nabla u - T_u\|_{L^{\infty}(B_{4R})} \\
        &\lesssim \frac{1}{R} \| u-\tfrac{1}{2}|\cdot|^2\|_{\mathcal{C}^0(B_{8R})} + \|\nabla u - T_u\|_{L^{\infty}(B_{4R})}. 
    \end{align*}
    It remains to estimate the latter term. To this end, notice that for a.e. $x\in B_{4R}$ we have $\nabla u(x) = -\nabla_x c(x, T_u(x))$, so that with the normalization assumption $\nabla_x c(0,0) = 0$ we may bound
    \begin{align*}
        |\nabla u(x) - T_u(x)| 
        &= |\nabla_x c(x, T_u(x)) + T_u(x)| \\
        &\leq |\nabla_x c(x, T_u(x)) - \nabla_x c(x,0) + T_u(x)| + |\nabla_x c(x,0) - \nabla_x c(0,0)| \\
        &\leq \int_0^1 |(\nabla_{xy}c (x,s T_u(x)) + \II) T_u(x)|\,\dd s + \int_0^1 |\nabla_{xx}c(t x, 0) x|\,\dd t.
    \end{align*}
    It now follows with \eqref{eq:Linfty-Tu}, definition \eqref{eq:c-hoelder-def}, and $\nabla_{xx}c(0,0) = 0$, $\nabla_{xy} c(0,0) = -\II$, that 
    \begin{align}
       |\nabla u(x) - T_u(x)|
        &\leq [\nabla_{xy}c]_{\alpha,4R} (|x|^{\alpha} + |T_u(x)|^{\alpha}) |T_u(x)| + [\nabla_{xx}c]_{\alpha,4R} |x|^{\alpha+1} \label{eq:first-gradu-Tu} \\
        &\stackrel{\eqref{eq:Linfty-Tu}}{\leq} C_{\Lambda_0} R^{\alpha}\left( [\nabla_{xy}c]_{\alpha,4R} + [\nabla_{xx}c]_{\alpha,4R}\right) R.\nonumber
    \end{align}
    In view of \eqref{eq:smallness-main-corollary}, we may assume 
    \begin{align*}
        C_{\Lambda_0} R^{\alpha}\left( [\nabla_{xy}c]_{\alpha,4R} + [\nabla_{xx}c]_{\alpha,4R}\right) \leq 1,
    \end{align*}
    so that
    \begin{align}\label{eq:c-lambda0}
        |\nabla u(x) - T_u(x)| \leq R.
    \end{align}
    Using that by Lemma \ref{lem:semi-convex} and the smallness assumption \eqref{eq:smallness-main-corollary}, we have
    \begin{align*}
        |\nabla u(x) - x| \lesssim \frac{1}{R} \|u-\tfrac{1}{2}|\cdot|^2\|_{\mathcal{C}^0(B_{8R})} \lesssim R,
    \end{align*} 
    and writing $T_u(x) = (T_u(x)-\nabla u(x)) + (\nabla u(x) -x) +x$, the estimate \eqref{eq:first-gradu-Tu} turns into
    \begin{align*}
        |\nabla u(x) - T_u(x)| &\lesssim [\nabla_{xy}c]_{\alpha,4R}(R^{\alpha} + |\nabla u(x)-T_u(x)|^{\alpha})(|\nabla u(x)-T_u(x)|+R) + R^{1+\alpha}[\nabla_{xx}c]_{\alpha,4R} \\
        &\overset{\eqref{eq:c-lambda0}}{\lesssim} R^{1+\alpha}\left([\nabla_{xy}c]_{\alpha,4R} + [\nabla_{xx}c]_{\alpha,4R}\right).
    \end{align*}
    This proves the claimed inequality \eqref{eq:cor-df-Linfty-bound}.
\end{proof}

\section{\texorpdfstring{$\epsilon$}{epsilon}-regularity for almost-minimizers}\label{sec:almost-min-general}
In this section we give a sketch of the proof of Theorem~\ref{thm:almost-min}. 
One of the main differences compared to the situation of Theorem~\ref{thm:main} is that our assumptions do not allow us to prove an $L^{\infty}$ bound on the displacement (which followed from $(c-)$monotonicity of $\supp\pi$). However, almost-minimality (on all scales) allows us to obtain an $L^p$ bound for arbitrarily large $p<\infty$.

\begin{proposition} \label{prop:Lp-bound-almost-min-main} 
    Assume that $\rho_0, \rho_1 \in C^{0,\alpha}$ with $\rho_0(0) = \rho_1(0) = 1$. Let $T$ be an almost-minimizing transport map from $\mu=\rho_0\,\mathrm{d}x$ to $\nu=\rho_1\,\mathrm{d}y$ with $\Delta_r \leq 1$. 
    Assume further that $T$ is invertible. Then there exists a radius $R_1 = R_1(\rho_0, \rho_1) >0$ such that for any $6R\leq R_1$,
    \begin{equation}\label{eq:Lp-bound-small-assumption-main}
        \EE_{6R}(\pi_T) + R^{\alpha} \left( [\rho_0]_{\alpha, 6R} + [\rho_1]_{\alpha, 6R} \right) \ll 1
    \end{equation} 
    implies that for any $p < \infty$, 
    \begin{equation}\label{eq:Lp-bound-main}
        \frac{1}{R} \left( \frac{1}{R^d} \int_{B_{2R}} |T(x)-x|^p \, \mu(\mathrm{d}x) + \frac{1}{R^d} \int_{B_{2R}} |T^{-1}(y)-y|^p \, \nu(\mathrm{d}y) \right)^{\frac{1}{p}} \lesssim_p  \EE_{6R}(\pi_T)^{\frac{1}{d+2}}. 
    \end{equation}
\end{proposition} 

The scale $R_1$ below which the result holds depends on the global H\"older semi-norms $[\rho_0]_{\alpha}$ and $[\rho_1]_{\alpha}$ of the densities and the condition $B_{R_1} \subset \supp \rho_0 \cap \supp \rho_1$. 

The proof of Proposition~\ref{prop:Lp-bound-almost-min-main} is given in Appendix~\ref{app:Lp}. Note that since $T^{-1}$ is also almost-minimizing, the $L^p$ bound for $T^{-1}$ follows from applying Proposition~\ref{prop:Lp-bound-almost-min} to $T^{-1}$. The $L^p$ estimate (for arbitrarily large $p<\infty$) allows us to split the particle trajectories into two groups: 
\begin{itemize}
    \item \emph{good trajectories} that satisfy an $L^{\infty}$ bound on the displacement, corresponding to starting points in the set
    \begin{align}\label{eq:good-trajectories}
        \mathcal{G} = \{x\in B_{2R}\cup T^{-1}(B_{2R}): |T(x) - x| \leq M R \},
    \end{align}
    where $M := (\EE_{6R}(\pi_T) + \DD_{6R}(\mu, \nu))^{\alpha}$ for some $\alpha\in (0, \frac{1}{d+2})$, that we fix in what follows, and
    \item \emph{bad trajectories} that are too long, corresponding to 
    \begin{align}\label{eq:bad-trajectories}
        \mathcal{B} = \{x\in B_{2R}\cup T^{-1}(B_{2R}): |T(x) - x| > M R \}.
    \end{align}
\end{itemize}
Due to the $L^p$ bound, the energy carried by bad trajectories is superlinearly small:
By definition of $\EE_{6R}(\pi_T)$ and $M$, 
\begin{align}\label{eq:measure-bad-trajectories}
    \frac{1}{R^d} \mu(\mathcal{B}) 
    \lesssim \frac{1}{M^2} \frac{1}{R^{d+2}} \int_{B_{2R}\cup T^{-1}(B_{2R})} |T(x) - x|^2\,\mu(\mathrm{d}x) 
    \lesssim \EE_{6R}(\pi_T)^{1-2\alpha},
\end{align}
hence
\begin{align*}
    \frac{1}{R^{d+2}} \int_{\mathcal{B}} |T(x) - x|^2\,\mu(\mathrm{d}x) 
    &\leq \left( \frac{1}{R^d} \mu(\mathcal{B}) \right)^{1-\frac{2}{p}} \left(\frac{1}{R^{d+p}} \int_{B_{2R}\cup T^{-1}(B_{2R})} |T(x) - x|^{p}\, \mu(\mathrm{d}x) \right)^{\frac{2}{p}} \\
    &\lesssim \EE_{6R}(\pi_T)^{\frac{2}{d+2} + (1-2\alpha)\left(1- \frac{2}{p} \right)}, 
\end{align*}
from which we see that, given $\alpha\in(0,\frac{1}{d+2})$, we may choose $p$ large enough so that the exponent is larger than $1$. In particular, for any $\tau>0$ we may bound
\begin{equation}\label{eq:energy-bad-trajectories}
    \frac{1}{R^{d+2}}\int_{\mathcal{B}} |T(x) - x|^2\,\mu(\mathrm{d}x) \leq \tau \EE_{6R}(\pi_T),
\end{equation} provided $\EE_{6R}(\pi_T)\ll 1$.

Once the bad trajectories have been removed, the good trajectories can be treated as before. More precisely, if we restrict the coupling $\pi_T$ to the set $\mathcal{G}\times T(\mathcal{G})$, then the resulting coupling is still deterministic and almost-minimizing with respect to quadratic cost (given its own boundary conditions). In particular, since $\mathcal{G}\times T(\mathcal{G}) \subset \cross{2R}$ the \emph{global} estimate
\begin{align*}
    \int_{\mathcal{G}} |T(x) - x|^2\,\mu(\mathrm{d}x) \leq \int_{\mathcal{G}} |\widetilde{T}(x) - x|^2\,\mu(\mathrm{d}x) + \Delta_{2R}
\end{align*}
for all $\widetilde{T}_{\#}\mu\lfloor_{\mathcal{G}} = T_{\#}\mu\lfloor_{\mathcal{G}} = \nu\lfloor_{T(\mathcal{G})}$ holds. This allows for a passage from the Lagrangian to the Eulerian point of view like in Lemma~\ref{prop:quasi-minimality-eulerian-main}.\footnote{Using that the optimal coupling between $\mu\lfloor_{\mathcal{G}}$ and $\nu\lfloor_{T(\mathcal{G})}$ is deterministic, so that one can appeal to almost-minimality within the class of deterministic couplings.} Moreover, the measures $\mu\lfloor_{\mathcal{G}}$ and $\nu\lfloor_{T(\mathcal{G})}$, as well as the coupling $\pi_{T}\lfloor_{\mathcal{G} \times T(\mathcal{G})}$, satisfy the assumptions of the harmonic approximation Theorem~\ref{thm:harmonic}.\footnote{Note that there is a slight mismatch in the power $\alpha$ in the $L^{\infty}$ bound between the definition of good trajectories and the setting of \cite{GHO19}. However, one can convince oneself easily that the results of \cite{GHO19} still apply.} Hence, given $0<\tau\ll 1$, there exists a threshold $\epsilon_{\tau}>0$, constants $C,C_{\tau}<\infty$, and a harmonic gradient field $\nabla\Phi$ (defined through \eqref{eq:harmonic-approx}) satisfying \eqref{eq:harmonic-estimate}, such that \eqref{eq:harmonic-energy} holds for the Eulerian description $(\rho,j)$ of $\pi_{T}\lfloor_{\mathcal{G} \times T(\mathcal{G})}$, provided $\EE_{6R}(\pi_T) + R^{2\alpha} \left([\rho_0]_{\alpha, 6R}^2 + [\rho_1]_{\alpha, 6R}^2\right) \leq \epsilon_{\tau}$. The harmonic gradient field allows us to define the affine change of coordinates from Lemma~\ref{lem:coordinate-change} with $B=\mathrm{e}^{\frac{A}{2}}$ with $A=\nabla^2\Phi(0)$ and $b = \nabla\Phi(0)$ satisfying \eqref{eq:B-b-est-main} (and $D = \mathbb{I}$) to obtain a new coupling $\widehat{\pi_T}= \pi_{\widehat{T}}$ between the measures $\widehat{\mu}$ and $\widehat{\nu}$ from the (full) coupling $\pi_{T}$. 

We can now use the harmonic approximation result Theorem~\ref{thm:harmonic} together with the harmonic estimates \eqref{eq:harmonic-estimate} to bound
\begin{align*}
    \frac{1}{R^{d+2}} \int_{\cross{R}\cap (\mathcal{G} \times T(\mathcal{G}))} |y-x - (Ax+b)|^2\,\mathrm{d}\pi_T  
    \leq \tau \EE_{6R}(\pi_T) + C_{\tau} R^{2\alpha} ([\rho_0]_{\alpha, 6R}^2 +[\rho_1]_{\alpha, 2R}^2) + \Delta_{2R}.
\end{align*}
For the bad trajectories we use the estimate \eqref{eq:energy-bad-trajectories} together with the bound 
\begin{align*}
    \int_{\cross{R}\cap (\mathcal{B} \times T(\mathcal{B}))} |Ax+b|^2\,\mathrm{d}\pi_T &\lesssim \int_{(B_R\cap \mathcal{B}) \times \RR^d} |Ax+b|^2\,\mathrm{d}\pi_T 
    +|A|^2 \int_{\mathcal{B}\times T(\mathcal{B})} |y-x|^2\,\mathrm{d}\pi_T \\
    &\qquad + |A|^2 \int_{\mathcal{B} \times (B_R \cap T(\mathcal{B}))} |y|^2\,\mathrm{d}\pi_T \\
    &\stackrel{\eqref{eq:energy-bad-trajectories}}{\lesssim} (|A|^2 R^2+|b|^2) \mu(\mathcal{B}) + \tau |A|^2 R^{d+2} \EE_{6R}(\pi_T) \\
    &\stackrel{\eqref{eq:measure-bad-trajectories}\&\eqref{eq:B-b-est-main}}{\lesssim} R^{d+2} \left(\EE_{6R}(\pi_T)^{2-2\alpha} + R^{2\alpha} ([\rho_{0}]_{\alpha, 6R}^2 + [\rho_1]_{\alpha, 6R}^2) + \tau \EE_{6R}(\pi_T)\right)
\end{align*}
to obtain, recalling that $\alpha < \frac{1}{d+2} < \frac{1}{2}$, 
\begin{align}\label{eq:harmonic-all-trajectories}
     \frac{1}{R^{d+2}} \int_{\cross{R}\cap (\mathcal{B} \times T(\mathcal{B}))} |y-x-(Ax+b)|^2\,\mathrm{d}\pi_T
     \leq \tau \EE_{6R}(\pi_T) + C_{\tau} R^{2\alpha} ([\rho_0]_{\alpha, 6R}^2 + [\rho_1]_{\alpha, 6R}^2).
\end{align}

In particular, for $\theta\in(0,1)$ we can write 
\begin{align*}
    (\theta R)^{d+2} \EE_{\theta R}(\pi_{\widehat{T}}) 
    &= \int_{\cross{\theta R}} |\widehat{y}-\widehat{x}|^2\,\pi_{\widehat{T}}(\mathrm{d}\widehat{x}\mathrm{d}\widehat{y}) 
    = |\det B| \int_{Q^{-1}(\cross{\theta R})} |\gamma B^{-*}(y-b) - Bx|^2 \,\pi_T(\mathrm{d}x\mathrm{d}y), 
\end{align*}
so that using the identity $\gamma B^{-*}(y-b) - Bx = \gamma B^{-*} (y-x-(Ax+b)) - \gamma B^{-*} (B^*B-\mathbb{I}-A) x + (\gamma-1) Bx$, the estimate \eqref{eq:harmonic-all-trajectories} together with \eqref{eq:harmonic-estimate} and \eqref{eq:B-b-est-main} imply that for any $\beta\in(0,1)$ there exist $0<\theta\ll 1$ and $C_{\beta}<\infty$ such that 
\begin{align*}
    \EE_{\theta R}(\pi_{\widehat{T}}) \leq \theta^{2\beta} \EE_{6R}(\pi_T) + C_{\beta} \left( R^{2\alpha} [\rho_0]_{\alpha, 6R}^2 + R^{2\alpha} [\rho_1]_{\alpha, 6R}^2 + \Delta_{2R} \right),
\end{align*}
which implies a one-step-improvement result for the case of general almost-minimizers. 

It remains to show that the transformed coupling $\pi_{\widehat{T}}$ is still almost-minimal on all small scales in the sense of Definition~\ref{def:almost-min-general}. To this end, let $r\leq R_1$, $(\widehat{x_0}, \widehat{y_0}) \in \supp\pi_{\widehat{T}}$, and $\pi_{\widehat{T'}} \in \Pi(\widehat{\mu},\widehat{\nu})$ with $\supp(\pi_{\widehat{T}}-\pi_{\widehat{T'}}) \subset (B_r(\widehat{x_0})\times \RR^d) \cup (\RR^d \times B_r(\widehat{y_0}))$. Then, writing $\widehat{T'}(\widehat{x}) = \gamma B^{-*} (T'(B^{-1}\widehat{x}) - b)$, where $T'_{\#}\mu = \nu$, one sees that 
\begin{align*}
    \int |\widehat{y} - \widehat{x}|^2\,\mathrm{d}(\pi_{\widehat{T}} - \pi_{\widehat{T'}}) 
    &= -2 \int (\widehat{T}(\widehat{x}) - \widehat{T'}(\widehat{x}))\cdot \widehat{x}\,\widehat{\mu}(\mathrm{d}\widehat{x}) \\
    &= -2 |\det B| \int \gamma B^{-*}(T(x) - T'(x))\cdot Bx \,\mu(\mathrm{d}x) \\
    &= - 2 \gamma |\det B| \int (T(x) - T'(x))\cdot x \,\mu(\mathrm{d}x) \\
    &= \gamma |\det B| \int |y-x|^2\,\mathrm{d}(\pi_{T} - \pi_{T'}).
\end{align*}
Note that
\begin{align*}
    \supp(\pi_T - \pi_{T'}) 
    = Q^{-1} \supp(\pi_{\widehat{T}} - \pi_{\widehat{T'}}) 
    &\subset (B_{|B|r}(x_0) \times \RR^d) \cup (\RR^d \times B_{\frac{|B|}{\gamma}r}(y_0)),
\end{align*}
where $x_0 = B\widehat{x_0}$ and $y_0 = \gamma^{-1}B \widehat{y_0} + b$. Since $\pi_T$ is almost-minimizing, it follows that 
\begin{align*}
    \int |\widehat{y} - \widehat{x}|^2\,\mathrm{d}(\pi_{\widehat{T}} - \pi_{\widehat{T'}}) 
    &\leq \gamma |\det B| (\max\{1, \gamma^{-1}\} |B| r)^{d+2} \Delta_{\max\{1, \gamma^{-1}\} |B| r},
\end{align*}
hence $\pi_{\widehat{T}}$ is almost-minimizing among deterministic couplings with rate 
\begin{align*}
    \widehat{\Delta}_r = \gamma |\det B| (\max\{1, \gamma^{-1}\} |B|)^{d+2} \Delta_{\max\{1, \gamma^{-1}\} |B| r}.
\end{align*}
Assuming that $\Delta_r = C r^{2\alpha}$, together with the bounds on $\gamma$ and $B$ from \eqref{eq:B-b-est-main} this gives
\begin{align*}
    \widehat{\Delta}_{r} \leq \left(1+ C (\EE_{R_1}^{\frac{1}{2}} + R_1^{\alpha}[\rho_0]_{\alpha,R_1} + R_1^{\alpha}[\rho_1]_{\alpha, R_1})\right) \Delta_r,
\end{align*}
in particular the rate $\Delta_r$ exhibits the same behaviour as the Hölder seminorm of $\nabla_{xy}c$ in \eqref{eq:estimate-seminorm-c}
and shows that the one-step-improvement can be iterated down to arbitrarily small scales, yielding the $C^{1,\alpha}$-regularity of $T$ in a ball with radius given by a fraction of $R$. \qed
\section{Partial regularity: Proof of Corollary \ref{cor:partial}} \label{sec:partial}
As a corollary of Theorem \ref{thm:main}, we obtain a variational proof of partial regularity for optimal transport maps proved in \cite{DF14}. The changes of variables used to arrive to a normalized situation are exactly the same as in \cite{DF14} and the argument to derive partial regularity from $\epsilon$-regularity follows \cite{GO17}.

\begin{proof}[Proof of Corollary \ref{cor:partial}]
A classical result in optimal transport states that the optimal map $T$ from $\rho_0$ to $\rho_1$ for the cost $c$ and the optimal map $T^*$ from $\rho_1$ to $\rho_0$ for the cost $c^*(y,x) := c(x,y)$ are almost everywhere inverse to each other, and are of the form 
\begin{align*}
    T(x) = \cexp_x(\nabla u(x)) 
    \quad \text{and} \quad 
    T^{-1}(y) := T^*(y) = \csexp_y(\nabla u^c(y)),
\end{align*}
where $u$ is a $c$-convex function and $u^c$ is the $c$-conjugate of $u$. $u$ and $u^c$ are semi-convex so that by Alexandrov's Theorem, they are twice differentiable almost everywhere. Therefore, we can find two sets of full measure $X_1 \subseteq X$ and $Y_1 \subseteq Y$ such that for all $(x_0,y_0) \in X_1 \times Y_1$, $u$ is twice differentiable at $x_0$, $u^c$ is twice differentiable at $y_0$ and 
\begin{align}\label{eq:inverse}
    T^{-1}(T(x_0)) = x_0 
    \quad \text{and} \quad 
    T(T^{-1}(y_0)) = y_0.
\end{align}
Now let 
\begin{align}\label{eq:Eprime-def}
    X' := X_1 \cap T^{-1}(Y_1) \quad \text{and} \quad Y' := Y_1 \cap T(X_1). 
\end{align} Because $\rho_0$ and $\rho_1$ are bounded and bounded away from zero, $T$ sends sets of measure $0$ to sets of measure $0$ so that $|X \setminus X'| = |Y \setminus Y'| = 0$. The goal is now to prove that $X'$ and $Y'$ are open sets and that $T$ is a $\mathcal{C}^{1, \alpha}$-diffeomorphism between $X'$ and $Y'$. 

Fix $x_0 \in X'$, then by \eqref{eq:Eprime-def}, $y_0 := T(x_0) \in Y'$. Up to translation, we may assume that $x_0 = y_0 = 0$. Define 
\begin{align}\label{eq:cbar-def}
    \overline{u}(x) &:= u(x)-u(0)+c(x,0)-c(0,0), \nonumber\\
    \overline{c}(x,y) &:= c(x,y)-c(x,0)-c(0,y)+c(0,0).
\end{align}
Then $\overline{u}$ is a $\overline{c}$-convex function and we have 
\begin{align}\label{eq:cbexp-equal-cexp}
    \cbexp_x(\nabla \overline{u}(x)) = \cexp_x(\nabla u(x)),
\end{align} so that $T(x) = \cbexp_x(\nabla \overline{u}(x))$, from which we know that $T$ is the $\overline{c}$-optimal transport map from $\rho_0$ to $\rho_1$. 

By Alexandrov's Theorem, there exist a symmetric matrix $A$ such that 
\begin{align*}
    \nabla \overline{u}(x) = \nabla \overline{u}(0) + Ax + o(|x|),
\end{align*}
so that, using that $(p,x) \mapsto \cexp_x(p)$ is $\mathcal{C}^1$ and setting $M := -\nabla_{xy}\overline{c}(0,0) = -\nabla_{xy}c(0,0)$, noticing that by Assumption \ref{item:cost-non-deg} $M$ is nondegenerate, a simple computation yields 
\begin{align*}
    T(x) = M^{-1}Ax + o(|x|).
\end{align*}
Therefore, we have 
\begin{align}\label{eq:smallness-partial}
    \frac{1}{R^{d+2}} \int_{B_R} |T(x)-M^{-1}Ax|^2 \rho_0(x) \, \dd x \underset{R \to 0}{\longrightarrow} 0.
\end{align}
The $\overline{c}$-convexity of $\overline{u}$ and the fact that $\cbexp_x(\nabla \overline{u}(x)) \in \partial_{\overline{c}} \overline{u}(x)$ imply that, see for instance \cite[Section 5.3]{Fig17},  \begin{align*}
    \nabla^2 \overline{u}(x) + \nabla_{xx} \overline{c}(x, \cbexp_x(\nabla \overline{u}(x))) \geq 0, 
\end{align*} so that, together with \eqref{eq:cbar-def}, \eqref{eq:cbexp-equal-cexp} and the property $T(0)=0$, the matrix $A = \nabla^2\overline{u}(0)$ is positive definite. We now make the change of variables $\widetilde{x} := A^{\frac{1}{2}}x$ and $\widetilde{y} := A^{-\frac{1}{2}}My$ so that 
\begin{align*}
    \widetilde{T}(\widetilde{x}) &:= A^{-\frac{1}{2}}MT(A^{-\frac{1}{2}}\widetilde{x}), \\
    \widetilde{c}(\widetilde{x}, \widetilde{y}) &:= \overline{c}(A^{-\frac{1}{2}}\widetilde{x}, M^{-1}A^{\frac{1}{2}}\widetilde{y}).
\end{align*}
Defining 
\begin{align*}
    \widetilde{\rho}_0(\widetilde{x}) := \det(A^{-\frac{1}{2}}) \rho_0(A^{-\frac{1}{2}}\widetilde{x}) \quad \text{and} \quad \widetilde{\rho}_1(\widetilde{y}) := |\det(M^{-1}A^{\frac{1}{2}})| \rho_1(M^{-1}A^{\frac{1}{2}}\widetilde{y}),
\end{align*} we get that $\widetilde{T}_{\#} \widetilde{\rho}_0 = \widetilde{\rho}_1$ $\widetilde{c}$-optimally. This may be seen noticing that 
\begin{align*}
    \widetilde{T}(\widetilde{x}) = \ctexp_{\widetilde{x}}(\nabla \widetilde{u}(\widetilde{x})),
\end{align*} 
where $\widetilde{u}(\widetilde{x}) := \overline{u}(A^{-1/2}x)$ is a $\widetilde{c}$-convex function. 
The cost $\widetilde{c}$ satisfies $\nabla_{\widetilde{x} \widetilde{y}}\widetilde{c}(0,0) = -\II$ and by the Monge--Ampère equation 
\begin{align*}
    \big|\det \nabla \widetilde{T}(x)\big| = \frac{\widetilde{\rho}_0(x)}{\widetilde{\rho}_1(\widetilde{T}(x))},
\end{align*} 
we obtain $\widetilde{\rho}_0(0) = \widetilde{\rho}_1(0)$. Up to dividing $\widetilde{\rho}_0$ and $\widetilde{\rho}_1$ by an equal constant, we may assume that $\widetilde{\rho}_0(0) = \widetilde{\rho}_1(0) = 1$.
Moreover, with this change of variables, \eqref{eq:smallness-partial} turns into
\begin{align}\label{eq:partial-smallness}
    \frac{1}{R^{d+2}} \int_{B_R} |\widetilde{T}(\widetilde{x})-\widetilde{x}|^2 \widetilde{\rho}_0(\widetilde{x}) \, \dd\widetilde{x} \underset{R \to 0}{\longrightarrow} 0.
\end{align}
Finally, $\widetilde{c}$ is still $\mathcal{C}^{2,\alpha}$ and satisfies Assumptions \ref{item:cost-inj-y}--\ref{item:cost-non-deg} and since $\rho_0$ and $\rho_1$ are bounded and bounded away from zero, $\widetilde{\rho}_0$ and $\widetilde{\rho}_1$ are $\mathcal{C}^{0, \alpha}$, and we have 
\begin{align}\label{eq:partial-H-smallness}
    \HH_R(\widetilde{\rho}_0, \widetilde{\rho}_1, \widetilde{c}) = R^{2\alpha} \left( [\widetilde{\rho}_0]_{\alpha, R}^2 + [\widetilde{\rho}_1]_{\alpha, R}^2 + \left[\nabla_{xy} \widetilde{c}\, \right]_{\alpha, R}^2 \right) \underset{R \to 0}{\longrightarrow} 0.
\end{align}
Hence by \eqref{eq:partial-smallness} and \eqref{eq:partial-H-smallness}, we may apply Theorem \ref{thm:main} to obtain that $\widetilde{T}$ is $\mathcal{C}^{1,\alpha}$ in a neighborhood of zero. 
By Remark \ref{rem:T-inverse}, we also obtain that $\widetilde{T}^{-1}$ is $\mathcal{C}^{1,\alpha}$ in a neighborhood of zero.

Going back to the original map, this means that $T$ is a $\mathcal{C}^{1,\alpha}$ diffeomorphism between a neighborhood $U$ of $x_0$ and the neighborhood $T(U)$ of $T(x_0)$. In particular, $U \times T(U) \subseteq X' \times Y'$ so that $X'$ and $Y'$ are both open and by \eqref{eq:inverse}, $T$ is a global $\mathcal{C}^{1,\alpha}$ diffeomorphism between $X'$ and $Y'$. 
\end{proof}
\appendix
\section{Some technical lemmata}\label{app:technical}

\subsection{Properties of the support of couplings}
The following lemma is an important ingredient in the proofs of our $L^{\infty}$ bounds on the displacement of couplings with $c$-monotone support, Proposition \ref{prop:Linfty-main} and Lemma \ref{lem:displacement-qualitative}:
\begin{lemma}\label{lem:cone}
	Let $\pi \in \Pi(\mu, \nu)$ and assume that there exists $R>0$ such that 
	\begin{align}\label{eq:cone-smallness}
		\frac{1}{R^{d+2}} \int_{B_{6R} \times \RR^d} |y-x|^2 \, \dd\pi + \DD_{6R} \ll 1.
	\end{align}
	Then
	\begin{enumerate}[label=(\roman*)]
		\item \label{item:cone-1} $(B_{\lambda R} \times B_{2\lambda R}) \cap \supp\pi \neq \emptyset$ provided $\lambda \gg \left(\frac{1}{R^{d+2}} \int_{B_{6R} \times \RR^d} |y-x|^2 \, \dd\pi + \DD_{6R}\right)^{\frac{1}{d+2}}$.
		\item \label{item:cone-3} For any $x \in B_{5R}$ and $e \in S^{d-1}$ we have that $(S_{R}(x,e) \times B_{7R}) \cap \supp\pi\neq \emptyset$, where \[S_{R}(x,e) := C(x,e) \cap (B_{R}(x)\setminus B_{\frac{R}{2}}(x))\] is the intersection of the annulus $B_{R}(x)\setminus B_{\frac{R}{2}}(x)$ with the spherical cone $C(x,e)$ of opening angle $\frac{\pi}{2}$ with apex at $x$ and axis along $e$.
	\end{enumerate}
\end{lemma}
\begin{figure}[!ht]
\begin{tikzpicture}[line cap=round,line join=round,>=triangle 45,x=1.0cm,y=1.0cm]
\clip(-2.8,-1) rectangle (5.654665125482122,4.8);
\draw [line width=1.pt] (0.,0.) circle (3.cm);
\draw [line width=1.pt,domain=-5.266953937349604:1.22] plot(\x,{(-8.7852--1.92*\x)/-3.54});
\draw [line width=1.pt,domain=1.22:5.654665125482122] plot(\x,{(-0.6038011460131179--2.592741456679323*\x)/1.4062326544701411});
\draw [->,line width=1.pt] (1.22,1.82) -- (0.4427647044109043,4.4395708110595455);
\draw [line width=1.pt] (0.,0.) circle (4.cm);
\draw (0.9841577351743219,4.700298113503307) node[anchor=north west] {$C(x,e)$};
\draw (0.4850195407902822,4.664645385333019) node[anchor=north west] {{$e$}};
\draw (3.586806891605386,2.2283756270299806) node[anchor=north west] {$B_{6R}$};
\draw (2.945057784540192,0.8141507429418756) node[anchor=north west] {$B_{5R}$};
\draw (1.2218425896429124,1.9669222871145327) node[anchor=north west] {{$x$}};
\draw [line width=1.pt] (1.22,1.82) circle (1.5cm);
\draw [line width=1.pt] (1.22,1.82) circle (0.75cm);
\draw (0.5682092398542888,1.15) node[anchor=north west] {$B_{\frac{R}{2}}(x)$};
\draw (-1.35,1.8599641026036675) node[anchor=north west] {$B_R(x)$};
\draw (-0.8,3.7) node[anchor=north west] {$\bm{S_R(x,e)}$};
\draw [line width=1.pt] (0.04313301465180075,2.750045213308171)-- (0.04451652942338402,2.4575503569229102);
\draw [line width=1.pt] (0.16258251566546922,2.8838929757375187)-- (0.16484101882796254,2.3922896169068677);
\draw [line width=1.pt] (0.2866719329815402,2.328384650594847)-- (0.28538485080534537,2.993241033588518);
\draw [line width=1.pt] (0.4015315730778378,2.2639150790086306)-- (0.40179614618818943,3.07560026226913);
\draw [line width=1.pt] (0.5175065672184895,3.145331270625895)-- (0.5179615565265798,2.200766613409313);
\draw [line width=1.pt] (0.6412517389945414,3.2038534786541373)-- (0.6445561644459589,2.301003526102308);
\draw [line width=1.pt] (0.8866017253819901,3.2824792615561202)-- (0.8870493173249011,2.4920445244968414);
\draw [line width=1.pt] (1.0009428624687338,3.303918451430741)-- (1.0011225543892308,2.5576980574288886);
\draw [line width=1.pt] (1.1140112003937337,3.3162507725505184)-- (1.1146281676330594,2.562560958402496);
\draw [line width=1.pt] (1.2357206092671997,3.319917618552522)-- (1.235651154841407,2.5698366764516996);
\draw [line width=1.pt] (1.3553130543488467,3.3138843252818426)-- (1.3521897171553574,2.5582586800562432);
\draw [line width=1.pt] (1.4835107704075157,3.2966726359891814)-- (1.475830662760142,2.5250182068512177);
\draw [line width=1.pt] (1.603898043194798,3.2700421691906776)-- (1.6069020708752728,2.533350693176284);
\draw [line width=1.pt] (1.716325339417105,3.2355073851635305)-- (1.7251178746687366,2.751311081420483);
\draw [line width=1.pt] (1.840175307258799,3.1857900967082955)-- (1.8445800508204075,2.9715694687001264);
\draw [line width=1.pt] (0.7576401339044782,3.246962982779908)-- (0.7605048901979343,2.412759853623698);
\draw [line width=1.pt] (-0.03837413108626102,2.5025080033010227)-- (-0.03921467759201147,2.635094102381343);
\begin{scriptsize}
\draw [fill=black] (0.,0.) circle (2.0pt);
\draw[color=black] (0.18,0.21) node {$O$};
\draw [fill=black] (1.22,1.82) circle (2.0pt);
\end{scriptsize}
\end{tikzpicture}
\end{figure}

\begin{remark}
    If \eqref{eq:cone-smallness} is replaced by $\frac{1}{R^{d+2}} \int_{\RR^d \times B_{6R}} |y-x|^2 \, \dd\pi + \DD_{6R} \ll 1$, then the symmetric results hold, namely $(B_{2\lambda R} \times B_{\lambda R}) \cap \supp\pi \neq \emptyset$ and $(B_{7R} \times S_R(y,e)) \cap \supp\pi \neq \emptyset$ for all $y \in B_{5R}$ and $e \in S^{d-1}$.     
\end{remark}

\begin{proof}
    To lighten the notation in the proof, let us set 
    \begin{align}\label{eq:Eplus-def}
        \EE_{6R}^+ := \frac{1}{R^{d+2}} \int_{B_{6R} \times \RR^d} |y-x|^2 \, \dd\pi.
    \end{align}
    
	We start with the more delicate statement \ref{item:cone-3}. Let $x \in B_{5R}$ and $e \in S^{d-1}$, note that
	\begin{align*}
		\pi(S_{R}(x,e) \times B_{7R}) 
		&= \pi(S_{R}(x,e) \times \RR^d) - \pi(S_{R}(x,e) \times B_{7R}^c) \\
		&= \mu(S_{R}(x,e)) - \pi(S_{R}(x,e) \times B_{7R}^c).
	\end{align*}
	Since $S_{R}(x,e) \subseteq B_{6R}$, we have
	\begin{align*}
		\pi(S_{R}(x,e) \times B_{7R}^c)
		\lesssim \frac{1}{R^2} \int_{B_{6R} \times B_{7R}^c} \frac{1}{2} |x'-y'|^2\,\pi(\dd x'\dd y') 
		\overset{\eqref{eq:Eplus-def}}{\lesssim} R^d \EE_{6R}^+.
	\end{align*}
	To estimate $\mu(S_{R}(x,e))$ from below, let $\eta$ be a smooth cut-off function equal to one on a ball of radius $\frac{r}{2}$ and zero outside a concentric ball of radius $r$ satisfying 
	\begin{align}\label{eq:cutoff-assumptions}
		\sup |\eta| + r \sup |\nabla\eta| + r^2 \sup |\nabla^2 \eta| \lesssim 1,
	\end{align}
	and such that $\supp\eta \subseteq S_{R}(x,e) \subseteq B_{6R}$, which is possible provided $r \leq \frac{R}{4}$. Then by \eqref{eq:cutoff-assumptions} 
	\begin{align*}
		\mu(S_{R}(x,e)) 
		&\gtrsim \int_{B_{6R}} \eta \,\dd\mu 
		= \int_{B_{6R}} \eta\, \kappa_{\mu} \,\dd x + \int_{B_{6R}} \eta \,(\dd\mu -\kappa_{\mu}\,\dd x) \\
		&\gtrsim \kappa_{\mu} \left(\frac{r}{2}\right)^d - \left| \int_{B_{6R}} \eta \,(\dd\mu -\kappa_{\mu}\,\dd x) \right|.
	\end{align*}
	We now use \eqref{eq:GHO19-cutoff-bound} with $\zeta=\eta$ to get, by the definition \eqref{eq:D-def} of $\DD_{6R}$, and since $\kappa_{\mu} \sim 1$ by \eqref{eq:cone-smallness}, that
	\begin{align*}
		\left| \int_{B_{6R}} \eta \,(\dd\mu -\kappa_{\mu}\,\dd x) \right| 
		&\lesssim r^{\frac{d-2}{2}} R^{\frac{d+2}{2}} \DD_{6R}^{\frac{1}{2}} + r^{-2} R^{d+2} \DD_{6R}.
	\end{align*}
	Hence, 
	\begin{align*}
		\frac{\mu(S_{R}(x,e))}{R^d} 
		&\gtrsim \left(\frac{r}{R}\right)^d \left( 1 - \left( \left(\frac{r}{R}\right)^{-(d+2)}\DD_{6R}\right)^{\frac{1}{2}} - \left(\frac{r}{R}\right)^{-(d+2)} \DD_{6R}\right).
	\end{align*}
	We may now choose $r = \frac{R}{4}$ so that 
	\begin{align*}
	    \frac{\mu(S_{R}(x,e))}{R^d} 
		&\gtrsim 1-\left(4^{d+2}\DD_{6R}\right)^{\frac{1}{2}}-4^{d+2}\DD_{6R}, 
	\end{align*}
	from which we conclude that $\pi(S_{R}(x,e) \times B_{7R})$ is strictly positive if $\DD_{6R}$ and $\EE_{6R}^+$ are small enough.
	
	In order to prove \ref{item:cone-1} we run a similar argument to obtain 
	\begin{align*}
		\frac{\pi(B_{\lambda R} \times B_{2\lambda R})}{R^d}
		\gtrsim \lambda^d \left( 1 - \left(\lambda^{-(d+2)}(\EE_{6R}^+ + \DD_{6R})\right)^{\frac{1}{2}} - \lambda^{-(d+2)} (\EE_{6R}^+ + \DD_{6R})\right).
	\end{align*}
	Hence $(B_{\lambda R} \times B_{2\lambda R}) \cap \supp\pi \neq \emptyset$ provided that $\lambda \gg (\EE_{6R}^+ +\DD_{6R})^{\frac{1}{d+2}}$.
\end{proof}

The next lemma, which is quite elementary, relates the support of a measure and the support of its push forward under an affine transformation:
\begin{lemma}\label{lem:equivalence-support}
    Let $\gamma$ be a measure on $\RR^n$ and set
    $\widetilde{\gamma} := F_{\#} \gamma$, where $F(x) := Ax+b$ with $A \in \RR^{n\times n}$ invertible and $b\in \RR^n$. Then 
    \begin{align*}
        \widetilde{x} \in \supp\widetilde{\gamma} 
        \quad \Leftrightarrow \quad 
        F^{-1}(\widetilde{x}) \in \supp\gamma.
    \end{align*}
\end{lemma}

\begin{proof}
    Let $\widetilde{x} \in \supp\widetilde{\gamma}$. Then for all $\epsilon > 0$ 
    \begin{align*}
        \widetilde{\gamma}(B_{\epsilon}(\widetilde{x})) = \gamma(F^{-1}(B_{\epsilon}(\widetilde{x}))) > 0.
    \end{align*}
    Now 
    \begin{align*}
        F^{-1}(B_{\epsilon}(\widetilde{x})) &= \{ x: |\widetilde{x} - F(x)| < \epsilon\} = \{x: |(\widetilde{x}-b) - Ax| < \epsilon \} \\
        &\subseteq \{x: |A^{-1}(\widetilde{x}-b) -x| < |A^{-1}|\epsilon \} = B_{|A^{-1}|\epsilon}( F^{-1}(\widetilde{x})),
    \end{align*}
    so that for all $\epsilon' >0$ we have 
    $\gamma(B_{\epsilon'}(F^{-1}(\widetilde{x}))) > 0$. The other implication follows analogously. 
\end{proof}

\subsection{Bound on \texorpdfstring{$\DD_R$}{D R}}
In this subsection we show how the quantity $\DD_R(\rho_0, \rho_1)$ can be bounded in terms of the Hölder semi-norms of the densities $\rho_0, \rho_1$:
\begin{lemma}\label{lem:D-Holder_rho}
	Let $\rho_0,\rho_1\in \mathcal{C}^{0,\alpha}$, $\alpha\in(0,1)$, be two probability densities with bounded support, and such that $\frac{1}{2} \leq \rho_j \leq 2$, $j=0, 1$, on their support. If $\rho_0(0) = \rho_1(0) = 1$, then for all $R>0$ such that $B_R \subseteq \supp \rho_j$, $j=0,1$, we have 
	\begin{align}\label{eq:D-bound-rhoA}
		\DD_R \lesssim R^{2\alpha} \left( [\rho_0]_{\alpha, R}^2 + [\rho_1]_{\alpha,R}^2 \right).
	\end{align}
\end{lemma}

\begin{proof}
	By the definition \eqref{eq:kappa-def} of $\kappa_j := \kappa_{\rho_j}$ and using $\rho_j(0) = 1$, Jensen's inequality implies 
	\begin{align*}
		|\kappa_j - 1|^2 
		&= |\kappa_j - \rho_j(0)|^2 
		\leq \frac{1}{|B_R|} \int_{B_R} |\rho_j(x) - \rho_j(0)|^2\,\dd x
		\leq [\rho_j]_{\alpha, R}^2 R^{2\alpha}.
	\end{align*}
	If $R>0$ is such that $B_R \subseteq \supp \rho_j$, the assumption that $\rho_j$ is bounded away from zero on its support implies that $\kappa_j \gtrsim 1$.
	The claimed inequality \eqref{eq:D-bound-rhoA} then follows with Lemma \ref{lem:W2-H-1}.
\end{proof}

\begin{lemma}\label{lem:W2-H-1}
	Let $\rho\in \mathcal{C}^{0,\alpha}$, $\alpha\in(0,1)$, be a density with bounded support, and such that $\frac{1}{2} \leq \rho \leq 2$ on its support. 
	Then
	\begin{align*}
		\frac{1}{R^{d+2}} W_{B_R}^2(\rho,\kappa) \lesssim R^{2\alpha} [\rho]_{\alpha,R}^2.
	\end{align*}
\end{lemma}

\begin{proof}
	By the assumptions on $\rho$, the 2-Wasserstein distance between $\rho \,\dd x $ and $\kappa \, \dd x$ restricted to $B_R$ can be bounded by 
	\begin{align*}
		W^{2}_{B_R}(\rho, \kappa) \lesssim \|\nabla \Psi \|_{L^2(B_R)}^2,
	\end{align*}
	where $\Psi$ is the mean-zero solution of the Neumann problem
	\begin{align}\label{eq:W2-Neumann}
	\left\{
	\begin{array}{rcll}
		-\Delta \Psi &=& \rho - \kappa &\text{in } B_R, \\
		\nabla \Psi \cdot \nu &=& 0 &\text{on } \partial B_R,
	\end{array}
	\right.
	\end{align}
	see \cite[Theorem 5.34]{San15}.
	Global Schauder theory\footnote{See for instance \cite[Theorem 3.16]{Tro87} for details.} for \eqref{eq:W2-Neumann} implies that  
	\begin{align*}
		\frac{1}{R^2} \sup_{B_R} |\nabla \Psi|^2 
		\lesssim \|\rho-\kappa\|_{\mathcal{C}^{0}(B_R)}^2 + R^{2\alpha} [\rho-\kappa]_{\alpha,R}^2
		\lesssim R^{2\alpha} [\rho]_{\alpha,R}^2,
	\end{align*}
	so that 
	\begin{align*}
		W^{2}_{B_R}(\rho, \kappa) 
		\lesssim \|\nabla \Psi \|_{L^2(B_R)}^2
		\lesssim R^{d} \sup_{B_R} |\nabla \Psi|^2
		\lesssim R^{d+2} R^{2\alpha} [\rho]_{\alpha,R}^2.
	\end{align*}
\end{proof}

\subsection{A property of semi-convex functions}
\begin{lemma}\label{lem:semi-convex}
    Let $C\in\RR$ and assume that $v: \RR^d \to \RR$ is $C$-semi-convex, that is, $x\mapsto v(x) + \frac{C}{2} |x|^2$ is convex. Then 
    \begin{align}\label{eq:bound-semiconvexity}
    \|\nabla v\|_{L^{\infty}(B_1)} \lesssim \|v\|_{\mathcal{C}^0(B_3)}^{\frac{1}{2}} \max\{ \|v\|_{\mathcal{C}^0(B_3)}^{\frac{1}{2}}, C^{\frac{1}{2}}\}.
\end{align}
\end{lemma}

\begin{proof}
    Notice first that by semi-convexity, the gradient of $v$ exists a.e. Convexity of $x\mapsto v(x) + \frac{C}{2} |x|^2$ implies that for a.e. $y$ and all $x$,
\begin{align*}
    \nabla v(y) \cdot (x-y) \leq v(x) - v(y) + \frac{C}{2} |x-y|^2.
\end{align*}
In particular, for a.e. $y \in B_1$ and every $x\in B_3$ lying in a cone of opening angle $\frac{2\pi}{3}$ with apex at $y$ and axis along $\nabla v(y)$, so that $\nabla v(y) \cdot (x-y) \geq \frac{1}{2} |\nabla v(y)| |x-y|$, we have 
\begin{align*}
    |\nabla v(y)| \lesssim \frac{\|v\|_{\mathcal{C}^0(B_3)}}{|x-y|} + C |x-y|.
\end{align*}
Optimizing in $|x-y|$ then gives \eqref{eq:bound-semiconvexity}.
\end{proof}

\section{The change of coordinates Lemma \ref{lem:coordinate-change}}\label{app:proofcoord}

\begin{proof}[Proof of Lemma \ref{lem:coordinate-change}]
First, we clearly have $\widehat{\rho}_0(0) = \widehat{\rho}_1(0) = 1$ and $\nabla_{\widehat{x} \widehat{y}}\widehat{c}(0,0) = -\II$. It is also easy to check that $\widehat{\pi} \in \Pi(\widehat{\rho}_0, \widehat{\rho}_1)$.
Let us now compute
\begin{align*}
    \int \widehat{c}(\widehat{x}, \widehat{y}) \, \widehat{\pi}(\dd \widehat{x} \dd \widehat{y}) &= \frac{\gamma |\det B|}{\rho_0(0)} \int  c(Q^{-1}(\widehat{x}, \widehat{y})) \, (Q_{\#} \pi)(\dd\widehat{x} \dd\widehat{y}) \\
    &= \frac{\gamma |\det B|}{\rho_0(0)} \int c(x,y) \, \pi(\dd x\dd y).
\end{align*}
Thus, $\widehat{c}$-optimality of $\widehat{\pi}$ is equivalent to $c$-optimality of $\pi$. Indeed, if $\widehat{\pi}$ is not optimal for the cost $\widehat{c}$ then one can find a coupling $\widehat{\sigma} \in \Pi(\widehat{\rho}_0, \widehat{\rho}_1)$ such that $\int \widehat{c}(\widehat{x}, \widehat{y}) \, \widehat{\sigma}(\dd\widehat{x} \dd\widehat{y}) < \int \widehat{c}(\widehat{x}, \widehat{y}) \, \widehat{\pi}(\dd\widehat{x} \dd\widehat{y})$. Defining now $\sigma := \rho_0(0) |\det B|^{-1} \, (Q^{-1})_{\#} \widehat{\sigma} \in \Pi(\rho_0, \rho_1)$, the previous computation would yield
\begin{align*}
	\int c(x,y) \, \sigma(\dd x\dd y) < \int c(x,y) \, \pi(\dd x\dd y),		
\end{align*} a contradiction.  
\end{proof}

\begin{remark}
    It is also possible to prove the $\widehat{c}$-optimality of $\widehat{\pi}$ by showing that $\supp\widehat{\pi}$ is $\widehat{c}$-cyclically monotone, which characterizes optimality (see for instance \cite[Theorem 1.49]{San15}). This property readily follows  from Lemma \ref{lem:equivalence-support} and the $c$-cyclical monotonicity of $\supp\pi$. 
\end{remark}

\section{Some aspects of Campanato's theory}\label{app:campanato}

\begin{lemma}\label{lem:campanato}
    Let $R>0$, $\frac{1}{2}\leq \rho_0\leq 2$ on $B_R$, and assume that the coupling $\pi\in\Pi(\rho_0, \rho_1)$ satisfies \eqref{eq:campanato-pi-full} for $\alpha\in(0,1)$. For $r<\frac{R}{2}$ let $A_r = A_r(x_0, \pi)\in\RR^{d\times d}$ and $b_r = b_r(x_0, \pi)\in\RR^d$ be the (unique) minimizers of 
    \begin{align*}
        \inf_{A,b} \int_{(B_r(x_0)\cap B_R) \times \RR^d} |y- (Ax+b)|^2\,\dd \pi.
    \end{align*}
    Then there exist $A_0\in\RR^{d\times d}$ and $b_0\in\RR^d$ such that $A_r \to A_0$ and $b_r \to b_0$ uniformly in $x_0$ and the estimates \eqref{eq:A-b-conv-rates} hold.
\end{lemma}

\begin{proof}[Proof of Lemma \ref{lem:campanato}] We only give the proof for $A_r$, as the one for $b_r$ is analogous. Without loss of generality we may assume that $r$ is small enough such that $B_r(x_0)\subset B_R$. For further reference in the proof, let us mention here the following estimates obtained by equivalence of the $L^{\infty}$ and $L^2$ norms in the set of polynomials of degree one: Let $P(x) = Ax + b$ for some $A\in \RR^{d\times d}$ and $b\in \RR^d$, then for any $x_0\in\RR^d$ and $r>0$ there holds
\begin{align}\label{eq:equival-norms-poly}
    |A|^2 &\lesssim \frac{1}{r^{d+2}} \int_{B_r(x_0)} |P(x)|^2\,\dd x, &\text{and}&&
    |b|^2 &\lesssim \frac{1}{r^{d}} \int_{B_r(x_0)} |P(x)|^2\,\dd x.
\end{align}
    \begin{enumerate}[label=\textsc{Step \arabic*},leftmargin=0pt,labelsep=*,itemindent=*]
        \item \label{item:campanato-step1} Define
        \begin{align}\label{eq:def-Pr}
            P_r(x) := A_r x + b_r.
        \end{align}
        We claim that for any $k\in\NN_0$ there holds
        \begin{align}\label{eq:Pr-dyadic}
            \int_{B_{r 2^{-k-1}}(x_0)} |P_{r 2^{-k}}(x) - P_{r 2^{-k-1}}(x)|^2\rho_0(x)\,\dd x \lesssim \normpi_{\alpha}^2 \, 2^{-k(d+2+2\alpha)} \, r^{d+2+2\alpha}.
        \end{align}
        
        Indeed, since $\pi\in\Pi(\rho_0, \rho_1)$ and $B_{r 2^{-k-1}}(x_0) \subset B_{r 2^{-k}}(x_0)$, we may estimate
        \begin{align*}
             &\int_{B_{r 2^{-k-1}}(x_0)} |P_{r 2^{-k}}(x) - P_{r 2^{-k-1}}(x)|^2\rho_0(x)\,\dd x 
             = \int_{B_{r 2^{-k-1}}(x_0) \times \RR^d} |P_{r 2^{-k}}(x) - P_{r 2^{-k-1}}(x)|^2 \,\dd \pi\\
             &\quad\lesssim \int_{B_{r 2^{-k}}(x_0)\times \RR^d} |y-P_{r 2^{-k}}(x)|^2\,\dd \pi + \int_{B_{r 2^{-k-1}}(x_0)\times \RR^d} |y-P_{r 2^{-k-1}}(x)|^2\,\dd \pi ,
        \end{align*}
        so that by the definition of $A_r$ and $b_r$ as minimizers and the definition \eqref{eq:campanato-pi-full} of $\normpi_{\alpha}$, the bound \eqref{eq:Pr-dyadic} easily follows.
        
        \medskip
        \item We claim that for any $i\in\NN$,
        \begin{align}\label{eq:Ar-dyadic}
            |A_r - A_{r 2^{-i}}| \lesssim \normpi_{\alpha} \sum_{k=0}^{i-1} 2^{-k\alpha} \, r^{\alpha}.
        \end{align}
        
        Indeed, writing the difference as a telescopic sum, we may apply \eqref{eq:equival-norms-poly} to the polynomial $P_{r 2^{-k}} - P_{r 2^{-k-1}}$ to obtain
        \begin{align*}
             |A_r - A_{r 2^{-i}}| 
             &\leq \sum_{k=0}^{i-1}  |A_{r 2^{-k}} - A_{r 2^{-k-1}}| 
             \lesssim \sum_{k=0}^{i-1} \left(2^{(k+1)(d+2)} \frac{1}{r^{d+2}} \int_{B_{r 2^{-k-1}}(x_0)} |P_{r 2^{-k}} - P_{r 2^{-k-1}}|^2\,\dd x\right)^{\frac{1}{2}}, 
        \end{align*}
        so using that $\frac{1}{2}\leq \rho_0 \leq 2$ on $B_{r 2^{-k-1}}(x_0)$, the claim follows with \eqref{eq:Pr-dyadic}.
        
        \medskip 
        \item We show next that the sequence $\{A_{r2^{-i}}\}_{i\in\NN}$ converges as $i\to\infty$ to a limit $A_0$ independent of $r$.
        
        Indeed, for any $i>j$ we may use \eqref{eq:Ar-dyadic} to estimate
        \begin{align*}
            |A_{r 2^{-j}} - A_{r 2^{-i}}| \lesssim \normpi_{\alpha} \sum_{k=j}^{i-1} 2^{-k \alpha}  \, r^{\alpha} \stackrel{i,j \to \infty}{\longrightarrow} 0,
        \end{align*}
        since the series $\sum_{k=0}^{\infty} 2^{-k\alpha}$ converges. Hence, the sequence $\{A_{r2^{-i}}\}_{i\in\NN}$ is Cauchy and there exists $A_0\in\RR^{d\times d}$ such that $A_{r2^{-i}} \to A_0$ as $i\to \infty$. 
        
        To see the independence of the limit of $r$, let $0<r<\rho<\frac{R}{2}$ be small enough. Then applying \eqref{eq:equival-norms-poly} to the function $P_{r 2^{-j}} - P_{\rho 2^{-j}}$ for $j\in\NN$ gives
        \begin{align*}
            |A_{r 2^{-j}} - A_{\rho 2^{-j}}|^2 
            \lesssim 2^{j(d+2)} r^{-d-2} \int_{B_{r 2^{-j}}(x_0)} |P_{r 2^{-j}}(x) - P_{\rho 2^{-j}}(x)|^2\,\dd x,
        \end{align*}
        which can be bounded similarly to \ref{item:campanato-step1} to yield
        \begin{align*}
             |A_{r 2^{-j}} - A_{\rho 2^{-j}}|^2 \lesssim 2^{-2j\alpha} r^{2\alpha} \normpi_{\alpha}^2 \left( 1+ \left(\frac{\rho}{r}\right)^{d+2+2\alpha}\right) \stackrel{j\to\infty}{\longrightarrow} 0.
        \end{align*}
        The claimed inequality \eqref{eq:A-b-conv-rates} now follows easily by letting $i\to \infty$ in \eqref{eq:Ar-dyadic}.\hfill \qedhere
    \end{enumerate}
\end{proof}

\section{An \texorpdfstring{$L^p$}{Lp} bound on the displacement for almost-minimizing transport maps} \label{app:Lp}

In this section we give a proof of the interior $L^p$ estimate for arbitrary $p<\infty$ on the displacement of almost-minimizing transport maps: 

\begin{proposition} \label{prop:Lp-bound-almost-min} 
    Assume $\mu$ has a $C^{0,\alpha}$ density $f$ satisfying $f(0) = 1$. Let $T$ be an almost-minimizing transport map from $\mu$ to $\nu$ in the sense of \eqref{eq:almost-min-monge} with a rate function $\Delta_r \leq 1$. Then there exists $R_1 > 0$ such that, if
    \begin{equation}\label{eq:Lp-bound-small-assumption} 
        \frac{1}{R_1^{d+2}} \int_{B_{R_1}} |T(x)-x|^2 \, \mu(\mathrm{d}x) + R_1^{\alpha} [f]_{\alpha, R_1} \ll 1,
    \end{equation} then for any $p < \infty$, 
    \begin{equation}\label{eq:Lp-bound}
        \left( \int_{B_{\frac{R_1}{4}}} |T(x)-x|^p \, \mu(\mathrm{d}x) \right)^{\frac{1}{p}} \lesssim_p R_1^{\frac{d}{p}}\left( \int_{B_{R_1}} |T(x)-x|^2 \, \mu(\mathrm{d}x) \right)^{\frac{1}{d+2}}. 
    \end{equation}
\end{proposition} 

\begin{proof}
\begin{enumerate}[label=\textsc{Step \arabic*},leftmargin=0pt,labelsep=*,itemindent=*]
\item (A smooth transport map from $\mu$ to the Lebesgue measure). We fix an open ball $B_0$ centered at $0$ such that $B_0 \subset \supp f$. Let $S$ be the Dacorogna-Moser transport map from $\mu \lfloor_{B_0}$ to $\frac{\mu(B_0)}{|B_0|} \dd x \lfloor_{B_0}$, which is $\mathcal{C}^{1, \alpha}$ on $B_0$, see for example \cite[Chapter 1]{Vil09}. By an affine change of variables in the target space, we may define for any $x_0 \in X$ a map $S_{x_0}$ with the same regularity as $S$ that pushes forward $\mu \lfloor_{B_0}$ to $f(x_0) \, \dd x \lfloor_{\widetilde{B}_0}$, where $\widetilde{B}_0$ is the modified target space under this affine transformation, such that 
\begin{equation*}
    S_{x_0} (x_0) = x_0 \quad \textrm{and} \quad D S_{x_0} (x_0) = \II. 
\end{equation*} In particular, for all $x \in B_0$, 
\begin{equation*}
    |S_{x_0}(x) - x| \leq \left[ DS_{x_0} \right]_{\alpha} |x-x_0|^{1+\alpha} \leq C |x-x_0|^{1+\alpha},  
\end{equation*} where $C$ depends on $S$ and therefore on $f$ through the global H\"older semi-norm $[f]_{\alpha}$. Moreover, since $S^{-1}$ has the same regularity as $S$, $S_{x_0}^{-1}$ is $\mathcal{C}^{1,\alpha}$ and for all $y \in B_R(x_0)$, 
\begin{equation}\label{eq:Lp-alpha-Sinverse}
    |S_{x_0}^{-1} (y) - y| \leq C |y-x_0|^{1+\alpha}. 
\end{equation} 
Letting $R_1 > 0$ small enough so that first, $B_{R_1} \subset B_0 \subset \supp f$, and second, $C R_1^{\alpha} \leq 1$, we may therefore assume, writing $S_{x_0}(x) - x_0 = S_{x_0}(x) - x + x - x_0$ and $S_{x_0}^{-1}(y) - x_0 = S_{x_0}^{-1}(y) - y + y - x_0$, that for all $x_0 \in B_{R_1}$,  
\begin{equation}\label{eq:Lp-bound-S-seminorm}
    |S_{x_0}(x) - x_0| \leq 2|x-x_0| \quad \textrm{and} \quad |S_{x_0}^{-1}(y) - x_0| \leq 2|y-x_0|, \quad \textrm{for all } x, \, y \in B_{R_1}.
\end{equation}
We note that $R_1$ depends on $f$ only through $[f]_{\alpha}$. In view of \eqref{eq:Lp-bound-small-assumption} and the condition $f(0) = 1$, we may assume that 
\begin{equation}\label{eq:Lp-bound-f-bounded}
    \frac{1}{2} \leq f \leq 2 \quad \textrm{on} \quad B_{R_1}.
\end{equation}

\item \label{step:Lp-bound-step1} (Use of almost-minimality at scale $R$). Let us denote the displacement by $u:=T-\id$. We claim that if $R_1$ is small enough, for any ball $B := B_{12R}(x_0) \subset B_{R_1}$ and any set $A \subset B_R(x_0)$, we have\footnote{We use the notation $\fint$ for the average of a function on a set, with respect to the Lebesgue measure.}

\begin{equation}\label{eq:Lp-bound-step1-result} 
    \int_A |u| \lesssim R^{d+1} + |A| \fint_B |u|. 
\end{equation} 

To see this, we start from the following pointwise identity, that holds for any map $\Phi$, 
\begin{align}\label{eq:Lp-bound-identity}
    &(T(x)-x) \cdot (\Phi(x)-x) + (T(\Phi(x)) - \Phi(x)) \cdot (x-\Phi(x)) \\
    &\quad = \frac{1}{2} \left( |T(x)-x|^2 + |T(\Phi(x)) - \Phi(x)|^2 \right) - \frac{1}{2} \left( |T(\Phi(x))-x|^2 + |T(x) - \Phi(x)|^2 \right) + |\Phi(x)-x|^2, \nonumber
\end{align} 
which is (a deformation of) a standard identity used to derive $(c-)$monotonicity of an optimal transport map. As usual in optimal transportation, such a monotonicity property follows from considering a competitor that swaps points in the target space. However, because of our assumption of almost-minimality, we cannot really work with a pointwise argument but have to consider small sets $A$ of positive measure. We will do this by applying \eqref{eq:Lp-bound-identity} with a map $\Phi$ that swaps some parts of the graph of $T$. Indeed, assuming $\Phi$ is a $\mu$-preserving involution such that $\Phi = \id$ outside of a set of diameter of order $R$, we may integrate \eqref{eq:Lp-bound-identity} with respect to $\mu$, so that the identities 
\begin{align*}
    \int |T(\Phi(x)) - \Phi(x)|^2 \, \mu(\mathrm{d}x) &= \int |T(x)-x|^2 \, \mu(\mathrm{d}x), \\
    \int |T(x) - \Phi(x)|^2 \, \mu(\mathrm{d}x) &= \int |\widetilde{T}(x)-x|^2 \, \mu(\mathrm{d}x) \\
    \textrm{and} \quad \int (T(\Phi(x)) - \Phi(x)) \cdot (x-\Phi(x)) \, \mu(\mathrm{d}x) &= \int (T(x)-x) \cdot (\Phi(x)-x) \, \mu(\mathrm{d}x), 
\end{align*} 
combined with almost-minimality of $T$ at a scale of order $R$, yield 
\begin{align}\label{eq:Lp-bound-integrated-identity}
    \int (T(x)-x) \cdot (\Phi(x)-x) \, \mu(\mathrm{d}x) &\lesssim R^{d+2} + \int |\Phi(x)-x|^2 \, \mu(\mathrm{d}x) \nonumber\\
    &\lesssim R^{d+2} + R^2 \mu(B) \overset{\eqref{eq:Lp-bound-f-bounded}}{\lesssim} R^{d+2}. 
\end{align} 
To implement this, we give ourselves a direction $e \in B_2 \setminus B_1$, swap the (images through $S$ of the)\footnote{In the rest of \ref{step:Lp-bound-step1}, we drop the index $x_0$ for the map $S_{x_0}$. Now, $S$ is the map such that $S(x_0) = x_0$ and $DS(x_0) = \II$.} sets 
\begin{equation*}
    A_+ := \left\{ x\in A \mid u(x) \cdot e \geq 0 \right\} \quad \textrm{and} \quad A_- := A \setminus A_+, 
\end{equation*} 
with their translates along $e$ or $-e$, and average over all directions. 
For this we first define a map $\phi$ via 
\begin{align}\label{eq:Lp-bound-def-phi}
	\phi(x) := \left\{
	\begin{array}{ll}
		x+ 2Re & \textrm{ for } x \in S(A_+), \\
		x-2Re & \textrm{ for } x \in S(A_-), \\
		x-2Re & \textrm{ for } x \in S(A_+) + 2Re, \\ 
		x+2Re & \textrm{ for } x \in S(A_-) - 2Re, \\
		x & \textrm{ otherwise}. 
	\end{array}
	\right.
\end{align}
Notice that $\phi$ is well-defined because, in view of \eqref{eq:Lp-bound-S-seminorm}, we have $S(A) \subset B_{2R}(x_0)$. Furthermore, we see that $\phi = \id$ outside of $B_{6R}(x_0)$, and that $\phi$ swaps $S(A_+)$ and its translate $S(A_+) + 2Re$ and swaps $S(A_-)$ and its translate $S(A_-) - 2Re$, so that $\phi$ is an involution; in particular, $\phi$ preserves the Lebesgue measure. Combining the second point with the inclusion $B_{6R}(x_0) \subset S_{x_0}(B)$ from \eqref{eq:Lp-bound-S-seminorm}, we get that $\Phi := S^{-1} \circ \phi \circ S$ is identity outside of $B$. From the third point, we see that $\Phi$ is also an involution and $\Phi_{\#} \mu = \mu$, and it follows that the map $\widetilde{T} := T \circ \Phi$ satisfies $\widetilde{T}_{\#} \mu = \nu$ and $\widetilde{T} = T$ outside of $B$. Thus, \eqref{eq:Lp-bound-integrated-identity} holds. 

Now, recalling that $u = T-\id$, 
\begin{equation*}
    \int (T(x)-x) \cdot (\Phi(x)-x) \, \mu(\mathrm{d}x) = \int u(S^{-1}(y)) \cdot (S^{-1}(\phi(y))-S^{-1}(y)) f(x_0) \, \dd y, 
\end{equation*} 
so using the fact that $\phi = \id$ everywhere except on $S(A)$ and $(S(A_+) + 2Re) \cup (S(A_-) - 2Re)$, we have two terms to estimate. From the identity 
\begin{equation*}
    \int_{S(A)} u(S^{-1}(y)) \cdot (\phi(y)-y) f(x_0) \, \dd y = 2R \int_A |u \cdot e| \, \dd \mu, 
\end{equation*} and \eqref{eq:Lp-alpha-Sinverse}, we first have 
\begin{align}\label{eq:Lp-bound-left1}
    \left| \int_{S(A)} \right. & \left. u(S^{-1}(y)) \cdot (S^{-1}(\phi(y))-S^{-1}(y)) f(x_0) \, \dd y - 2R \int_A |u \cdot e| \, \dd \mu \right| \nonumber\\
    &\leq C \int_{S(A)} |u(S^{-1}(y)| \left( |\phi(y)-x_0|^{1+\alpha} + |y-x_0|^{1+\alpha} \right) \, f(x_0) \, \dd y \nonumber\\
    &\leq C R^{1+\alpha} \int_A |u| \, \dd \mu. 
\end{align}
We then estimate, using \eqref{eq:Lp-alpha-Sinverse} and \eqref{eq:Lp-bound-f-bounded},
\begin{align}\label{eq:Lp-bound-left2}
    &\left|\int_{(S(A_+) + 2Re) \cup (S(A_-) - 2Re)} u(S^{-1}(y)) \cdot (S^{-1}(\phi(y))-S^{-1}(y)) f(x_0) \, \dd y \right| \nonumber\\
    &\qquad \qquad \qquad \lesssim R \int_{(S(A_+) + 2Re) \cup (S(A_-) - 2Re)} |u(S^{-1}(y))| \, \dd y, 
\end{align} so that combining \eqref{eq:Lp-bound-integrated-identity}, \eqref{eq:Lp-bound-left1} and \eqref{eq:Lp-bound-left2} gives 
\begin{align}\label{eq:Lp-bound-e-identity}
    \int_A |u \cdot e| \, \dd \mu - C R^{\alpha} \int_A |u| \, \dd \mu \lesssim R^{d+1} + \int_{(S(A_+) + 2Re) \cup (S(A_-) - 2Re)} |u(S^{-1}(y))| \, \dd y. 
\end{align}

It remains to integrate \eqref{eq:Lp-bound-e-identity} with respect to $e \in B_2 \setminus B_1$, using that 
\begin{align}\label{eq:Lp-bound-int-e}
    \int_{B_2 \setminus B_1} |u(x) \cdot e| \, \dd e \gtrsim |u(x)| 
\end{align} and the equivalence 
\begin{equation*}
    x \in (S(A_+) + 2Re) \cup (S(A_-) - 2Re) \quad  \Longleftrightarrow \quad e \in \frac{1}{2R}(x-S(A_+)) \cup \frac{1}{2R}(x+S(A_-)), 
\end{equation*} which gives 
\begin{align}\label{eq:Lp-bound-int-e2}
    \int_{B_2 \setminus B_1} 1_{(S(A_+) + 2Re) \cup (S(A_-) - 2Re)}(x) \, \dd e \leq \left|\frac{1}{2R}(x-S(A_+)) \cup \frac{1}{2R}(x+S(A_-)) \right| \lesssim \frac{|S(A)|}{|B|}. 
\end{align} From \eqref{eq:Lp-bound-int-e}, \eqref{eq:Lp-bound-int-e2} and the inclusion
\begin{equation*}
    (S(A_+) + 2Re) \cup (S(A_-) - 2Re) \subset B_{6R}(x_0), 
\end{equation*} we obtain 
\begin{align*}
    \int_{B_2 \setminus B_1} \int_{(S(A_+) + 2Re) \cup (S(A_-) - 2Re)} |u(S^{-1}(y))| \, \dd y \, \dd e &\overset{\eqref{eq:Lp-bound-f-bounded}}{\lesssim} \frac{|S(A)|}{|B|} \int_{B_{6R}(x_0)} |u(S^{-1}(y))| \, f(x_0) \, \dd y \\
    &= \frac{|S(A)|}{|B|} \int_{S^{-1}(B_{6R}(x_0))} |u| \, \dd \mu \\
    &\overset{\eqref{eq:Lp-bound-S-seminorm}}{\leq} \frac{|S(A)|}{|B|} \int_{B} |u| \, \dd \mu. 
\end{align*} Finally, the transport condition $|\det DS| = \frac{f}{f(x_0)}$ ensures that we have $|S(A)| \lesssim |A|$, so that \eqref{eq:Lp-bound-e-identity} yields
\begin{equation*}
    (1-C R^{\alpha}) \int_A |u| \, \dd \mu \lesssim R^{d+1} + \frac{|A|}{|B|} \int_B |u| \, \dd \mu. 
\end{equation*} Choosing $R_1$ so that $C R_1^{\alpha} \leq \frac{1}{2}$ and using \eqref{eq:Lp-bound-f-bounded}, we obtain \eqref{eq:Lp-bound-step1-result}. We note again that the choice of $R_1$ depends on $f$ only through $[f]_{\alpha}$.

\medskip

\item \label{step:Lp-bound-step2} We claim that for any set $A \subset B_{\frac{R_1}{2}}$ with $\textrm{diam} \, A \leq R \ll R_1$, for any $\beta > 0$, we have 
\begin{equation}\label{eq:Lp-bound-step2-result}
    \int_A |u| \lesssim_{\beta} R^{d+1} + \frac{|A|}{R^{\beta}} \left( \int_{B_{R_1}} |u|^2 \right)^{\frac{1+\beta}{d+2}}. 
\end{equation}

\begin{enumerate}[label=\textsc{Step 3.\Alph*.},leftmargin=0pt, labelsep=*,itemindent=*]

\item Let us show that for all $\beta > 0$ and for any two balls $B_r(x_0) \subset B_R(x_0) \subset B_{R_1}$, 
\begin{equation}\label{eq:Lp-substep1}
    \fint_{B_r(x_0)} |u| \lesssim_{\beta} \left( \frac{R}{r} \right)^{\beta} \left( R + \fint_{B_R(x_0)} |u| \right). 
\end{equation}

To this end, we momentarily introduce 
\begin{equation*}
    \Lambda := \sup_{r \leq \rho \leq R} \rho^{\beta} \fint_{B_{\rho}(x_0)} |u|. 
\end{equation*}
Applying \ref{step:Lp-bound-step1} to the sets $A := B_{\rho}(x_0)$ and $B := B_{M \rho}(x_0)$, we obtain, provided $12\rho \leq M\rho \leq R$, 
\begin{equation*}
    \fint_{B_{\rho}(x_0)} |u| \lesssim M^{d+1} \rho + \fint_{B_{M\rho}} |u|, 
\end{equation*} 
so that, provided $r\leq \rho \leq \frac{R}{M}$, 
\begin{equation*}
    \rho^{\beta} \fint_{B_{\rho}(x_0)} |u| \lesssim M^{d-\beta} R^{1+\beta} + \frac{\Lambda}{M^{\beta}}. 
\end{equation*}
Fixing a large enough $M \sim 1$, this yields in the range $r\leq \rho \leq \frac{R}{M}$, 
\begin{equation}\label{eq:Lp-ineq-Lambda}
    \rho^{\beta} \fint_{B_{\rho}(x_0)} |u| - \frac{\Lambda}{2}  \lesssim R^{1+\beta}. 
\end{equation} 
In the remaining range $\frac{R}{M} \leq \rho \leq R$, we have 
\begin{equation*}
    \rho^{\beta} \fint_{B_{\rho}(x_0)} |u| \lesssim_{\beta} R^{\beta} \fint_{B_{R}(x_0)} |u|, 
\end{equation*} 
so that \eqref{eq:Lp-ineq-Lambda} turns into 
\begin{equation*}
    \frac{\Lambda}{2}  \lesssim R^{1+\beta} + R^{\beta} \fint_{B_{R}(x_0)} |u|, 
\end{equation*} 
which is \eqref{eq:Lp-substep1}. 

\medskip
\item We prove \eqref{eq:Lp-bound-step2-result}. 

By applying Cauchy-Schwarz to \eqref{eq:Lp-substep1}, we obtain 
\begin{equation*}
    \fint_{B_r(x_0)} |u| \lesssim \left( \frac{R}{r} \right)^{\beta} \left( R + \frac{1}{R^{\frac{d}{2}}} \int_{B_{R_1}} |u|^2 \right), 
\end{equation*} 
so that, optimizing in $R$ by choosing $R = \left( \int_{B_{R_1}} |u|^2 \right)^{\frac{1}{d+2}} \ll R_1$ by \eqref{eq:Lp-bound-small-assumption}, we get, provided $B_r(x_0) \subset B_{R_1}$,  
\begin{equation*}
    r^{\beta} \fint_{B_r(x_0)} |u| \lesssim \left( \int_{B_{R_1}} |u|^2 \right)^{\frac{1+\beta}{d+2}}. 
\end{equation*} 
In combination with \ref{step:Lp-bound-step1}, this yields \eqref{eq:Lp-bound-step2-result}.

\end{enumerate}

\item (Conclusion). We now have all the ingredients to prove the estimate \eqref{eq:Lp-bound}. 

Given a threshold $M < \infty$ and a ball $B \subset B_{\frac{R_1}{2}}$ of radius $R \ll R_1$, we apply \ref{step:Lp-bound-step2} to $A := \{|u|>M\} \cap B$, to the effect that 
\begin{equation}\label{eq:Lp-step4-M}
    M R^{\beta} \gg \left( \int_{B_{R_1}} |u|^2 \right)^{\frac{1+\beta}{d+2}} \quad \Longrightarrow \quad M|A| \lesssim R^{d+1}. 
\end{equation}
Provided $M \geq \left( \int_{B_{R_1}} |u|^2 \right)^{\frac{1}{d+2}}$, thanks to \eqref{eq:Lp-bound-small-assumption}, there are radii $R \ll R_1$ for which the statement on the left-hand side of \eqref{eq:Lp-step4-M} holds. Hence, by covering $B_{\frac{R_1}{4}}$ by these balls, we obtain 
\begin{equation}\label{eq:Lp-weak-Lp}
    M \left| \{|u| \geq M\}\cap B_{\frac{R_1}{4}} \right| \lesssim R_1^d \left( \frac{1}{M} \left( \int_{B_{R_1}} |u|^2 \right)^{\frac{1+\beta}{d+2}} \right)^{\frac{1}{\beta}}. 
\end{equation} 
Note that \eqref{eq:Lp-weak-Lp} is trivially satisfied for $M \leq \left( \int_{B_{R_1}} |u|^2 \right)^{\frac{1}{d+2}}$, so that with $p := 1+\frac{1}{\beta}$, 
\begin{equation*}
    M \left| \{|u| \geq M\}\cap B_{\frac{R_1}{4}} \right|^{\frac{1}{p}} \lesssim R_1^{\frac{d}{p}} \left( \int_{B_{R_1}} |u|^2 \right)^{\frac{1}{d+2}}
\end{equation*} 
holds for all $M$. This amounts to an estimate in the weak $L^p$ norm of $u$ on $B_{\frac{R_1}{4}}$. Because of \eqref{eq:Lp-bound-small-assumption}, we trivially have 
\begin{equation*}
    \left( \int_{B_{\frac{R_1}{4}}} |u|^2 \right)^{\frac{1}{2}} \lesssim R_1^{\frac{d}{2}} \left( \int_{B_{R_1}} |u|^2 \right)^{\frac{1}{d+2}}, 
\end{equation*} 
so by interpolation, we obtain \eqref{eq:Lp-bound}.  \qedhere
\end{enumerate}
\end{proof}

\vfill
\end{document}